\documentclass[10pt]{article}

\usepackage{amsfonts}
\usepackage{amsmath}
\usepackage{amssymb}
\usepackage{amsthm}
\usepackage[alphabetic]{amsrefs}
\usepackage{calc} 
\usepackage{graphicx}
\usepackage[hidelinks]{hyperref}
\usepackage[left=1in,top=1in,right=1in,foot=1in]{geometry}
\usepackage{todonotes}
\usepackage{microtype}
\usepackage{tikz-cd}
\usepackage{forest}
\usepackage{dynkin-diagrams}
\usepackage{arydshln}

\newcommand{\diag}{\mathrm{diag}}
\newcommand{\pair}[2]{\left\langle #1 , #2\right\rangle}

\DeclareMathOperator{\spn}{span}

\def\into{\mathrel{\hookrightarrow}}
\def\onto{\mathrel{\twoheadrightarrow}}

\newcommand{\sets}[2]{\left\{#1\,\middle|\,#2\right\}}
\newcommand{\genrel}[2]{\left\langle #1\,\middle|\,#2\right\rangle}

\newcommand{\sett}[1]{\left\{#1\right\}}

\newcommand{\stab}[2]{\mathrm{Stab}_{#1}(#2)}

\DeclareMathOperator{\rk}{rk}

\newcommand{\Hom}[2]{\mathrm{Hom}(#1, #2)}

\newcommand{\id}{\mathrm{id}}

\DeclareMathOperator{\Ind}{Ind}

\newcommand{\Mod}{\mathbf{Mod}}

\DeclareMathOperator{\Spec}{Spec}

\newcommand{\Bb}{\mathcal{B}}

\newcommand{\X}{\mathcal{X}}
\newcommand{\HH}{\mathbf{H}}

\newcommand{\Ee}{\mathcal{E}}

\newcommand{\norm}[1]{\left\| #1 \right\|}
\newcommand{\dInt}[2]{\int {#1} \,\mathrm{d} {#2}}
\newcommand{\dIntOver}[3]{\int_{#1} {#2} \,\mathrm{d} {#3}}

\DeclareMathOperator{\vol}{vol}

\newcommand{\trace}[2]{\mathrm{trace}\left({#1}\,,{#2}\right)}

\newcommand{\R}{\mathbb{R}}
\newcommand{\N}{\mathbb{N}}
\newcommand{\Z}{\mathbb{Z}}
\newcommand{\C}{\mathbb{C}}
\newcommand{\Q}{\mathbb{Q}}

\newcommand{\F}{\mathbb{F}}
\newcommand{\Oo}{\mathcal{O}}

\newcommand{\g}{\mathfrak{g}}

\newcommand{\bB}{\mathfrak{b}}

\newcommand{\aA}{\mathfrak{a}}

\newcommand{\G}{\mathbf{G}}

\newcommand{\Waff}{W_{\mathrm{aff}}}

\newcommand{\cc}{\mathbf{c}}

\newcommand{\GL}{\mathrm{GL}}

\newcommand{\SL}{\mathrm{SL}}

\newcommand{\SO}{\mathrm{SO}}
\newcommand{\Sp}{\mathrm{Sp}}

\newcommand{\GSp}{\mathrm{GSp}}
\newcommand{\PGSp}{\mathrm{PGSp}}

\newcommand{\Gm}{{\mathbb{G}_\mathrm{m}}}

\newcommand{\Sn}{\mathfrak{S}}

\newcommand{\St}{\mathrm{St}}
\newcommand{\triv}{\mathrm{triv}}
\newcommand{\sgn}{\mathrm{sgn}}

\newcommand{\bq}{\mathbf{q}}

\newtheorem{theorem}{Theorem}
\newtheorem*{theorem*}{Theorem}
\newtheorem{cor}{Corollary}
\newtheorem{prop}{Proposition}
\newtheorem{lem}{Lemma}

\newtheorem{conj}{Conjecture}

\theoremstyle{definition}
\newtheorem{dfn}{Definition}
\newtheorem*{dfn*}{Definition}
\theoremstyle{remark}
\newtheorem{ex}{Example}
\newtheorem*{ex*}{Example}
\newtheorem{rem}{Remark}

\DeclareMathAlphabet\mathbfcal{OMS}{cmsy}{b}{n}

\providecommand{\keywords}[1]{\textbf{\textit{Keywords---}} #1}

\newcommand{\Gd}{G^\vee}

\newcommand{\aAG}{\mathfrak{a}_G}

\DeclareMathOperator{\rank}{rank}

\newcommand{\Cc}{\mathcal{C}}

\newcommand{\T}{\mathbb{T}}
\DeclareMathOperator{\val}{\mathrm{val}}

\newcommand{\bk}{\bar{\mathcal{K}}}
\newcommand{\nbk}{\mathcal{K}}

\newcommand{\PrInv}{\bar{\kappa}}

\newcommand{\Ppp}{\mathcal{P}}

\newcommand{\HHh}{\mathbfcal{H}^-}
\newcommand{\JJ}{\mathbfcal{J}}

\newcommand{\NN}{\mathbf{N}}
\newcommand{\Aa}{\mathbf{A}}
\newcommand{\B}{\mathbf{B}}
\newcommand{\PP}{\mathbf{P}}

\newcommand{\Vect}{\mathbf{Vect}_\C}

\title{Denominators in Lusztig's asymptotic Hecke algebra via the Plancherel formula}
\date{\today}
\author{Stefan Dawydiak \thanks{Mathematical Institute, Universit\"{a}t Bonn, Bonn 53111 Germany; email \texttt{dawydiak@math.uni-bonn.de}}}

\begin{document}
\maketitle

\begin{abstract}
Let $\Waff$ be an extended affine Weyl group, $\HH$ be the corresponding affine Hecke algebra over the ring 
$\C[\bq^\frac{1}{2}, \bq^{-\frac{1}{2}}]$, and $J$ be 
Lusztig's asymptotic Hecke algebra, viewed as a based ring with basis $\{t_w\}$. Viewing
$J$ as a subalgebra of the $(\bq^{-\frac{1}{2}})$-adic completion of $\HH$ via Lusztig's map $\phi$,
we use Harish-Chandra's Plancherel formula for $p$-adic groups to 
show that the coefficient of $T_x$ in $t_w$ is a rational function of $\bq$, with denominator
depending only on the two-sided cell containing $w$, and
dividing a power of the 
Poincar\'{e} polynomial of the finite Weyl group. As an application, we conjecture that these denominators encode more detailed information about the failure of the Kazhdan-Lusztig classification at roots of the 
Poincar\'{e} polynomial than is currently known. 

Along the way, we show that upon specializing $\bq=q>1$, the map from
$J$ to the Harish-Chandra Schwartz algebra is injective.
As an application of injectivity,
we give a novel criterion for an Iwahori-spherical representation to have fixed vectors under a larger
parahoric subgroup in terms of its Kazhdan-Lusztig parameter.
\end{abstract}
%

\tableofcontents
\section{Introduction}
\label{section Intro}
Let $\Waff$ be an affine Weyl group or extended affine Weyl group, and let $\HH$ be its associated
Hecke algebra over $\mathcal{A}:=\C[\bq^{1/2},\bq^{-1/2}]$, where $\bq$ is a formal variable. The representation theory
of $\HH$ is very well understood, behaving well and uniformly when $\bq$ is specialized to any $q\in\C^\times$
that is not a root of the Poincar\'{e} polynomial $P_W$ of the finite Weyl group $W\subset\Waff$.

When $\bq$ is specialized to a prime power $q$, the category of finite-dimensional modules over the specialized algebra $H$ is 
equivalent
to the category of admissible representations with nonzero Iwahori-fixed vector of some $p$-adic group.
A form of local Langlands correspondence, the Deligne-Langlands conjecture, has been established
by Kazhdan and Lusztig in \cite{KLDeligneLanglands}, where they classified modules over the generic algebra $\HH$ using algebraic $K$-theory. A slightly different approach to this classification due to Ginzburg 
is explained in \cite{CG}. In both treatments, a first step is to fix a central character. In particular, one 
must chose a complex number $q\in\C^\times$ by which $\bq$ will act. Decomposing the $K$-theory of certain 
subvarieties of Springer fibres into irreducible representations of a certain finite group yields
the \emph{standard modules}. It can happen that the standard modules are themselves simple (for 
example, simple tempered representations, which play an essential role in the present paper, are of this form), but in 
general simple modules are obtained as a certain unique nonzero quotient of standard modules. This quotient 
exists when $q\in\C^\times$ is not a root of unity, but can be zero otherwise. Lusztig conjectured
in \cite{LRepAffHecke} that this classification would in fact hold whenever $q$ was not a root 
of $P_W$, and this result was proven by Xi in \cite{XiJAMS}. 
Xi also showed that the classification fails in 
general at roots of the Poincar\'{e} polynomial, and presented this failure by giving an example related to a lack of simple 
$\HH|_{\bq=q}$-modules attached to the lowest two-sided cell.
Our results in this paper explain that the lowest two-sided cell is, in a precise sense, maximally singular 
with respect to the parameter $\bq$.

One way Lusztig expressed the uniformity in $q$ of the representation theory of the various algebras $
\HH|_{\bq=q}$ is via the asymptotic Hecke algebra $J$. This is a $\C$-algebra (in fact, a $\Z$-algebra)
$J$ with distinguished basis $\{t_w\}_{w\in\Waff}$, equipped with an injection
$\phi\colon\HH\into J\otimes_{\C}\mathcal{A}$. In this way there is a map from $J$-modules to 
$\HH$-modules, and Lusztig has shown in \cite{affineIII} and \cite{affineIV} that when $q$ 
is not a root of unity (other than $1$), that the specialized map $\phi_q$ induces a bijection between simple
$\HH|_{\bq=q}$-modules and simple $J$-modules, these last being defined over $\C$. Moreover, he showed that when 
$P_W(q)\neq 0$, the map $\phi_q$ induces an isomorphism 
\[
(\phi_q)_*\colon K_0(J-\Mod)\to K_0(\HH|_{\bq=q}-\Mod)
\]
of Grothendieck groups. 
The map $\phi$ becomes a bijection after completing $\HH$ and $J\otimes_\C\mathcal{A}$ by replacing 
$\mathcal{A}$ with 
$\C((\bq^{-1/2}))$ and allowing infinite sums convergent in the $(\bq^{-1/2})$-adic topology. 
In this way one can write a basis element $t_w$ as an infinite sum
\begin{equation}
\label{eq intro t_w as T_x lin combo}
{}^\dagger(-)\circ\phi^{-1}(t_w)=\sum_{x\in\Waff}a_{x,w}T_x,
\end{equation}
where each $a_{x,w}$ is a formal Laurent series in $\bq^{-1/2}$, and ${}^\dagger(-)$ is the involution of $\HH$ defined in 
Definition \ref{dfn Goldman involution}. It agrees with the Goldman involution of $\HH$ when $G$ is simply-connected.
In this paper we will almost exclusively work with $\phi\circ{}^\dagger(-)$, for reasons explained in Section \ref{subsection J and p-adic groups}.

In light of the above, it is natural
to ask how $a_{x,w}$ behaves when $\bq$ is specialized to a root of unity. 
\subsection{The asymptotic Hecke algebra and $p$-adic groups}
\label{subsection J and p-adic groups}
This paper is prompted by the work of Braverman and Kazhdan in \cite{BK}, who 
related the asymptotic Hecke algebra to harmonic analysis on $p$-adic groups. Specifically,
in \cite{affineIV}, Lusztig relates simple $J$-modules to certain $\HH\otimes_\mathcal{A}\C(\bq^{-1/2})$-modules
termed \emph{tempered} because their definition is made in analogy with Casselman's criterion
for temperedness of $p$-adic groups. In \cite{BK}, Braverman and Kazhdan showed essentially 
that the analytic meaning of the word ``tempered" can be substituted into Lusztig's results 
from \cite{affineIV}. 

More precisely, let $\G$ be a connected reductive group defined and split over a non-archimedean 
local field $F$ whose extended affine Weyl group is $\Waff$. Then in \cite{BK}, Braverman and Kazhdan 
define a map 
expressing $J$ as sitting between the Iwahori-Hecke algebra of $G=\G(F)$ and the Harish-Chandra Schwartz algebra $\mathcal{C}$,
and propose a spectral characterization of $J$ via the operator Payley-Wiener theorem, obtaining 
a diagram
\begin{equation}
\label{eqn BK summary diagram}
\begin{tikzcd}
H(G,I)\arrow[dd, "\sim"]\arrow[dr, hook, "\phi_q\circ{}^{\dagger}(-)"]\arrow[rr, hook]&&\mathcal{C}^{I\times I}\arrow[dd, "\sim"]\\
&J\arrow[ur, "\tilde{\phi}"]\arrow[d, "\eta"]&\\
\mathcal{E}^I\arrow[r, hook]&\mathcal{E}^I_J\arrow[r, hook]&\mathcal{E}_t^I,
\end{tikzcd}
\end{equation}
where the outer vertical maps are Fourier transform $f\mapsto\pi(f)$ and the rings $\Ee^I$
and $\Ee^I_t$ of endomorphisms of forgetful functors to vector spaces are as described by the operator 
Paley-Wiener theorem, as we recall in Section \ref{subsection the HC Schwartz algebra}.

The map $\eta$ is defined in \cite{BK} and we will recall its definition below. In particular,
it induces the map $\tilde{\phi}$, which then associates a Harish-Chandra Schwartz function
to every element of $J$ such that $\eta(j)=\pi(\tilde{\phi}(j))$, giving another way of associating to $t_w$ an expression similar to \eqref{eq intro t_w as T_x lin combo}.

This prompts several questions: whether $\eta$ (equivalently $\tilde{\phi}$)
is injective, whether it is surjective, and the nature of the relationship between the Schwartz 
function $\tilde{\phi}(t_w)$ and the expression \eqref{eq intro t_w as T_x lin combo}.

\subsubsection{Denominators in the affine Hecke algebra and injectivity of $\eta$}
In the first part of this paper we prove that $\eta$ is injective.
Along the way, we prove in Proposition \ref{prop phi_1=phi inverse} that $\tilde{\phi}$ is essentially the map $\phi^{-1}$. 
Our strategy is to determine that the Schwartz functions $f_w$ on the $p$-adic group $G$
satisfy the statements of Theorem \ref{thm general G denominators} below, and are in addition sufficiently 
well-behaved so as to lift to elements of a certain completion $\HHh$ of $\HH$, thus defining a map 
\[
\phi_1\colon J\otimes_\C\mathcal{A}\to\HHh.
\]
We therefore obtain two maps 
$J\subset J\otimes_\C\mathcal{A}\to\HHh$: the inverse ${}^\dagger(-)\circ\phi^{-1}$ of Lusztig's map, and our map $\phi_1$ induced by the construction in \cite{BK}. We prove in Proposition \ref{prop phi_1=phi inverse} that these maps agree, at which point Theorem \ref{thm conjecture is true GLn} and the first statement of Theorem \ref{thm general G denominators} follow.

In particular, $\phi_1$ is injective. Using that the representation theory of $J$ is sufficiently
independent of $q$, we then show in Corollary \ref{cor BK map injective} that $\tilde{\phi}$ is injective for any $q>1$. This is obviously equivalent to 
\begin{theorem}[Corollary \ref{cor BK map injective}]
The map $\eta$ is injective for any $q>1$.
\end{theorem}
A weaker form of the following Theorem, which is the main result of this paper, was conjectured by Kazhdan.

%
%
\begin{theorem}
\label{thm general G denominators}
Let $\Waff$ be an affine Weyl group, $\HH$ its affine Hecke algebra over $\mathcal{A}$, and 
$J$ its asymptotic Hecke algebra. Let $\phi\colon \HH\into J\otimes_\C\mathcal{A}$ be Lusztig's map.
\begin{enumerate}
\item
For all $x,w\in\Waff$, $a_{x,w}$ is a rational function of $\bq$. The denominator of 
$a_{x,w}$ is independent of $x$. As a function of $w$, it is constant on two-sided cells. 
\item 
There exists $N_{\Waff}\in\N$ such that upon writing 
\[
{}^\dagger(-)\circ\phi^{-1}(t_w)=\sum_{x\in\Waff}a_{x,w}T_x,
\]
we have
\[
P_W(\bq)^{N_{\Waff}}a_{x,w}\in\mathcal{A}
\]
for all $x,w\in\Waff$.
\item
If $d$ is a distinguished involution in the lowest two-sided cell, then $a_{1,d}=1/P_W(\bq)$ exactly.
\end{enumerate} 
\end{theorem}
In \cite{D}, the author proved Theorem \ref{thm general G denominators} in type $\tilde{A}_1$, but with 
different conventions. To translate to the conventions of this paper, the reader should replace $j$
with the involution ${}^\dagger(-)$, and the completion of $\HH$ with respect to the $C_w$ basis
and positive powers of $\bq^{1/2}$ with the completion of $\HH$ with respect to the basis 
$\{(-1)^{\ell(w)}C'_w\}_{w\in\Waff}$  and negative powers of $\bq^{1/2}$. Note also that we write $a_{x,w}$ 
instead of $a_{w,x}$ as in \cite{D}. In \cite{N}, Neunh\"{o}ffer described the coefficients $a_{x,w}$ for finite 
Weyl groups. 

In future, it would be desirable to also treat the case of unequal parameters, where a result like that of 
\cite{BDD} governing denominators is not yet available. On the other hand, when $G=\GL_n$, we are able to be 
slightly more precise than Theorem \ref{thm general G denominators}, while also not appealing to \cite{BDD}. 
For this reason we treat the case $G=\GL_n$ separately as Theorem \ref{thm conjecture is true GLn} in Section \ref{subsection proof of them conjection is true GLn for GLn}.

Our main tool is Harish-Chandra's Plancherel formula for the $p$-adic group $G$ associated to $\HH$
and the surjection of cocentres induced by Lusztig's map $\phi$ after inverting $P_W(\bq)$
proved in \cite{BDD}\footnote{In fact, as proved by Bezrukavnikov-Braverman-Kazhdan in the appendix of \textit{loc. cit.}, $\phi_q$ induces an isomorphism whenever $q$ is not a root of unity, but we will not need this.}. We invoke \cite{BDD} only at the very end of our argument, which, absent \cite{BDD}, still proves
that $a_{x,w}$ are rational functions with denominator depending only the two-sided cell containing $w$.
We do so with an eye to future work dealing with Hecke algebras with unequal parameters.

In \cite{D}, the author related a conjecture of Kazhdan concerning 
the positivity of some coefficients related to the coefficients $a_{x,w}$.
Historically, proofs of such positivity phenomena have also provided 
interpretations of the positive quantities in question. 
While we cannot currently prove the conjecture in \cite{D}, our results in Section 
\ref{section parahoric-fixed} hint at a possible interpretation of $a_{1,d}$ for certain distinguished 
involutions $d$.
\begin{rem}
\label{rem more precise version}
It is tempting to conjecture the following more precise version of Theorem \ref{thm general G denominators}, based on the factorization $P_{\mathbf{M}_{\mathbf{P}}}(\bq)P_{\G/\PP}(\bq)=P_{\G/\B}(\bq)$: for every Levi subgroup $\mathbf{M}_{\mathbf{P}}$ 
of $\G$ and all $\omega\in\mathcal{E}_2(M_P)$, the formal degree
$d(\omega)$ is a rational function of $q$ the denominator of which divides a power of 
$P_{\mathbf{M}_{\mathbf{P}}}(q)$.
The integral over all induced twists $\Ind_P^G(\nu\otimes\omega)$ is a rational function of $q$
with denominator dividing a power of the Poincar\'{e} polynomial of the partial flag variety $(\G/\PP)(\C)$. For example, if $G=\GL_6(F)$ and $M=\GL_3(F)\times\GL_3(F)$.
In this case the integral itself (omitting the factor $C_M$ in the notation of Section \ref{subsection Plancherel formula for GLn}) is
\[
\frac{1}{2\pi i}\frac{1}{2\pi i}\int_\T\int_\T\frac{(z_1-z_2)(z_1-z_2)}{(z_1-q^3z_2)(z_1-q^{-3}z_2)}\frac{\mathrm{d}z_1}{z_1}\frac{\mathrm{d}z_2}{z_2}=\frac{(1-q^3)^2}{(1-q^6)^2q^3}+1
=\frac{(1-q^3)q^{-3}}{1+q^3}+1,
\]
and by \cite{BottTu} Proposition 23.1, (with $\bq=t^2$) we have
\[
P_{\G/\PP}(\bq)=P_{G(3,6)}(\bq)=(1+\bq^2)(1+\bq+\bq^2+\bq^3+\bq^4)(1+\bq^3).
\]
In examples such as the above, this does indeed happen, but only after cancellation with some terms in 
the numerator. In general, we will not track numerators precisely enough to show this
version of Theorem \ref{thm general G denominators}. We shall however see a limited demonstration of this behaviour in Corollary \ref{cor denominators of fd(1) poincare}.
\end{rem}
\subsubsection{Denominators in the affine Hecke and the Kazhdan-Lusztig classification at roots of unity}
\label{subsubsubection Denominators and the KL classification at roots of 1}
The affine Hecke has a filtration by two-sided ideals
\[
\HH^{\geq i}=\spn{\sets{C_w}{a(w)\geq i}},
\]
where $a$ is Lusztig's $a$-function. 
As such, for any $q\in\C^\times$ and any simple $H=\HH|_{\bq=q}$-module $M$ there is an integer $a(M)$ such 
that $H^{\geq i}M\neq 0$ but $H^{\geq i+1}M= 0$.
Define $a(M)=i$ to be this integer. One can also define $a(E)=a(\cc(E))$ where $E$ is a simple $J$-module and 
$J_{\cc(E)}$ is the unique two-sided ideal not annihilating $E$.

The algebra $J$ linearizes the above filtration into an honest direct sum, and implements
the almost-independence on $q\in\C^\times$ of the representation theory of 
$H=\HH|_{\bq=q}$ as follows. 
\begin{theorem}[\cite{XiJAMS}]
\begin{enumerate}
\item 
Suppose that $q$ is not a root of the Poincar\'{e} polynomial of $\G$.
Then for each simple $J$-module $E$, the $H$ module 
${}^{\phi_q} E$ has a unique simple quotient $L$ such that $a(E)=a(L)$.
For all other simple subquotients $L'$ of $E$, we have $a(L')<a(E)$.

Equivalently, for all admissible triples $(u,s,\rho)$, the representation 
$K(u,s,\rho, q)$ of $H$ has a unique nonzero simple quotient 
$L=L(u,s,\rho, q)$ such that $a(L)=a(\cc(u))$. That is, the Deligne-Langlands conjecture is true for $\HH|_{\bq=q}$.
\item
If $q$ is a root of the Poincar\'{e} polynomial of $\G$, then
the Deligne-Langlands conjecture is false for the lowest cell. That is,  
if $u=\{1\}$, then every simple subquotient $L'$ of $K(u,s,\rho,q)$
has $a(L')<a(\cc_0)$.
\end{enumerate}
\end{theorem}

By Theorem \ref{thm general G denominators}, the coefficients $a_{1,d}$ have poles
at every root of $P_W$, for all distinguished involutions $d$ in the lowest two-sided cell $\cc_0$. On the other hand, as we show in Example \ref{ex less singular cell}, there do exist cells $\cc\neq\cc_0$ 
such that the coefficients $a_{x,w}$ are nonsingular at certain roots of $P_W$, for all $w\in\cc$ 
and $x\in\Waff$. We encode the hope that this is no accident as
\begin{conj}
\label{conj KL failure cell-by-cell}
Let $\tilde{W}$ be an affine Weyl group, and let $q\in\C^\times$ be a 
root of $P_W$. Let $\cc$ be a two-sided cell such that if $w\in\cc$, then $a_{x,w}$ does not have a pole
at $\bq=q$ for any $x\in\Waff$. Let $u=u(\cc)$. Let $K(u,s,\rho)$ be a standard module in the notation of \cite{KLDeligneLanglands}. Then the module ${}^\dagger{}^*K(u,s,\rho,q)$ (see Definition \ref{dfn Goldman involution} (a), (b) and the discussion following Theorem \ref{BK nsp theorem}) has a unique simple quotient $L=L(u,s,\rho, q)$ such that $a(L)=a(E)$, where $E$ is the simple $J$ module corresponding via $\phi$ to $(u,s,\rho)$ under \cite[Thm. 4.2]{affineIV}. Two such simple modules are isomorphic if and only if 
their corresponding triples are conjugate.
\end{conj}
Note that in type $\tilde{A}_n$, the number of two-sided cells grows as $e^{\sqrt{n}}$, whereas the number of 
subsets of roots of $P_W$ is $2^{n(n+1)/2}$. For example in type $\tilde{A}_1$, there is only one root of $P_W=\bq+1$,
but there are two two-sided cells (and both are singular at $\bq=-1$.) However, already in type $\tilde{A}_3$, one can see from 
Theorem \ref{theorem fd(1) integral} that the two-sided cell corresponding to the partition
$4=2+2$ is not singular at two of the roots of $P_{A_3}(\bq)=(1+\bq)(1+\bq+\bq^2)(1+\bq+\bq^2+\bq^3)$; see Example \ref{ex less singular cell}.
\subsubsection{Application: representations with parahoric-fixed vectors}
\label{subsubsection intro application parahoric}
In Section \ref{section parahoric-fixed}, we use the existence of the 
action of the asymptotic Hecke algebra on tempered $G$-representations to give a simple criterion for the 
existence of vectors fixed under a parahoric subgroup of $G$:
\theoremstyle{Theorem}
\newtheorem*{thm:parahoric}{Theorem \ref{thm existence of parahoric-fixed vectors}}
\begin{thm:parahoric}
Let $\pi=K(u,s,\rho)$ be a simple tempered $I$-spherical representation of 
$G$. Let $\mathcal{P}$ be a parahoric subgroup of $G$ and let $w_\mathcal{P}$ be the longest 
element in the corresponding subgroup of $\Waff$. Let $\Bb_u^\vee$ be the Springer 
fibre for $u$. 
\begin{enumerate}
\item 
If $\ell(w_\mathcal{P})>\dim\Bb_u^\vee$, then $\pi^\mathcal{P}=\{0\}$.
\item
Conversely, let $u_\Ppp$ be the unipotent conjugacy class corresponding to the two-sided cell
containing $w_{\Ppp}$. Then there exists a semisimple element $s\in Z_{G^\vee}(u_\Ppp)$, a Levi subgroup 
$M^\vee$ of $G^\vee$ 
minimal such that $(u_\Ppp,s)\in M^\vee$, and a discrete series representation $\omega\in\mathcal{E}_2(M)$ 
such that 
\[
\pi^\Ppp=i_{P_M}^G(\omega\otimes\nu)^{\Ppp}\neq \{0\}
\]
for all $\nu$ non-strictly positive and the parameter of $\pi$ is $(u_\Ppp,s)$.
\end{enumerate}
\end{thm:parahoric}
Thus starting from the regular unipotent class, $\pi(t_{w_\Ppp})=0$ until
reaching the unipotent attached to $w_\Ppp$. At this unipotent, $\Ppp$-fixed
vectors are first encountered, and $t_{w_\Ppp}$ acts by a nonzero projector 
with image contained in $\pi^{\Ppp}$. For lower cells, it may still be the case that $\pi^{\Ppp}\neq 0$, but $t_{w_\Ppp}$ will act by zero on such representations, too. 
Therefore the nonzero action of $t_{w_\Ppp}$ detects the precisely the ``most regular" unipotent attached
to $\Ppp$-spherical representations, in the sense that
if a representation $\pi$ such that $\pi^{\Ppp}\neq 0$ has the unipotent part of its
parameter equal to $u$, then $a(u)\geq a(u_\Ppp)$. In this way the distinguished 
involutions $t_{w_\Ppp}$ are more exact versions of the corresponding indicator 
functions $1_{\Ppp}$, at the expense of being more complicated to understand.
\begin{rem}
Recall from Section \ref{subsubsubection Denominators and the KL classification at roots of 1}
that for every simple $H$-module 
$M$ there is a number $a(M)$ such that $H^{\geq i}M=0$ for all $i>a(M)$. However, 
if $(u,s,\rho)$ is the $KL$-parameter of $M$, then 
without knowing that $M$ extends to a simple $J$-module, it does not follow
that $a(M)=\dim\Bb_u^\vee$.
\end{rem}
\begin{rem}
\label{rem quasiregular no parahoric}
By \cite[Theorem 4.8(d)]{affineIV} and the proof of \cite[Lemma 5.5]{affineIV}, every two-sided cell
contains a distinguished involution contained in a finite parabolic subgroup of $\Waff$, but not every
distinguished involution of a finite Coxeter group is the longest word of a parabolic subgroup, \textit{i.e.}
is of the form $w_{\Ppp}$; approximately half of two-sided cells of the finite Weyl group $W\subset\Waff$ 
do not contain any distinguished involutions
contained in proper parabolic subgroups, because of the cell-preserving bijection $w\mapsto w_0w$.
For example, this happens for the second-lowest cell for $E_8$.
\end{rem}
The existence of parahoric-fixed vectors is a rigid question, in the sense of the rigid cocentre of
Ciubotaru-He \cite{CH}. We investigate this connection further in forthcoming work.

Some time after completing the present paper, we became aware of \cite{GP}, which also studies the 
connection between the asymptotic Hecke algebra and the Plancherel theorem in type $\tilde{G}_2$, for
unequal parameters. In \textit{op. cit.} the authors speculate that the ``asymptotic Plancherel measure"
of \textit{op. cit.} should be related to the perspective of \cite{BK} on $J$. We defer investigation of
this to future work, but note that in light of both the classic work \cite{Morris} of Morris, and recent
work \cite{SolleveldProGen} of Solleveld, the unequal parameters case is relevant even to split $p$-adic 
groups. In particular, establishing results similar to those of the present paper for unequal parameters
may provide an effective way to study the algebra $\mathcal{J}$ of Braverman and Kazhdan given in Definition 
1.9 of \cite{BK}. 
\subsection{Outline of the argument}

This paper is organized according to our strategy for proving Theorems 
\ref{thm general G denominators}, Theorem \ref{thm conjecture is true GLn}, and Corollary  
\ref{cor BK map injective}.

These results
are each simple corollaries of computations with the Plancherel formula and some of Lusztig's results on $J$.
The remainder of this section will introduce $\HH$ and $J$ precisely, and recall their basic representation
theory. In Section \ref{section Plancherel}, we introduce Harish-Chandra's Plancherel formula in detail,
along with all the numerical constants that appear in it. In Section \ref{subsection algebra J as subalgebra of the Schwartz algebra}, we recall the results of Braverman-Kazhdan from \cite{BK}. There is no original 
material in the first two sections. In Section \ref{section reg trace}, we prepare to apply the 
Plancherel formula by proving that, if $f_w$ is the Schwartz function associated by Braverman-Kazhdan
to $t_w$, and $\pi$ is a tempered representation, then $\trace{\pi}{f_w}$ is sufficiently
regular so as not to complicate the denominators of $a_{x,w}$. This section is also mostly a recollection
of standard material, the only original result being Lemma \ref{Lemma trace is regular}.

In Section \ref{section proofs}, we prove most of our main results. As we are able to be more precise
in type $\tilde{A}_n$, we perform each step in parallel for type $\tilde{A}_n$ and for other types:
in Sections \ref{subsection the functions fw for GLn} and \ref{subsection the functions fw for general G}
we prove statements like those of Theorem \ref{thm general G denominators} for the Schwartz
functions $f_w$. In these sections $\bq$ is specialized to a prime power $q$. In 
Section \ref{subsection relating tw and fw} we relate the functions $f_w$ to the basis elements
$t_w$, turning statements that hold for all prime powers $q$ into statements that hold for 
the formal variable $\bq$. We are then able to prove our main results. 

In Section \ref{section parahoric-fixed}, we state our application about the existence of parahoric-fixed vectors.
\subsection{The affine Hecke algebra}
Let $F$ be a non-archimedean local field, $\Oo$ its ring
of integers and $\varpi$ be a uniformizer. Let $q$ be the cardinality of the residue field. Then
$q=p^r$ is a prime power. We write $|\cdot|_F$ for the $p$-adic absolute value on $F$; when necessary, 
$|\cdot|_\infty$ will denote the archimedean absolute value on $\C$.

Let $\G$ be a connected reductive algebraic group defined and split over $F$, $\Aa$ a maximal $F$-split torus of $\G$, and 
$X_*=X_*(\Aa)$ the cocharacter lattice of $\Aa$. Let $\pi_1(G)=X_*/\Z\Phi^\vee$ be the fundamental group, the quotient of the cocharacter lattice by the coroot lattice.
Let $\NN$ be unipotent radical of a chosen Borel subgroup $\B$, so that $\B=\Aa\NN$. Let $W$ be the finite Weyl group of $\G$, and $\Waff=W\ltimes X_*(\Aa)$ be 
the extended affine Weyl group. Write $S$ for the set of simple 
reflections in $\Waff$. Let $\mathbf{G}^\vee$ be the Langlands dual group of $\G$, taken over $\C$. We write $G=\G(F)$, $A=\Aa(F)$, etc. Where there is no danger of confusion, we also write
$\Gd$ for $\mathbf{G}^\vee(\C)$, 
$M^\vee$ for $\mathbf{M}^\vee(\C)$, etc. Let $K$ be the maximal 
compact subgroup $\G(\Oo)$. Also let $I$ be the Iwahori subgroup of $G$ that is the preimage of 
$\B(\F_q)$ in $K$. We sometimes write $P_{\G/\B}$ for $P_W$, as this polynomial is also the Poincar\'{e} polynomial of the flag variety $(\G/\B)(\C)$.

We write $\HH$ for the affine Hecke algebra of $\Waff$. It is a unital associative algebra over the ring 
$\mathcal{A}=\C[\bq^{\frac{1}{2}},\bq^{-\frac{1}{2}}]$ (in fact, it is defined over 
$\Z[\bq^{\frac{1}{2}},\bq^{-\frac{1}{2}}]$ but we will work over $\C$ to avoid having to introduce extra notation later), where $\bq^{\frac{1}{2}}$ is a formal variable. We will think of $\C^\times$ as 
$\Spec\mathcal{A}$. The algebra $\HH$ has the Coxeter presentation with standard basis 
$\{T_w\}_{w\in\tilde{W}}$ with $T_wT_{w'}=T_{ww'}$ if $\ell(ww')=\ell(w)+\ell(w')$ and quadratic relation 
$(T_s+1)(T_s-\bq)=0$ for $s\in S$. We write $\theta_\lambda$ for the generators of the Bernstein subalgebra.

Recall from \cite{KL79} the two Kazhdan-Lusztig bases $\{C_w\}_{w\in\Waff}$
and $\{C'_w\}_{w\in\Waff}$ of $\HH$, where 
\[
C'_w=\bq^{-\frac{\ell(w)}{2}}\sum_{x\leq w}P_{x,w}(\bq)T_x
\]
for the Kazhdan-Lusztig polynomials $P_{x,w}$. Write
$C_xC_y=\sum_{z\in \Waff}h_{x,y,z}C_z$.
The \emph{inverse Kazhdan-Lusztig polynomials} $Q_{y,x}$ are the unique family of polynomials
satisfying
\[
T_x=\sum_{y\leq x}(-1)^{\ell(x)-\ell(y)}\bq^{\frac{\ell(y)}{2}}Q_{y,x}(\bq)C'_y,
\]
or equivalently,  satisfying
\[
\sum_{z\leq y\leq x}(-1)^{\ell(x)-\ell(y)}Q_{y,x}(\bq)P_{z,y}(\bq)=\delta_{z,x}
\]
along with some restrictions on their degrees. For example, we shall use in Section 
\ref{subsubsection the functions fw and the basis elements tw} that 
$\deg Q_{y,x}\leq \frac{1}{2}(\ell(x)-\ell(y)-1)$. See \cite{BjornerBrenti} for further exposition.

If 
\[
\varphi\colon(\Waff, S)\to (\Waff, S)
\]
is a Coxeter group automorphism of $\Waff$, then 
\[
T_w\mapsto T_{\varphi(w)}
\]
is an algebra automorphism of $\HH$ commuting with the bar involution, and therefore given equivalently by
\[
C'_w\mapsto C'_{\varphi(w)}
\]
and
\begin{equation}
\label{eqn Cox gp aut on KL basis}
C_w\mapsto C_{\varphi(w)}.
\end{equation}

It is well-known that there is an isomorphism of associative $\C$-algebras 
\[
\HH|_{\bq=q}:=\HH\otimes_\mathcal{A}\C\to C_c^\infty(G)^{I\times I}=:H,
\]
where $\bq$ acts on $\C$ by multiplication by $q$.
\subsection{The asymptotic Hecke algebra}
\begin{dfn}
\label{dfn a-function}
\emph{Lusztig's a-function} $a\colon\tilde{W}\to\Z_{\geq 0}$ is defined such that $a(w)$ is the minimal value such that 
$\bq^{\frac{a(w)}{2}}h_{x,y,w}\in\mathcal{A}^+=\C[\bq^{1/2}]$ for all $x,y\in\tilde{W}$.
\end{dfn}
The $a$-function is constant of two-sided cells of $\Waff$. Obviously, $a(1)=0$,
and the $a$-function obtains its maximum, equal to the number of positive roots, on the two-sided cell 
containing the longest word $w_0\in W$. In general, under the bijection between two-sided cells $\cc$ of 
$\Waff$ and unipotent conjugacy classes $u=u(\cc)$ in $\Gd$ of \cite{affineIV}, we have 
\[
a(\cc)=\dim_\C(\Bb_u^\vee),
\]
where $\Bb_u^\vee$ is the Springer fibre of $u$. We have $a(w)\leq \ell(w)$ for all $w\in\Waff$.

In \cite{affineII}, Lusztig defined an associative algebra
$J$ over $\C$ equipped with an injection $\phi\colon \HH\into J\otimes_\C\mathcal{A}$
which becomes an isomorphism after taking a certain completion, to be recalled in Section 
\ref{subsubsection completions of HH and JotimesZA},
of both sides. As a $\C$-vector space, $J$ has basis $\{t_w\}
_{w\in \Waff}$. The structure constants of $J$ are obtained
from those of $\HH$ written in the $\{C_w\}_{w\in \Waff}$-basis 
under the following procedure: first, the integer $\gamma_{x,y,z}$  is defined by the 
condition
\[
\bq^{\frac{a(z)}{2}}h_{x,y,z^{-1}}-\gamma_{x,y,z}\in \bq\mathcal{A}^+.
\]
One then defines
\[
t_xt_y=\sum_{z\in\Waff}\gamma_{x,y,z}t_{z^{-1}}.
\]
Surprisingly, this defines a unital associative algebra with unit
\[
1_J=\sum_{d\in\mathcal{D}}t_d,
\]
where $\mathcal{D}$ is the (finite) set of \emph{distinguished involutions}
\cite{affineII}. The elements $t_d$ are orthogonal idempotents in $J$, which
decomposes as a direct sum indexed 
by the two-sided cells of $\Waff$ in the sense of \cite{affineI}. Each left cell, again in the sense of \textit{op. cit.}, contains
a single distinguished involution which is the unit in the ring $t_dJt_d$.
If $\cc$ is a two-sided cell, then $J_{\cc}$ is a 
is a unital ring with unit 
\[
1_{J_{\cc}}=\sum_{d\in\mathcal{D}\cap\cc}t_d.
\]

Lusztig further defined a map of algebras
\[
\phi\colon\HH\to J\otimes_{\C}\mathcal{A}
\]
given by
\[
\phi(C_w)=\sum_{z\in W,~d\in\mathcal{D},~a(z)=a(d)}h_{w,d,z}t_z.
\]
Write $\phi_q$ for the specialization of this map when $\bq=q$.
It is known \cite[Proposition 1.7]{affineIII} that $\phi_q$ is injective for all $q\in\C^\times$.
\begin{lem}
\label{lem based ring auts linearity}
Let 
\[
\varphi\colon(\Waff, S)\to (\Waff, S)
\]
be a Coxeter group automorphism. 
Then 
\begin{enumerate}
\item 
The map $t_w\mapsto t_{\varphi(w)}$ defines a  based ring automorphism of $J$, which we also denote $\varphi$.
\item
The map $\phi$ is $\varphi$-linear, in the sense that it commutes with the automorphism from the first statement and the automorphism \eqref{eqn Cox gp aut on KL basis}.
\end{enumerate}
\end{lem}
\begin{proof}
The first statement is \cite[2.2(g)]{LX}. 

For the second statement, note that 
$\varphi$ acts on $\mathcal{D}$ and that 
\[
\phi(C_{\varphi(w)})=\sum_{d_1\sim_L z_1}h_{\varphi(w),d_1,z_1}t_{z_1}
\]
and
\[
\varphi\left(\phi(C_w)\right)=\sum_{d_2\sim_Lz_2}h_{w,d_2,z_2}t_{\varphi(z_2)}.
\]
Now if $z_1=\varphi(z_2)$, then $d_1=\varphi(d_2)$, and
\[
h_{\varphi(w), d_1, z_1}=h_{\varphi(w), \varphi(d_2), \varphi(d_2)}=h_{w, d_2, z_2}
\]
by \eqref{eqn Cox gp aut on KL basis}.
\end{proof}
\subsubsection{Deformations of the group ring}
Upon setting $\bq=1$, $\HH|_{\bq=1}$ is isomorphic to $\C[\tilde{W}]$, and so $\HH$ is a deformation of
the group algebra of $\Waff$.

Let, temporarily, $W$ be any finite Coxeter group. Then one can define its Hecke algebra $\HH$, an algebra over 
$\Z[\bq^{\frac{1}{2}},\bq^{-\frac{1}{2}}]$ which deforms the group ring $\Z[W]$. Let $q\in\C^\times$. For 
all but finitely-many values of $q$, all roots of unity, the algebras $\HH_{\bq=q}$
are trivial deformations of $\C[W]$, and hence are all isomorphic. 
However, this isomorphism requires choosing a square root of $q$. The affine Hecke algebra provides a canonical 
isomorphism: away from finitely-many $q$, we have that $\HH|_{\bq=q}$ 
is isomorphic to $J$, and $J$ is defined over $\Z$ (although, as stated above, we will view it as a $\C$-
algebra to unburden notation), see \cite[Section 20.1 (e)]{LusztigBook}.
\begin{ex}
Let $W=\genrel{1,s}{s^2=1}$ be the Weyl group of type $A_1$. The Kazhdan-Lusztig $C_w$-basis elements
are $C_1=T_1$ and $C_s=\bq^{-\frac{1}{2}}T_s-\bq^{\frac{1}{2}}T_1$, and $\mathcal{D}=W$ in this case. There are two two-sided cells in $W$,
and one can easily check that 
\[
\phi\colon C_1\mapsto t_1+t_s
\]
and
\[
\phi\colon C_s\mapsto -\left(\bq^{\frac{1}{2}}+\bq^{-\frac{1}{2}}\right)t_s.
\]
Specializing $\bq=q$, we see that $\phi$ becomes an isomorphism whenever 
$\left(q^{\frac{1}{2}}+q^{-\frac{1}{2}}\right)\neq 0$, that is, whenever $q\neq -1$.
\end{ex}
\subsection{Representation theory of $\HH$ and $J$}
Recall the classification of finite-dimensional $\HH$-modules given in \cite{KLDeligneLanglands}. For an 
extended exposition with slightly 
different conventions, we refer the reader to \cite{CG}. The primary difference between the setup we require and that of \cite{CG} is that we must be able to defer specializing $\bq$ until the last possible moment, whereas specializing $\bq$ is the first step of the construction as given in \cite{CG}. In particular, let 
$u\in\mathbf{G}^\vee(\C)$ be a unipotent element, and $s\in\mathbf{G}^\vee(\C)$ be a semisimple element such that $us=su$.
Let $\rho$ be an irreducible representation of the simultaneous centralizer $Z_{\Gd}(s,u)$. 
The \emph{standard} $H$-\emph{module} $K(s,u,\rho)$ is a certain, and in general reducible, $H$-module
defined using the geometry of the flag variety of the Langlands dual group. It may be the zero module;
we say that $(u,s,\rho)$ is \emph{admissible} when this does not happen. Having fixed $s$ and $u$,
we say that $\rho$ is admissible if $(u,s,\rho)$ is.

We now recall an algebraic version of the Langlands classification.
As we do not have access to the notion of absolute value of $\bq$, the classical definitions 
of tempered and discrete-series representations of the corresponding $p$-adic group $G$ are not available to us. However, Kazhdan-Lusztig
provide the following algebraic generalization. Let $\nbk=\C(\bq^{-\frac{1}{2}})$ and $\bk$
be the algebraic closure of $\nbk$. We write $\HH_{\bk}$ for $\HH\otimes_\mathcal{A}\bk$ and recall another definition of Lusztig's from \cite{affineIV}. 
\begin{dfn}
Let $M$ be a $H_{\bk}$-module finite-dimensional over $\bk$. Say that $m\in M$ is an \emph{eigenvector} if $\theta_x\cdot m=\chi_m(x)m$
for all dominant $x$ in $X_*$. As $\chi_m$ is a character of the coweight lattice, it corresponds to 
an element $\sigma_m\in \mathbf{A}^\vee(\bk)$ in the sense that, for all cocharacters $x$ of $A$, we have
\[
\chi_m(x)=x(\sigma_m)
\]
where $x$ is viewed as a character of $\mathbf{A}^\vee$. Then $M$ is of \emph{constant type} if 
there is a semisimple element $s'\in\mathbf{G}^\vee(\C)$ and a morphism of algebraic groups
\[
\phi'\colon\SL_2(\C)\to Z_{\Gd}^0(s')
\]
such that for all eigenvectors $m$ of $M$, the element $\sigma_m$ is $\mathbf{G}^\vee(\bk)$-conjugate to 
\[
\phi'(\diag(\bq^{1/2},\bq^{-1/2}))s',
\]
where by abuse of notation we have written $\phi'$ again for the base-change to $\bk$.
\end{dfn}
The idea of the name of the definition is that $s'\in\mathbf{G}^\vee(\C)$ is a ``constant element" not 
depending on $\bq$.

Next is a generalization of Casselman's criterion, which as such, 
requires a choice of dominant weights. Following \cite{KLDeligneLanglands} (see the proof of Prop. 1.6 of \textit{loc. cit.}) and \cite[Section 1.6]{affineIV}, we choose the positive roots to be those occurring in $\g/\bB$.


Following \cite{affineIV}, we choose a morphism of groups
$V\colon\bk^\times\to\R$ such that $V(\bq^{\frac{1}{2}})=1$ and $V(a\bq^{\frac{1}{2}}+b)=0$
for all $a\in \C$, $b\in \C^\times$.
\begin{dfn}[\cite{affineIV}, c.f. \cite{KLDeligneLanglands}]
Let $M$ be any finite-dimensional $\HH_{\bk}$-module. We say that $M$ is $V$-\emph{tempered} 
if all eigenvalues $\nu$ of $\theta_\lambda$ for all dominant $\lambda\in X_*(\Aa)$ 
satisfy $V(\nu)\leq 0$. 
\end{dfn}
The representation theory of $J$ is very well understood. We shall recall some notation and then 
state some major classification results of Lusztig, which relate the representation theory of $J$ to 
certain $\HH$-modules defined by Kazhdan-Lusztig.
\begin{dfn}
Let $E$ be a $J$-module. Then $E\otimes_\C\nbk$ is a $J\otimes_\C\nbk$-module. Hence
$\HH_\nbk$ acts on $E$ via $\phi$. Denote the resulting $\HH_\nbk$ module by ${}^\phi E$.
\end{dfn}
\subsubsection{Involutions on $\HH$}
\label{subsubsection Involutions on HH}
For $x\in\Waff$, let $\omega(x)\in\pi_1(G)$ label the $W\ltimes\Z\Phi^\vee$-coset of $\Waff$ containing $x$, and write $\omega(x)=\omega(x)_f\omega(x)_t\in W\ltimes X_*$. Then every $y\leq x$ is also in the coset of $\omega(x)$.

\begin{dfn}
Let $j\colon\HH\to\HH$ be the ring (and not $\mathcal{A}$-algebra) involution of $\HH$ defined
by $j(\sum_wa_wT_w)=\sum_w\bar{a}_w(-1)^{\ell(w)}\bq^{-\ell(w)}T_w$.
\end{dfn}
The $j$-involution exchanges the $\{C_w\}$ and $\{C_w'\}$-bases
\cite{KL79}.

The reason for our choice of conventions, which differ slightly from those of 
\cite{affineIII} and \cite{affineIV}, is the presence of the involution $\dagger$ and its exchange of temperedness and anti-temperedness in the relationship between $H$-modules and $J$-modules; see Theorem \ref{J summary theorem}
and Lemma \ref{lem Goldman involution exchanges KL bases} below. 
\begin{dfn}
\label{dfn Goldman involution}
\begin{enumerate}
\item[(a)]
Define the $\mathcal{A}$-algebra involution ${}^\dagger(-)$ of $\HH$
by setting
\[
{}^\dagger T_w=(-1)^{\ell(w_f)}\bq^{\ell(w)}T_{w^{-1}}^{-1},~~~w=w_f\lambda\in\Waff.
\]
Note that the sign factor depends only on $w_f$.
\item[(b)]
Let $h\mapsto {}^*h$ be the $\mathcal{A}$-involution defined in terms of the Bernstein presentation of $\HH$: by 
\[
{}^*T_w=(-1)^{\ell(w)}\bq^{\ell(w)}T_{w^{-1}}^{-1},~~~ w\in W
\]
and
\[
{}^*\theta_\lambda=\theta_{\lambda}^{-1},~~~\lambda\in X_*. 
\]
\item[(c)] (\cite[Section 5.a]{PrasadDuke}.)
Let $\PrInv$ be the $\mathcal{A}$-linear involution defined by 
\[
\PrInv(T_w)=T_{\PrInv(w)}
\]
induced by the Coxeter group automorphism given by 
$\PrInv(s)=w_0sw_0$ for $s\in S\setminus\sett{s_0}$ and $\PrInv(s_0)=s_0$, so that $\bar{\kappa}(\lambda)=-w_0(\lambda)$ for $\lambda\in X_*$; equivalently
\[
\PrInv(T_w)=T_{w_0ww_0^{-1}},~~~ w\in W,~~~
\PrInv(\theta_\lambda)=\theta_{-w_0(\lambda)},~~~\lambda\in X_*.
\]
\item[(d)] (\cite[Section 5]{Barbasch-Moy}.)
Let $h\mapsto {}^\bullet h$ be the $\mathcal{A}$-linear anti-involution defined by 
\[
{}^\bullet T_w=T_{w^{-1}},~~~w\in W
,~~~
{}^\bullet\theta_\lambda=\theta_\lambda,~~\lambda\in X_*.
\]
\item[(e)] (\cite[Section 5]{Barbasch-Moy}.)
Let $h\mapsto {}^\star h$ be the $\mathcal{A}$-linear anti-involution defined by 
\[
{}^\star T_w=T_{w^{-1}},~~~w\in W
,~~~
{}^\star\theta_\lambda=T_{w_0}\theta_{\PrInv(\lambda)}T_{w_0}^{-1},~~\lambda\in X_*.
\]
\end{enumerate}
\end{dfn}
By \cite[Prop. 2.9]{OpdamSpectral}, ${}^\star T_w=T_{w^{-1}}$ for $w\in \Waff$.
When $G$ is simply-connected, the involution $\dagger$ agrees with the \emph{Goldman involution} of $\HH$.
Now we have
\begin{lem}
\label{lem Goldman involution exchanges KL bases}
We have
\begin{enumerate}
\item[(a)] 

We have ${}^\dagger C_x=(-1)^{\ell(\omega(x)_f)}j(C_w)=(-1)^{\ell(x)+\ell(\omega(x)_f)}C'_x$ for all $x\in\Waff$.
\item[(b)]
We have ${}^\dagger\theta_\lambda= T_{w_0}\theta_{-\PrInv(\lambda)}T_{w_0}^{-1}=T_{w_0}\theta_{\PrInv(\lambda)}^{-1}T_{w_0}^{-1}$.
\item[(c)]
We have ${}^\bullet{}^\star{}^\dagger (-)={}^*(-)$ as involutions on $\HH$.

%
\item[(d)]
We have ${}^\bullet(-) \circ{}^\star(-)=T_{w_0}^{-1}\PrInv(-)T_{w_0}$ as automorphisms of $\HH$.
\item[(e)]
We have the commutative diagram
\[
\begin{tikzcd}
\HH\arrow[d, "{}^*(-)"{swap}]\arrow[r, "{}^\dagger(-)"]&\HH\arrow[d, "T_{w_0}^{-1}\PrInv(-)T_{w_0}"]\arrow[rr, "\phi"]&&J\otimes\mathcal{A}\arrow[d, "\phi(T_{w_0})^{-1}\PrInv(-)\phi(T_{w_0})"]\\
\HH\arrow[r, "\id"]&\HH\arrow[rr, "\phi"]&&J\otimes\mathcal{A}.
\end{tikzcd}
\]
\end{enumerate}
\end{lem}
\begin{proof}
By definition, if $x=\omega x'$ for $x'\in W\ltimes\Z\Phi^\vee$ and $\omega\in\pi_1(G)$, then $C'_x=T_\omega C'_{x'}$ and $C_x=T_\omega C_{x'}$. Further, if $\omega=\omega_f\omega_t\in\Waff$ has $\ell(\omega)=0$, then 
${}^\dagger T_\omega=(-1)^{\ell(\omega_f)}T_\omega$. Therefore it suffices to show that 
\[
\overline{{}^\dagger h}={}^\dagger\left(\bar{h}\right)=j(h)
\]
as $\mathcal{A}$-antilinear automorphisms of $\HH$ for $G$ simply connected. In this case, $\ell(\lambda)\in 2\Z$ for any dominant $\lambda\in X_*$, whence ${}^\dagger(-)$ agrees with the involution $T_w\mapsto(-1)^{\ell(w)}\bq^{\ell(w)}T_{w^{-1}}^{-1}$, $w\in W\ltimes\Z\Phi^\vee$. Therefore we compute 
\[
\overline{\sum_{x}b_x{}^\dagger T_x}=\sum_{x}\bar{b}_x(-1)^{\ell(x)}\overline{\bq^{\ell(x)}\bar{T}_x}
=\sum_{x}\bar{b}_x(-1)^{\ell(x)}\bq^{-\ell(x)}T_x=j\left(\sum_{x}b_xT_x\right)
\]
whereas
\[
{}^\dagger\left(\overline{\sum_{x}b_xT_x}\right)={}^\dagger\left(\sum_{x}\bar{b}_x(-1)^{\ell(x)}\bq^{-\ell(x)}
{}^\dagger T_x\right).
\]
Thus we have $\overline{^\dagger C_w}=(-1)^{\ell(w)}C'_w$ for $G$ simply-connected, whence (a).

It suffices to prove (b) for $\lambda$ dominant, in which case
%
\[
{}^\dagger\theta_\lambda=\bq^{-\frac{\ell(\lambda)}{2}}{}^\dagger T_\lambda
=
\bq^{\frac{\ell(\lambda)}{2}}T_{-\lambda}^{-1}
=
\bq^{\frac{\ell(\lambda)}{2}}\left({}^\star T_{\lambda}\right)^{-1}
=
\left({}^\star \theta_{\lambda}\right)^{-1}
=
\left(T_{w_0}\theta_{\PrInv(\lambda)}T_{w_0}^{-1}\right)^{-1}
=
T_{w_0}\theta_{\PrInv(\lambda)}^{-1}T_{w_0}^{-1},
\]
where we used the equivalence of the definitions in Definition \ref{dfn Goldman involution} (e). This shows (b).

For (c), we again compute, for $\lambda$ dominant, that
\[
{}^\bullet{}^\star{}^\dagger\theta_\lambda
={}^\bullet{}^\star \left(T_{w_0}\theta_{-\PrInv(\lambda)}T_{w_0}^{-1}\right)
={}^\bullet\theta_{-\lambda}=\theta_{-\lambda}.
\]
Agreement on $T_w$ for $w\in W$ follows from the fact that ${}^{\bullet\star}(-)$ is the identity on the finite Hecke algebra.
%
%
This shows (c).

On the Bernstein subalgebra, part (d) follows from the definitions. On the finite Hecke algebra, we must show that the right hand side is the identity automorphism. For $s$ a finite reflection, write $w_0sw_0=\PrInv(s)$, again a finite reflection, so that $\PrInv(s)w_0s=w_0$, and 
\[
T_{w_0}T_sT_{w_0}^{-1}=T_{w_0}T_s T_s^{-1}T_{w_0s}^{-1}=T_{w_0}T_{w_0s}^{-1}=T_{\PrInv(s)}T_{w_0s}T_{w_0s}^{-1}=T_{\PrInv(s)}.
\]
Therefore on the finite Hecke algebra, $\PrInv$ is given by conjugation by $T_{w_0}$.

Commutativity of the right square in (e) follows from Lemma \ref{lem based ring auts linearity}, and commutativity of the left square follows from (c) and (d).

%
\end{proof}
Given an $\HH$ (or $\HH_{\nbk}$)-module $M$, define ${}^\dagger M$ to be the same vector space with the $\HH$-action twisted by this involution, and likewise for other automorphisms.
\begin{cor}
Twisting by ${}^\dagger(-)$ exchanges tempered and anti-tempered $H$-modules.
\end{cor}
\begin{proof}
Immediate from Lemma \ref{lem Goldman involution exchanges KL bases} (b).
\end{proof}
In fact, more is true: by \cite[Thm. 2]{Kato}, ${}^\dagger M$ is the Aubert-Zelevinski dual of $M$.
\begin{rem}
There is another natural auto-equivalence of the category of admissible representations of $G$ which exchanges tempered and anti-tempered representations, namely Bernstein's cohomlogical duality. This functor differs from Aubert-Zelevinksi duality by the contragredient \cite{SchneiderStuhler} (c.f. \cite{PrasadNori}).
On semisimple $H$-modules $\pi$, we have ${}^{\PrInv}\pi=\tilde{\pi}$, by \cite[Prop. 6.3]{PrasadDuke}. Therefore twist by the involution ${}^*(-)$, which manifestly exchanges tempered and anti-tempered representations, induces on semisimple modules the cohmological duality. However, it does not do so in general, because $\PrInv$ does not induce the contragredient in general. Indeed, for $G=\SL_2(F)$, $\PrInv=\id$ but not all non-unitary principal series of $G$ are self-dual. Therefore the operation ${}^*(-)$ is less natural than ${}^\dagger(-)$; this is perhaps reflected by the fact the formulas in Theorem \ref{thm general G denominators} are nicer than those for
$(\phi\circ{}^*(-))^{-1}(t_w)$, which are related to those of the Theorem by Lemma \ref{lem Goldman involution exchanges KL bases} (e).
%
%
\end{rem}

We now summarize the relationship between representations of $\HH$ and of $J$. 
\begin{theorem}[\cite{affineIV}, Prop. 2.11, Thm. 4.2, Prop. 4.4]
\label{J summary theorem}
There are bijections of sets
\begin{center}
\begin{tikzcd}
(u,s,\rho)\arrow[d, mapsto]
&
\sets{(u,s,\rho)}{\rho~\text{admissible},~us=su}/\mathbf{G}^\vee(\C)\arrow[d]
\\
{}^* K(s,u,\rho)\otimes_{\mathcal{A}}\nbk\arrow[d, equal] 
&
\left\lbrace M\in \HH_{\nbk}-\Mod\mid {}^*M\otimes_\nbk \bk~\text{simple,}~V-\text{tempered}
\right.
\\
{}^{\phi}E=E\otimes_{\C}\nbk\in \HH_{\nbk}-\Mod 
& \left.\HH_{\bk}- \text{module of constant type} \right \rbrace
\\
E\arrow[u, mapsto]&\sets{E\in J-\Mod}{E~\text{is simple}}\arrow[u],
\end{tikzcd}
\end{center}
where $K(u,s,\rho)$ is a standard module as in \cite{KLDeligneLanglands}. Moreover, for a simple $J$-module $E$,
\begin{enumerate}
\item 
$E$ is finite-dimensional over $\C$;
\item
There is a unique two-sided cell $\cc=\cc(E)$ of $\Waff$ such that $\trace{E}{t_w}\neq 0$
implies $w\in\cc$.
\item
$\trace{E}{t_w}$ is the constant term of the polynomial
\[
(-\bq^{1/2})^{a(\cc(E))}\trace{M}{C_w}\in\C[\bq^{\frac{1}{2}}]
\]
where $M\simeq {}^\phi E$.
\end{enumerate}
\end{theorem}
In particular, $\trace{E}{t_w}$ is independent of $\bq$, and upon specializing 
$\bq=q$ a prime power, will be a regular function in the twisting character in the setting of the Payley-Wiener theorem for 
the Iwahori-Hecke algebra of the $p$-adic group $G$, as we will explain in greater detail below.

 
We will comment even further in Section \ref{subsection algebra J as subalgebra of the Schwartz algebra} on the necessity of twisting by some $\HH$-involution exchanging tempered and anti-tempered modules, and how it is sufficient to twist $\phi$ by either $(-)^*$ or $(-)^\dagger$, but that the latter twist leads to nicer formulas.
%
%
\section{Harish-Chandra's Plancherel formula}
\label{section Plancherel}
We recall the notation and classical results we will need about the Plancherel formula. For Iwahori-biinvariant Schwartz functions, the Plancherel formula is known explicitly
for all connected reductive groups, and is due to Opdam in \cite{OpdamSpectral}. In the case of $G=\GL_n(F)$, we shall refer
instead to \cite{AubPlym} (where in fact the Plancherel formula is computed explicitly in its entirety for 
$\GL_n$). In the case $\G=\Sp_4$, we shall refer to the unpublished work \cite{AKSp4} of Aubert and Kim.
For $\G=G_2$, we will refer to Parkinson \cite{Parkinson}.

In this section $q$ is a prime power (or at least a real number of absolute value strictly greater 
than $1$). The formal variable $\bq$ will not appear in this section. 
\subsection{Tempered and discrete series representations}
We take our definitions of discrete series and tempered representations from \cite[III.1]{Waldspurger} and \cite[III.2]{Waldspurger} respectively. By parabolic induction we always understand normalized induction.

\begin{dfn}
A smooth admissible representation $\omega$ of $G$ belongs to the \emph{discrete series} if $\omega$ admits a unitary central character and all matrix coefficients of $\omega$ are square-integrable
modulo $Z(G)$. We write $\mathcal{E}_2(G)$ for the space of irreducible discrete series, and $\mathcal{E}_2(G)^I$ for the space of irreducible discrete series with nontrivial Iwahori-fixed vectors.
\end{dfn}
Let $v$ be the $K$-fixed vector in the self-contragredient representation $\Ind_{B}^G(\triv)^K$ such that $v(1)=1$. Define $\Xi(g)=\pair{\pi(g)v}{v}$ to be the corresponding matrix coefficient. 
\begin{dfn}
\label{dfn tempered function}
A smooth function $f$ on $G$ is \emph{tempered} if there is $C>0$ and $r\in\R$ such that 
\[
|f(g)|\leq C\Xi(g)\left(1+\log\norm{g}\right)^r,
\]
where $\norm{g}\geq 1$ is defined as in \cite[p.242]{Waldspurger}.
\end{dfn}

\begin{dfn}
A smooth admissible representation $\pi$ of $G$ is \emph{tempered} if all its matrix coefficients are tempered functions in the sense above.
\end{dfn}
We write $\mathcal{M}_t(G)$ for the category of tempered representations of $G$.
If a tempered representation admits a central character, the central character
takes values in the circle group $\T\subset\C^\times$.

%
%
%
The tempered representations are built from the discrete series according to the following theorem 
of Harish-Chandra, as related in \cite[Prop. III.4.1]{Waldspurger}.
\begin{theorem}[Harish-Chandra]
Let $P$ be a parabolic subgroup of $G$ with Levi subgroup $M$, and let $\omega\in\Ee_2(M)$.
Let $\nu$ be a unitary character of $M$. Then $\pi=\Ind_P^G(\omega\otimes\nu)$ is a tempered
representation. Every simple tempered representation is a direct summand of a representation of this
form.
\end{theorem}
\subsubsection{Formal degrees of discrete series representations}
We will soon study the Plancherel decomposition $f=\sum_{M}f_M$ of the Schwartz function $f$ determined by an element of $J$ as explained in Section \ref{subsection algebra J as subalgebra of the Schwartz algebra}. As will
be explained below, each function $f_M$ is given by an integral formula that involves several 
constants that depend on the Levi subgroup $M$, or are functions on the discrete series of $M$. These 
constants are rational functions of $q$, the most sensitive of which is the formal degree $d(\omega)$ of 
$\omega\in\mathcal{E}_2(M)^I$. Much is known about formal degrees for $I$-spherical $\omega$; the most general current result seems to be
\begin{theorem}[\cite{Solleveld}, \cite{fomFormalDegUnip}, Theorem 5.1 (b), \cite{Gross-Reeder} Proposition 4.1]
\label{FOS theorem}
Let $\G$ be connected reductive over $F$. Let $\omega$ be any unipotent---in particular, any Iwahori-spherical---discrete series 
representation of $G=\G(F)$. Then $d(\omega)$ is a rational function of $q$, the numerator and denominator
of which are products of factors of the form $q^{m/2}$ with $m\in \Z$ and $(q^n-1)$ with $n\in\N$.
Moreover, there is a polynomial $\Delta_G$ depending only on $G$ and $F$ such that $\Delta_Gd(\omega)$
is a polynomial in $q$.
\end{theorem}
This result is proven by first proving that the Hiraga-Ichino-Ikeda conjecture
\cite{HII} holds for unipotent discrete series representations. Note that \cite{Solleveld}, \cite{fomFormalDegUnip} and \cite{HII} all use the normalization of the Haar measure on $G$ defined in \cite{HII}. This 
normalization gives in our setting $\mu_{\mathrm{HII}}(K)=q^{\dim\G}\#\G(\F_q)$. Hence, noting that $\#\G(\F_q)=P_{\G/\B}(q)\cdot\#\B(\F_q)$
and that, as $\F_q$ is perfect, $\#B(\F_q)$ is a polynomial in $q$, we have
\[
\mu_I=\frac{1}{q^{\dim\G}\#\B(\F_q)}\mu_{\mathrm{HII}},
\]
and so this question of normalization cannot affect the denominators of $d(\omega)$, for any
Levi subgroup.

In the Iwahori-spherical case,
Opdam showed the above result in \cite[Proposition 3.27 (v)]{OpdamSpectral}, although with less control 
over the possible factors appearing in the numerator and denominator of $d(\omega)$. We emphasize that
\textit{op. cit.} does not make the splitness assumption we allow ourselves.
\begin{rem}
Proposition 4.1 of \cite{Gross-Reeder} studies not the $\gamma$-factor we are interested in, but 
rather its quotient by the $\gamma$-factor for the Steinberg representation. However, accounting for  
the use of the Euler-Poincar\'{e} measure $\mu_{\mathrm{HII}}$ on $G$, and known formula for the formal 
degree of the Steinberg representation, one may recover our desired statement about formal degrees from the 
main theorem and equation (61) of \cite{Gross-Reeder}.
\end{rem}
\subsection{Harish-Chandra's canonical measure}
\label{subsection HC canonical measure}
In this section, we recall the standard coordinates used in \cite{AubPlym}, and \cite{Waldspurger}.
We follow both references closely. These conventions differ slightly from the original \cite{HCP}. 
Everything in this section is standard, but we include details because we do require explicit measures with which to compute. 
We state the below for general $\G$, for application to each standard Levi subgroup of $\G$.
\subsubsection{Unramified characters}
\label{subsubsection Unramified characters}
Let $X^*(\G)=\Hom{\G}{\Gm}$ denote the rational characters of $\G$ defined over $F$.
Let $\aAG:=(X^*(\Aa)\otimes_\Z\R)^*=(X^*(\G)\otimes_\Z\R)^*$ be the real Lie algebra of the maximal split central torus $\Aa=\Aa_G$ of $\G$ and let ${\aAG}_\C$ be its complexification.  We have a map 
\[
X^*(\G)\to\Hom{G/G^1}{\C^\times}=:\X(G)
\]
given by $\chi\mapsto|\chi|_F$, where $|\chi|_F(g)=|\chi(g)|_F$ and 
$G^1=\bigcap_{\chi\in X^*(\G)}\ker|\chi|_F$. 
This gives the unramified characters $\X(G)$ a complex manifold structure under which 
$\X(G)\simeq(\C^\times)^{\dim_\R\aAG}$. For indeed, we have
\begin{center}
\begin{tikzcd}
{\aAG}_\C^*\arrow[r]& \X(G)\arrow[r]&1
\end{tikzcd}
\end{center}
given by 
\[
\chi\otimes s\mapsto (g\mapsto|\chi(g)|_F^s=q^{-s\val(\chi(g))}).
\]
The kernel is spanned by all $\chi\otimes s$ such that $s\val(\chi(g))\in\frac{2\pi i}{\log q}\Z$ for all 
$g\in G$. Hence the kernel is $\frac{2\pi i}{\log q}R$, where $R\subset X^*(\Aa)\otimes_\Z\Q$ is a 
lattice. In this way the quotient $\X(G)$ is a complex manifold.

\subsubsection{Unitary unramified characters}
Denote $\mathrm{Im}\,\X(G)$ the set of unitary unramified characters taking values in 
the unit circle $\T\subset\C^\times$. This notation is justified, as if $|\chi|_F=|\chi'|_F$, then
$\mathrm{Re}(\chi)=\mathrm{Re}(\chi')\in\aAG^*$. Hence we can define $\mathrm{Im}\,\X(G)$ to be the unramified 
characters coming from pure imaginary elements of ${\aAG}_\C^*$.

The surjection $\mathrm{Im}\,\X(G)\to\mathrm{Im}\,\X(A)$ has finite kernel, and $\mathrm{Im}\,\X(A)$ is 
compact. We chose the Haar measure on it with volume one.
\subsubsection{Action by twisting and the canonical measure}
The group $\X(G)$ acts on admissible representations of $G$ by 
twisting: $\omega\mapsto\omega\otimes\nu$ for $\nu\in\X(G)$. This restricts to an action of $\mathrm{Im}\,\X(G)$ on $\mathcal{E}_2(G)$. 

Pulling this action back to $i\aAG^*$, and given a representation $\omega$, let $L^*$ be its stabilizer in $i\aAG^*$, so that the orbit $\mathfrak{o}$ of $\omega$ is identified as $i{\aAG}^*/L^*$. This gives
$\mathfrak{o}=\mathrm{Im}\,\X(G)\cdot\omega$ the structure of a real submanifold of the larger orbit 
$\mathfrak{o}_\C={\aAG}^*_\C/L^*=\X(G)\cdot\omega$.
\begin{dfn}
Given an orbit $\mathfrak{o}\subset\mathcal{E}_2(G)$, the \emph{Harish-Chandra 
canonical measure} on $\mathfrak{o}$ is the Euclidean measure on $\mathfrak{o}$ whose pullback to 
$\mathrm{Im}\,\X(G)$ agrees with the pullback of the Haar measure on $\mathrm{Im}\,\X(A)$.
\end{dfn}
Hence even though the construction of the canonical measure is slightly involved, in practice it will
be easy to recognize as being essentially the Haar measure on the compact torus $\mathrm{Im}\,\X(A)$.
%
%
%
%
%
%
%
\begin{ex}
\label{ex SL2 unit circle conventions}
If $G=\SL_2(F)$ and $M=A$ is the diagonal torus, we have $\aAG^*\simeq\R$ and $R=\Z$ so that a fundamental domain for
$\aAG^*/\frac{2\pi}{\log q}R$ is $\left[-\frac{\pi}{\log q},\frac{\pi}{\log q}\right)$ and the canonical
measure $d\nu=\frac{\log q}{2\pi}
\mathrm{d}x$, where $\mathrm{d}x$ is the Lebesgue measure. To obtain quasicharacters of $G$,
we associate to $\nu\in{\aAG}^*_\C$ the quasicharacter $\chi_\nu(g)=q^{\pair{\nu}{H_G(g)}}=|\nu(g)|_F$, where the second equality defines $H_G\colon G\to\aAG$.

Therefore to compute the integral of a function $f$ on $\mathcal{E}_2(G)$ supported on the 
unramified unitary characters of $A$, we compute
\[
\dIntOver{\mathcal{E}_2(A)}{f(\omega)}{\omega}=
\dInt{f(\chi_\nu)}{\nu}=
\frac{\log q}{2\pi}\int_{\frac{-\pi}{\log q}}^{\frac{\pi}{\log q}}f(e^{it\log q })\,\mathrm{d}t
=
\frac{1}{2\pi i}\dIntOver{\T}{\frac{f(z)}{z}}{z}
\]
if $t\mapsto e^{it\log q}=q^s=:z$ (here $s=it$) parameterizes the unit circle $\T$.
\end{ex}
In general, $\mathcal{E}_2(M_P)$ is a disjoint union of compact tori, and the Plancherel density descends to 
the quotients of these compact tori by certain finite groups, namely, the Weyl groups 
of $(P,A_P)$. The set $\mathcal{E}_2(M_P)^I$ is finite up to twists by unramified characters, by a result of 
Harish-Chandra \cite{Waldspurger}.
\subsection{The Harish-Chandra Schwartz algebra}
\label{subsection the HC Schwartz algebra}
Let $\Cc=\Cc(G)$ be the Harish-Chandra Schwartz algebra of $G$; see \cite{HCH}, \cite{Waldspurger}, or \cite{BK} for the definition and associated notation. In particular, we will record for future comparison with the argument in the proof of Proposition \ref{prop phi1 existence} that
\[
q^{\frac{\ell(w)}{2}}q^{-\#W}\leq\Delta(IwI)\leq q^{\frac{\ell(w)}{2}}.
\]
Let $\Cc^{I\times I}$ the subalgebra of Iwahori-biinvariant functions. 
As explained in \cite{Waldspurger}, the Fourier transform $f\mapsto\pi(f)$ defines an endomorphism
of $\pi$ for every $f\in\Cc$ and every tempered representation $\pi$. 

The Plancherel theorem is the statement (see \cite[Thm. VIII.1.1]{Waldspurger}) that this assignment defines an isomorphism of rings
\[
\Cc\to\mathcal{E}_t(G),
\]
where $\mathcal{E}_t(G)$ is the subring of the endomorphism ring
of the forgetful functor $\mathcal{M}_t(G)\to\Vect$ defined by the 
following conditions:
\begin{enumerate}
\item 
For all $\pi=\Ind_P^G(\nu\otimes\omega)$, the endomorphism
$\eta_\pi=\eta_{\nu,\omega}$ is a smooth function of the unramified unitary character $\nu$
and $\omega\in\Ee_2(M_P)$;
\item
The endomorphism $\eta_{\pi}$ is biinvariant with respect to some open 
compact subgroup of $G$.
\end{enumerate}
We have the obvious inclusion $\iota\colon H\into\Cc^{I\times I}$. 

Define the subring $\Ee(G)$ of the endomorphism ring
of the forgetful functor $\mathcal{M}(G)\to\Vect$ from the category of all smooth representations,
by replacing, in (1) above, unitary characters with 
all unramified characters, and ``smooth" with "algebraic", and adding
\begin{enumerate}
\item[3.] 
The endomorphisms $\eta_\pi$ are compatible with supercuspidal support, \textit{i.e.} if $(M_1, \sigma_1)$ is the supercuspidal support of $\pi$, so that
\[
\pi=i_P^G(\omega\otimes\nu)\into\pi_1=i_{P_1}^G(\sigma_1\otimes\nu_1),
\]
then $\eta_{\pi_1}$ preserves $\pi$ and $\eta_{\pi_1}|_{\pi}=\eta_{\pi}$.
\end{enumerate}
The matrix Paley-Wiener theorem of Bernstein \cite{Bernstein} says that $f\mapsto\pi(f)$ is 
an isomorphism from the full Hecke algebra of $G$ onto $\Ee(G)$. Denote $\Ee^I$ and $\Ee_t$ the 
subrings of $I$-invariant endomorphisms.
\begin{rem}
Property 3 means that unlike Schwartz functions, the Fourier transform a compactly supported function may be freely specified only on inductions of supercuspidal representations. In particular, an algebraic family of endomorphisms defined for all unramified characters may fail to come from a compactly-supported function. An example of a non-compactly supported function with regular Fourier transform is the element of $J$ given in Corollary \ref{cor t1 formula}, because $\pi(t_1)=0$ except if $\pi$ is the Steinberg representation of $G$.
\end{rem}
	
For computational purposes such as ours, we require that these isomorphisms
be explicit. In the Iwahori-spherical case, harmonic analysis 
on $\Cc^{I\times I}$ can be phrased internally to $H$ and various
completions of $H$. In this setting Opdam gave an explicit Plancherel
formula in \cite{OpdamSpectral}. In more general settings there are
explicit formulas for $\GL_n(F)$, $\Sp_4(F)$, and $G_2(F)$, which we will also 
make use of.
\subsection{The algebra $J$ as a subalgebra of the Schwartz algebra.}
\label{subsection algebra J as subalgebra of the Schwartz algebra}
In \cite{BK}, Braverman and Kazhdan constructed a map of $\C$-algebras $J\to \mathcal{C}^{I\times I}$. We 
shall review this construction now. 
\begin{dfn}[\cite{BK}, Section 1.7]
\label{dfn non-strictly-positive}
Let $P=M_PN_P$.
A character $\chi\colon M_P\to\C^\times$ of $M_P$ is \emph{non-strictly positive} if for all root subgroups
$U_\alpha\subset N_P$, we have $|\chi(\alpha^\vee(x))|_\infty\geq 1$ for $|x|_F\geq 1$.

We say a non-strictly positive character $\chi$ is \emph{strictly positive} if for all root 
subgroups $U_\alpha\subset N_P$, we have $|\chi(\alpha^\vee(x))|_\infty >1$ for $|x|_F> 1$.
\end{dfn}
Of course, it suffices to test this for $x=\varpi^{-1}$. 
\begin{ex}
\label{ex nsp character examples}
For $\G=\SL_2$, in the conventions fixed in Example \ref{ex SL2 unit circle conventions}, an unramified character $\chi$ of $A$ is non-strictly positive if it corresponds to $z$ such that $|z|\geq 1$.

If $\G=\GL_n$ and $\nu$ corresponds to the vector $(z_1,\dots, z_n)\in\C^n$, then the condition that $\nu$
is non-strictly positive translates to $|z_1|\geq |z_{2}|\geq\cdots\geq |z_n|$. Such conditions
divide $\C^n$ into chambers, on which the Weyl group $\Sn_n$ clearly acts simply-transitively. Interior 
points correspond to strictly positive $\nu$.
\end{ex}
Following \textit{op. cit.}, let $\mathcal{E}^I_J(G)$
denote the subring of $\mathcal{E}_t(G)$ defined by the following conditions on the endomorphisms
$\eta_{\pi}$:
\begin{enumerate}
\item 
For all $\pi=\Ind_P^G(\nu\otimes\omega)$, the endomorphism $\eta_\pi=\eta_{\nu,\omega}$
is a rational function of $\nu$, regular on the set of non-strictly positive $\nu$.
\item
The endomorphism $\eta_\pi$ is $I\times I$-biinvariant.
\end{enumerate}
%
%
\begin{theorem}[\cite{BK}, Theorem 1.8]
\label{BK nsp theorem}
Let $\G$ be a connected reductive group defined and split over $F$. Then the following statements hold:
\begin{enumerate}
\item 
Let $\pi$ be a tempered representation of $G$. Then the action of $H$ on $\pi^I$ extends uniquely to $J$ via $\phi\circ{}^\dagger(-)$.
\item
Let $P=MN$ be a parabolic subgroup of $G$ with Levi subgroup $M$ and let $\omega$ an irreducible tempered representation of $M$.
Let $\nu$ be a non-strictly positive character of $M$ and let $\pi=\Ind_P^G(\omega\otimes\nu)$.
Then the action of $H(G,I)$ on $\pi^I$ extends uniquely to an action of $J$.
\item
The action of $J$ on the representations $\pi^I$ in (2) extends rationally in $\nu$ to define a homomorphism
%
%
%
\[
\eta\colon J\to\mathcal{E}_J^I(G).
\]
denoted
\[
t_w\mapsto (\eta_\pi(w))_{\pi\in\mathcal{M}_t(G)}.
\]
\end{enumerate}
\end{theorem}
%

The proof of Theorem \ref{BK nsp theorem} (1) and (2) in \cite{BK} uses Theorem \ref{J summary theorem} and \cite{XiJAMS}, to show that, in the notation of Section \ref{subsubsection Involutions on HH},
\[
E(u,s,\rho)|_H\simeq {}^*K(u,s,\rho,q)={}^*i_{P}^G(\sigma\otimes\nu)
\]
where the restriction is via $\phi$, and hence that
\[
E(u,s,\rho)_{H, \phi\circ{}^*(-)}\simeq K(u,s,\rho,q)=i_{P}^G(\sigma\otimes\nu),
\]
where the restriction is now via $\phi\circ {}^*(-)$.

By Lemma \ref{lem based ring auts linearity}, the map $\phi$ is $\PrInv$-linear, where $\PrInv$ is as in Definition \ref{dfn Goldman involution}. Of course, $\phi$ also intertwines conjugation by $T_{w_0}$ and $\phi(T_{w_0})$. 

As noted in Section \ref{subsubsection Involutions on HH}, if $\pi=K(u,s,\rho)$ is simple tempered (or more generally, is any semisimple module), then by \cite[Prop. 6.3]{PrasadDuke},
\begin{equation}
\label{eqn star restriction}
E(u,s,\rho)|_{\phi\circ(-)^\dagger}={}^\dagger{}^*K(u,s,\rho)\simeq\widetilde{K(u,s,\rho)}
\end{equation}
is simple tempered. Moreover, by Lemma \ref{lem Goldman involution exchanges KL bases} (e) and \eqref{eqn star restriction}, we have, if $\Psi_{w_0}(h)=T_{w_0}hT_{w_0}^{-1}$, that
\begin{align*}
{}^{\PrInv\circ\phi(\Psi_{w_0}^{-1})}E(u,s,\rho)|_{\phi\circ{}^\dagger(-)}
&={}^{\phi(\Psi_{w_0}^{-1})}E(u,s,\rho)|_{\PrInv\circ\phi\circ{}^\dagger(-)}
\\
&=E(u,s,\rho)|_{\phi(\Psi_{w_0}^{-1})\circ\PrInv\circ\phi\circ{}^\dagger(-)}
\\
&=E(u,s,\rho)|_{\phi\circ\Psi_{w_0}^{-1}\circ\PrInv\circ{}^\dagger(-)}
\\
&=E(u,s,\rho)|_{\phi\circ{}^*(-)}
\\
&=K(u,s,\rho)
\end{align*}
for any standard $H$-module $K(u,s,\rho)$. Therefore any standard $H$-module extends to a simple $J$-module via $\phi\circ{}^\dagger(-)$.

Composing the morphism $\eta$ with the inverse Fourier transform, Braverman and Kazhdan define an 
algebra map 
\[
\tilde{\phi}\colon J\to\Cc^{I\times I}
\]
sending 
\[
t_w\mapsto (\eta_\pi(w))_{\pi\in\mathcal{M}_t(G)}\mapsto f_w=\sum_{x\in\Waff}A_{x,w}T_x\in\Cc^{I\times I},
\]
where $A_{x,w}=f_w(IxI)$. By definition, $\eta_\pi(w)=\pi(f_w)$ as endomorphisms of $\pi$.
We will show later that $\tilde{\phi}$ is essentially the map $\phi^{-1}$. 
\begin{rem}
There are gaps in the proofs of injectivity and, as pointed out to us by R. Bezrukavnikov
and I. Karpov after an early version of the present paper was completed, of surjectivity of the map 
$\eta$ in \cite{BK}. In the present paper we prove injectivity by proving injectivity of $\tilde{\phi}$ in
Corollary \ref{cor BK map injective}. We prove surjectivity in \cite{rigidDet} for all but a small number of 
cells for exceptional groups; \cite{BKK} proves that $\eta$ is an isomorphism for all two-sided cells.
\end{rem}
Implicit in \cite{BK} is 
\begin{lem}
\label{lemma tilde phi is inverse to phi on H}
We have the commutative diagram \eqref{eqn BK summary diagram}.
%
%
\end{lem}
\begin{proof}
Let $\pi$ be a tempered representation of $G$. By Theorem \ref{BK nsp theorem} and \eqref{eqn star restriction}, the $H$-action on it extends to a 
$J$-action via $\phi\circ^{\dagger}(-)$, such that
\[
\eta_{\pi}(\phi_q({}^\dagger f))=\pi(f)
\]
in $\Ee_t^I$ for any $f\in H$. Therefore $\tilde{\phi}\circ\phi_q({}^\dagger f)=f$ by the Plancherel theorem.
\end{proof}
\subsection{The Plancherel formula for $\GL_n$}
\label{subsection Plancherel formula for GLn}
For $G=\GL_n(F)$, we have access to an explicit Plancherel measure and its 
Bernstein decomposition, thanks to \cite{AubPlym}. 

Recall that for $G=\GL_n(F)$, we have bijections
\begin{align}
\label{levi}
\{\text{partitions}~\text{of}~n\}&\leftrightarrow\{\text{Standard Levi subgroups}~M~\text{of}~\GL_n(F)\}
\\
\label{dist}
&\leftrightarrow\sett{\text{unipotent conjugacy classes in}~\GL_n(\C)}
\\
&\leftrightarrow\mathcal{N}/\GL_n(\C)\nonumber
\\
&
\label{cells}
\leftrightarrow\{2-\text{sided cells}~\cc~\text{in}~\tilde{W}\}
\\
&
\leftrightarrow\{\text{direct summands}~J_\cc~\text{of}~J\}\nonumber
\end{align}
where \eqref{levi} $\leftrightarrow$ \eqref{dist} sends a unipotent conjugacy class $u$ to the standard Levi $M$ such that a member of $u$ is distinguished in $M^\vee$, and \eqref{dist} $\leftrightarrow$ \eqref{cells} is Lusztig's bijection from \cite{affineIV}.
\begin{dfn}
Let $u$ be a unipotent element of a semisimple group $S$ over the complex numbers. Then $u$ 
is \emph{distinguished in} $S$ if $Z_S(u)$ contains no nontrivial torus.
\end{dfn}
Let $P=M_PN_P$ be a parabolic subgroup of $G$ and let $\Oo$ be an orbit in $\mathcal{E}_2(M_P)$ under the action 
of the unitary unramified characters of $M_P$ as explained in Section \ref{subsection HC canonical measure}.
Write $W_{M_P}\subset W$ for the finite Weyl group of $(M_P,A_P)$. Let $\stab{W_M}{\Oo}$ be the 
stabilizer of $\Oo$. Recall that a parabolic subgroup of $G$ is said to be \emph{semistandard} if it
contains $A$. Then the Plancherel decomposition reads
\[
f=\sum_{(P=M_PN_P,\Oo)/\text{association}}f_{M_P,\Oo}
\]
where $f\in\mathcal{C}(G)$, the sum is taken over semistandard parabolic subgroups $P$ up to association, and 
\[
f_{M_P,\Oo}(g)=c(G/M)^{-2}\gamma(G/M)^{-1}\#\stab{W_M}{\Oo}^{-1}\dIntOver{\Oo}{\mu_{G/M}(\omega)d(\omega)\trace{\pi}{R_g(f)}}{\omega},
\]
where $R_gf(x)=f(xg)$ is the right translation of $f$
and $\pi=\Ind_P^G(\nu\otimes\omega)$ is the normalized parabolic induction of the twist of $\omega$
by a unitary unramified character $\nu$ of $M$. In \cite{AubPlym}, each term above
is explicitly calculated as a rational function of $q$. 
\begin{lem}
\label{lem finitely-many discrete series}
There is a finite set $S=\sets{(M,\omega)}{\omega\in\mathcal{E}_2(M)^I}$ such that $\trace{\Ind_P^G(\omega\otimes\nu)}{f}$ is nonzero only for $\omega\in S$, for all $I$-biinvaraiant Schwartz functions $f$.
\end{lem}
\begin{proof}
This is entirely standard.
As $\mathcal{E}_2(M)^{I_M}$ is finite for every $M$, and there are finitely-many standard parabolics
of $G$, we need only show that $\Ind_P^G(\omega\otimes\nu)^I\neq 0$ only if $\omega^{I_M}\neq 0$,
where $I_M$ is the Iwahori subgroup of the reductive group $M$ relative to the Borel subgroup
$M(\F_q)\cap B(\F_q)$ of $M(\F_q)$. Note that $I_M$ is naturally a subgroup of $I$. For any representation 
$\sigma$ of $M$, if $f\in\Ind_P^G(\sigma)$ is $I$-fixed, then for $i_M\in I_M$, we have
\[ 
f(i_M)=\sigma(i_M)\delta^{\frac{1}{2}}_P(i_M)f(1)=f(1).
\]
As $\delta_P=1$ on every compact subgroup of $P$, we have $\delta^{\frac{1}{2}}_P(i_M)=1$, and $f(1)\in\sigma^{I_M}$. 
\end{proof}
Let $\pi=\Ind_P^G(\omega\otimes\nu)$ 
be a tempered representation and let $(u,s)$ be the KL-parameter of its discrete
support. Then by \cite[Theorem 8.3]{KLDeligneLanglands},
$\mathbf{M}_\mathbf{P}^\vee$ is minimal such that $(u,s)\in M_P^\vee$. 
By \textit{op. cit.}, this condition is equivalent to $Z_M(s)$ being semisimple
and $u$ being distinguished in $Z_{M}(s)$.
\begin{prop}[c.f. \cite{OpdamSpectral} Proposition 8.3]
\label{prop plancherel decomposition compatible with cell decomposition}
The Plancherel decomposition is compatible with the decomposition $J=\bigoplus_{\cc} J_{\cc}$
in the sense that if $w\in\cc$, $f=f_w$ and $u=u(\cc)$ under Lusztig's bijection, then $f_M\neq 0$ only for those $M$ such that there exists $s\in M^\vee(\C)$ such that $Z_{M^\vee}(s)$ is semisimple
and $Z_{M^\vee}(s)\cap Z_{M^\vee}(u)$ contains no nontrivial torus.
\end{prop}
\begin{proof}
Let $\pi :=\Ind_P^G(\nu\otimes\omega)$ be a tempered irreducible representation of $G$ induced as usual from a 
standard parabolic subgroup $P$ with Levi subgroup $M$. Then $\pi^I$ is a tempered
irreducible $H$-module, and is of the form $K(u,s,\rho,q)$ for $u,s\in\G^\vee(\C)$ and $\rho$ 
a representation of $\pi_0(Z_{G^\vee}(u,s))$ with $s$ compact. By Theorem \ref{BK nsp theorem}, 
$K(u,s,\rho,q)$ extends to a $J$-module.
By definition of Lusztig's bijection, 
$\pi(f_w)\neq 0$ only if $w$ is in the two-sided cell $\cc=\cc(u)$ of 
$\Waff$ corresponding to $u$. On the other hand, 
$K(u,s,\rho,q)$ is induced from a square-integrable standard module $K_M(u,s,\tilde{\rho},q)$ of 
$H(M,I_M)$. But now \cite[Theorem 8.3]{KLDeligneLanglands} says that $(u,s)$ must be precisely
as in the statement of the proposition.
%
Thus only such summands $(f_w)_{M}$ in the Plancherel decomposition of $f_w$ are nonzero.
\end{proof}
Of course, when $\G=\GL_n$, the bijections \eqref{levi}, \eqref{dist}, and \eqref{cells}, imply that 
there is a unique nonzero summand $(f_w)_{M}$ for each $w$.
%
%
\subsection{Plancherel measure for $\GL_n$}
\label{subsection Plancherel measure for GLn}
We refer to \cite[Section 5]{AubPlym}, for a summary of the Bernstein decomposition of the tempered
irreducible representations of $\GL_n$, in particular we use the description in \textit{loc. cit.}
of the Bernstein component parameterizing the $I$-spherical representations of $G$.

As we shall be applying the Plancherel formula only
to Iwahori-biinvariant functions, it suffices to consider only irreducible tempered representations
with Iwahori-fixed vectors. For $\GL_n$, the only such representations are of the form
\[
\pi=\Ind_P^G(\nu_1\St_1\boxtimes\cdots\boxtimes\nu_k\St_k)
\]
where $\St_i$ is the Steinberg representation of $\GL_i$, and 
$P\supset M=\GL_{l_1}\times\cdots\times\GL_{l_k}$.

These representations are parameterized as follows. Let $M$ be a Levi subgroup corresponding to the partition $l_1+\cdots+l_k=n$, and recall that we write $\T$ for the circle group. Define $\gamma\in\Sn_n$ by $\gamma=(1\dots l_1)(1\dots l_2)\cdots (1\dots l_k)$, so 
that the fixed-point set $(\T^n)^\gamma=\{(z_1,\dots, z_1,\dots,z_k,\dots, z_k)\}\simeq\T^k$.
Then the irreducible tempered representations with Iwahori-fixed vectors induced from $M$
are parameterized by the compact orbifold $(\T^n)^\gamma/Z_{\Sn_n}(\gamma)$.
\begin{theorem}[\cite{AubPlym}, Remark 5.6]
\label{theorem AubPly Plancherel measure form}
Let $G=\GL_n$ and $M=\GL_{l_1}\times\cdots\times\GL_{l_k}$ be a Levi subgroup. Then
the Plancherel measure of $H$ on $(\T^n)^\gamma/Z_{\Sn_n}(\gamma)$ is
\[
\mathrm{d}\nu_H(\omega)=
\prod_{i=1}^{k}\frac{q^{l_i^2-l_i}(q-1)^{l_i}}{l_i(q^{l_i}-1)}\cdot q^{\frac{n-n^2}{2}}\cdot\prod_{(i,j,g)}\frac{q^{2g+1}(z_i-z_jq^{g})(z_i-q^{-g}z_j)}{(z_i-z_jq^{-g-1})(z_i-q^{g+1}z_j)},
\]
where the tuples $(i,j,g)\in\Z\times\Z\times\frac{1}{2}\Z$ are tuples such that $1\leq i<j\leq k$ and 
$|g_i-g_j|\leq g\leq g_i+g_j$, where $g_i=\frac{l_i-1}{2}$.
\end{theorem}
This is the measure that we will integrate against, by successively applying the residue theorem.
When carrying out explicit calculations, we will usually elide the constant
\[
\prod_{i=1}^{k}\frac{q^{l_i^2-l_i}(q-1)^{l_i}}{l_i(q^{l_i}-1)}\cdot q^{\frac{n-n^2}{2}}
\]
as it depends only on $M$. We shall abbreviate 
\[
\Gamma_{i,j,g}:=
q^{-2g-1}\left|\Gamma_F\left(q^{-g}\frac{z_i}{z_j}\right)\right|^2=\frac{(z_i-q^{g}z_j)(z_i-q^{-g}z_j)}{(z_i-z_jq^{-g-1})(z_i-q^{g+1}z_j)},
\]
and recall that, as noted in the proof of Theorem 5.1 in \cite{AubPlym}, the function $(z_1,\dots,z_k)\mapsto\prod_{(i,j,g)}\Gamma_{i,j,g}$ is $Z_{\Sn_n}(\gamma)$-invariant. Hence for 
the purposes of integration, we may allow ourselves to integrate simply over $\T^k$. Moreover,
there are many cancellations between the $\Gamma_{i,j,g}$ for a fixed pair $i<j$ as $g$ varies.
Indeed, putting $q_{ij}=q^{|g_i-g_j|}$, $q^{ij}=q^{ g_i+g_j+1}$, and
\[
\Gamma^{ij}:=\frac{(z_i-q_{ij}z_j)(z_i-(q_{ij})^{-1}z_j)}{(z_i-q^{ij}z_j)(z_i-(q^{ij})^{-1}z_j)}
\]
we have
\[
\prod\Gamma_{i,j,g}=\Gamma^{ij},
\]
where the product is taken over all integers $g$ appearing in triples $(i,j,g)$ for $i<j$ fixed. We
set
\[
c_M:=\prod_{(i,j,g)}q^{2g+1}\cdot\prod_{i=1}^{k}\frac{q^{l_i^2-l_i}(q-1)^{l_i}}{l_i(q^{l_i}-1)}\cdot q^{\frac{n-n^2}{2}},
\]
where the first product is taken over $(i,j,g)$ such that $1\leq i<j\leq k$ and $|g_i-g_j|\leq g\leq 
g_i+g_j$.
\subsection{Beyond type $A$: the Plancherel formula following Opdam}
\label{subsection plancherel following Opdam}
Beyond type $A$, we still have available Opdam's explicit Plancherel formula for the Iwahori-Hecke algebra
\cite{OpdamSpectral}. 
Let $\G$ be a connected reductive algebraic group defined and split over $F$, of Dynkin type other than type $A$ (to avoid redundancy). Let $f$ be an Iwahori-biinvariant Schwartz function on $G$ and let $M$
be a Levi subgroup of $G$. Given a parabolic subgroup $\PP$ of $\G$, let $R_{1,+}$ and $R_{P,1,+}$ be defined as in \cite{OpdamSpectral}, Section 2.3. Recall that the group of unramified characters of a Levi 
subgroup $M$ has the structure of a complex torus, and is in fact a maximal torus of $\mathbf{M}^\vee(\C)$. 
In particular, if $P$ is a parabolic subgroup and $\alpha$ is a root of $(M_P,A_P)$, then it makes
sense to write $\alpha(\nu)$ for any unramified character $\nu$ of $M$. Then, altering Opdam's notation to match our own from Section \ref{subsection Plancherel formula for GLn}, the Plancherel formula reads
\begin{theorem}[\cite{OpdamSpectral} Thm. 4.43]
\begin{equation}
\label{Opdam integral general}
f_{M,\Oo}(1)=\frac{q^{-\ell(w^P)}}{\# \stab{W_M}{\Oo}}\int_{\Oo}d(\omega)\prod_{\alpha^\vee\in R_{1,+}\setminus R_{P, 1,+}}\frac{|1-\alpha^\vee(\nu)|^2}{|1+q^\frac{1}{2}_{\alpha}\alpha^\vee(\nu)^{1/2}|^2|1-q^\frac{1}{2}_{\alpha}q_{2\alpha}\alpha^\vee(\nu)^{1/2}|^2}\trace{\pi}{f}\,\mathrm{d}\omega,
\end{equation}
where $\pi=\Ind_P^G(\omega\otimes \nu)$, $P\supset M$, and where $q_\alpha$ and $q_{2\alpha}$ are powers of 
$q$, and $w^P$ is the longest element in the complement $W^P$ to the parabolic subgroup $W_P$ of $W$.
\end{theorem}
Note that whenever $q_{2\alpha}=1$, which holds whenever $\alpha^\vee\not\in 2X_*$, the factor for $\alpha$ 
reduces to 
\begin{equation}
\label{eqn opdam factor simplification}
\frac{|1-\alpha^\vee(\nu)|^2}{|1-q_\alpha\alpha^\vee(\nu)|^2}.
\end{equation}
In types $A$ (as we have used above) and $D$, this simplification always occurs. In types $B$ and $C$,
it happens for all roots except $\alpha=2\varepsilon_i\in R_{1,+}(B_n)$ and $\alpha=4\varepsilon_i\in R_{1,+}(C_n)$, where $\varepsilon_i$ is the character $\diag(a,\dots, a_n)\mapsto a_i$.

For explicit evaluation we rewrite \eqref{Opdam integral general} in coordinates as follows. Recalling the 
setup of Section \ref{subsubsection Unramified characters}, a chose basis $\{\beta_i^\vee\}$ of the coweight lattice 
of $\G$. Then we obtain coordinates $z_i$, such that if $\alpha^\vee=\sum_{i}e_i\beta_i^\vee$, then the 
factor in \eqref{Opdam integral general} labelled by $\alpha^\vee$ is 
\begin{equation}
\label{eqn Opdam factor in coordinates}
\frac{|1-z_1^{e_1}\cdots z_n^{e_n}|^2}{|1+q^\frac{1}{2}_{\alpha}(z_1^{e_1}\cdots z_n^{e_n})^{1/2}|^2|1-q^\frac{1}{2}_{\alpha}q_{2\alpha}(z_1^{e_1}\cdots z_n^{e_n})^{1/2}|^2}.
\end{equation}
When integrating, the coordinates
$z_i$ are restricted to the residual coset corresponding to $\Oo$, in the sense of \cite{OpdamSpectral}. 
For $\G=\GL_n$, we used the basis afforded by the characters $\varepsilon_i$.
\subsection{Regularity of the trace}
\label{section reg trace}
In order to extract information about the expansion of the elements $t_w$ in terms of the 
basis $T_x$ via the Plancherel formula, we must establish a regularity property
of $\trace{\pi}{f_w}$ where $\pi$ is an irreducible tempered representation of $G$.
The needed property follows trivially from Theorem \ref{BK nsp theorem}.
\subsubsection{Intertwining operators}
\label{subsection intertwining operators}
The goal is to use the property that elements of $\mathcal{E}^I_J(G)$ commute with all intertwining operators 
in $\mathcal{M}_t(G)$, and regularity of the trace for unitary and non-strictly positive characters of $M$ to 
deduce regularity of the trace at all characters of $M$.

Let $\omega$ be a discrete series representation of a Levi subgroup $M$ of $G$, and let 
$\nu$ be any unramified character of $M$, not necessarily unitary. Then we may form the representation
$\pi=\Ind_P^B(\nu\otimes\omega)$ of $G$, where $P$ is a parabolic with Levi factor $M$.
We will now recall some well-known facts about the action of the Weyl group of $M$ on
such representations $\pi$. Let $\theta$ and $\theta'$ be two subsets of $\Delta$ corresponding to 
Levi subgroups $M$ and $M'$. Let $w\in W$ be such that $w\theta=\theta'$. Then there is 
an intertwining operator 
\[
J_{P|P'}(\omega, \nu)\colon\Ind_P^B(\nu\otimes\omega)\to \Ind_{P'}^B(\nu\otimes\omega)
\]
for each $w\in W(\theta,\theta')=\sets{w\in W}{w(\theta)=\theta'}$. For $\G=\GL_n$, this set is nonempty only 
if
\[
M_\theta=\GL_{l_1}\times\cdots\times\GL_{l_N}~~M_{\theta'}=\GL_{l'_1}\times\cdots\times\GL_{l'_N}
\]
and $\{l_1,\dots, l_N\}=\{l_1'\,\dots, l'_N\}$ are equal multisets. 
In this case , $W(\theta,\theta')\simeq\Sn_N$ can be viewed as acting by permuting the blocks of $M$. It is well-known that $J_{P|P'}(\omega, \nu)$
is a meromorphic function of $\nu$ with simple poles. The poles of these operators have been studied by Shahidi in \cite{Shahidi}, and in the language of modules over the full Hecke algebra by Arthur in \cite{Art}. 
The results will be stated for certain renormalizations $A(\nu, \omega, w)$ of the operators $J_{P|P'}(\omega, \nu)$ as explained in equation (2.2.1) of \cite{Shahidi}.

\subsubsection{Conventions on parabolic subgroups}

We now recall the notation of Shahidi \cite{Shahidi}. Given a subset $\theta$, we set 
$\Sigma_\theta=\spn_\R\theta$, and $\Sigma_\theta^+=\Psi^+\cap\Sigma_\theta$ 
and likewise define $\Sigma_\theta^-$. Let $\Sigma(\theta)$ be the roots of 
$(P,A_P)$. Define the positive roots $\Sigma^+(\theta)$ to be the roots obtained 
by restriction of an element of $\Psi^+\setminus\Sigma_\theta^-$. 

Given two subsets $\theta,\theta'\subset\Delta$, following \cite{Shahidi} we set 
\[ 
W(\theta,\theta')=\sets{w\in W}{w(\theta)=\theta'},
\]
and then for $w\in W(\theta,\theta')$, we define
\[
\Sigma(\theta,\theta',w)=\sets{[\beta]\in\Sigma^+(\theta)}{\beta\in\Psi^+-\Sigma_\theta^+~\text{and}~w(\beta)\in\Psi^-},
\]
and then
\[
\Sigma^\circ(\theta,\theta',w):=\sets{[\beta]\in \Sigma(\theta,\theta',w)}{w_{[\beta]}\in W(A_P)}.
\]
\begin{rem}
In \cite{Shahidi} additional care about the relative case is taken in the notation. In our simple case 
this is of course unnecessary, and we omit it.
\end{rem}

In the case $\G=\GL_n$, this specializes as follows.
Let $\alpha_{ij}\colon\diag(t_i)\mapsto t_it_j^{-1}$ be characters of $T$. The $\alpha_{ij}$ for $j>i$ are 
the positive roots of $(\B,\Aa)$ and $\beta_i:=\alpha_{ii+1}$ are our chosen simple roots. 
If $P$ corresponds to the partition $n_1+\cdots +n_p$ of $n$ and subset $\theta\subset\Delta$, then 
\[
\Sigma(\theta)=\spn{\{\beta_{n_1},\beta_{n_1+n_2},\dots\}},
\]
where we view $\beta_i$ as restricted to $\aA_P\into\aA_G$. Note that all the positive roots $\Sigma^+(\theta)$ of $(P_\theta, N_\theta)$ are in $N_P$. Denoting by $[\alpha]$ the coset representing a root $\alpha$ of $G$ restricted to $P$,
the positive roots in $N_P$ are the $\alpha_{ij}$ such that $[\alpha_{ij}]=[\beta_{n_1+\cdots+n_k}]$.
\begin{ex}
\label{shahidi notation example}
If $\G=\GL_6$ and $P$ be the parabolic subgroup of block upper-triangular matrices
%
%
corresponding to the partition $6=2+2+1+1$ and $\theta=\{\alpha_{12}, \alpha_{34}\}$, then
we have 
\[
A_P=\sett{\diag(t_1,t_1,t_2,t_2,t_3,t_4)}. 
\]
The positive simple roots are $\Sigma^+(\theta)=\{[\beta_2],
[\beta_4],[\beta_5]\}$. The Weyl group $W(A_P)\simeq\Sn_2\times\Sn_2$ acts by permuting the blocks.
Note that the simple reflection sending $\beta_4\mapsto -\beta_4$ does not arise by permuting the 
blocks (\textit{i.e.} $\begin{pmatrix}
4 & 5
\end{pmatrix}$ does not send blocks to blocks) , hence $w_{[\beta_4]}\not\in W(A_P)$. Hence for any $w\in W(\theta,\theta')$ we have
$\Sigma^\circ(\theta,\theta',w)\subseteq\sett{[\beta_2],[\beta_5]}$.
\end{ex}
\subsubsection{Regularity of the trace}
\label{subsection regularity of trace}
We have the following information about the poles of intertwining operators, due, according to \cite{Shahidi}, 
to Harish-Chandra:
\begin{theorem}[\cite{Shahidi}, Theorem 2.2.1]
\label{thm Shahidi entire}
Let $\omega$ be an irreducible unitary representation of $M$. Say that $\omega$
is a subrepresentation of $\Ind_{P_*}^M(\omega_*)$ for a parabolic subgroup $P_*=M_*N_*$
and $\omega_*$ is an irreducible supercuspidal representation of $M_*$. Let $\theta_*\subset\Delta$
be such that $P_*=P_{\theta_*}$ as parabolic subgroups of $M_*$. 

Then the operator
\[
\prod_{\alpha\in\Sigma^0_r(\theta_*,w\theta_*,w)}(1-\chi_{\omega,\nu}^2(h_\alpha))A(\nu,\omega,w)
\]
is holomorphic on $\mathfrak{a}_{\theta\C^*}$. Here $\chi_{\omega,\nu}$ is the central character of the twisted representation 
$\omega\otimes q^{\pair{\nu}{H_\theta(-)}}$.
\end{theorem}
In particular for the purposes of the Plancherel formula, the 
only relevant $\omega$ are unitary, hence have unitary central characters. Therefore $A(\nu,\omega, w)$ is 
holomorphic at $\nu$ if $|q^{\pair{\nu}{H_\theta(-)}}|\neq 1$, or equivalently if 
\[
\Re(\pair{\nu}{H_\theta(-)})\neq 0.
\]
In particular, there is a finite union of hyperplanes away from which each operator
$A(\nu,\pi,w)$ is holomorphic, for any $w\in W$.
\begin{lem}
\label{Lemma trace is regular}
Let $M$ be a Levi subgroup of $G=\G(F)$ and let $\omega$ be a discrete series representation of $M$.
Let $P=M_PA_PN_P$ be the standard parabolic subgroup containing $M=M_P$ and let $k=\rk A_P$. 
Let $z_1,\dots, z_k\in(\C^\times)^k= X^*(A_P)\otimes_\Z\C$ define an unramified quasicharacter 
$\nu=\nu(z_1,\dots, z_k)$ of $A_P$ as in Section \ref{subsection HC canonical measure}. Let 
$\pi=\Ind_P^G(\omega\otimes\nu)$. Let $f\in J$. Then
\[
\trace{\pi}{f}\in\C[z_1,\dots,z_k,z_1^{-1},\dots, z_k^{-1}].
\]
That is, the trace is a regular function on $(\C^\times)^k$.
\end{lem}
\begin{proof}
We know \textit{a priori} that
\[
\trace{\pi}{f}\in\C(z_1,\dots, z_k)
\]
is a rational function of $\nu$, as the operator $\pi(f)$ itself depends rationally on the variables $z_i$
by Theorem \ref{BK nsp theorem}. Therefore
\[
\trace{\pi}{f}=\frac{p(\nu)}{h(\nu)}\in\C(z_1,\dots,z_k).
\]
By Theorem \ref{thm Shahidi entire} and the discussion following it, 
there is an open subset $U$ of the unitary characters, such that
for all $\nu\in U$, and all $w\in W$ we have
%
\begin{equation}
\label{eq trace W-equivariant rational}
\frac{p(\nu)}{h(\nu)}=\frac{p(w(\nu))}{h(w(\nu))}.
\end{equation}
Holding $z_2,\dots, z_n$  constant and in $U$, \eqref{eq trace W-equivariant rational} becomes an equality of meromorphic functions
of $z_1$ that holds on a set with an accumulation point, and hence
\eqref{eq trace W-equivariant rational}
holds for all $z_1\in\C^\times$. Now holding $z_1\in\C^\times$
constant and arbitrary, and $z_3,\dots, z_k$ constant and in $U$,
we see that \eqref{eq trace W-equivariant rational} holds also
for all $z_2\in\C^\times$.
Therefore \eqref{eq trace W-equivariant rational} actually holds
for all $\nu$, \textit{i.e.} $\trace{\pi}{f}$ is a $W$-invariant rational function of $\nu$.

When $\nu$ is non-strictly positive with respect to $M$, by Theorem \ref{BK nsp theorem}, $\trace{\pi}{f}$ has poles only of the form $z_i^{n_i}=0$.
The claim now follows from the $W$-invariance, and thus $\trace{\pi}{f}$ is a regular function 
on $(\C^\times)^k$.
\end{proof}
%
%
%
We will therefore allow ourselves to write $\trace{\pi}{f}$ for functions $f\in J$ even for $\nu$ such that
the operator $\pi(f)$ itself is not defined.
\begin{lem}
\label{lem trace td is constant}
Let $d\in\Waff$ be a distinguished involution and $\pi$ be a tempered representation induced from one of the Levi subgroups attached to the two-sided cell containing $d$ in the sense of Proposition \ref{prop plancherel decomposition compatible with cell decomposition}. Then $\trace{\pi}{f_d}$ is constant and a natural number. In fact, the same is true for any idempotent in $J$.
\end{lem}
\begin{proof}
Let $j$ be an idempotent in $J$.
We have $\trace{\pi}{f_j}=\rank(\pi(f_j))$. The trace is continuous in $\nu$ by Lemma 
\ref{Lemma trace is regular}, and the present lemma follows as $\T$ is connected.
%
%
\end{proof}
%
%
\section{Proof of Theorem  \ref{thm general G denominators} for general $\G$, and the case of $\GL_n$}
\label{section proofs}
\subsection{The functions $f_w$ for $\GL_n$}
\label{subsection the functions fw for GLn}
In this section, $q>1$.

To compute with the Plancherel formula, 
we will need to apply the residue theorem successively in each variable $z_i$, and in doing so will
we will need to sum over a certain tree that will track, for each variable, at which residues we evaluated.
Upon integrating with respect to each variable $z_i$, we will have poles of the form $z_i=0$
or $z_i=(q^{ij})^{-1}z_j$. For example, if we have $4$ variables $z_1,z_2,z_3,z_4$ 
corresponding to a Levi subgroup $\GL_{l_1}\times\GL_{l_2}\times\GL_{l_3}\times\GL_{l_4}$,
then some of the summands obtained by successively applying the residue theorem are
labelled by paths on the tree
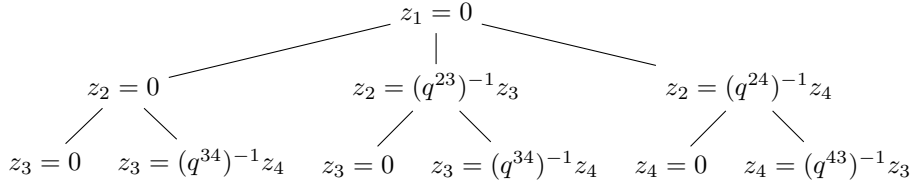
\begin{center}
\begin{forest}
  [{$z_1=0$} 
    	[{$z_2=0$} 
  		[{$z_3=0$}] [{$z_3=(q^{34})^{-1}z_4$}]
  	  	] 
  	  	[{$z_2=(q^{23})^{-1}z_3$} 
  		[{$z_3=0$}] [{$z_3=(q^{34})^{-1}z_4$}]
  	  	] 
  	[{$z_2=(q^{24})^{-1}z_4$}
  	[{$z_4=0$}] [{$z_4=(q^{43})^{-1}z_3$}]
  	]  
  ]
\end{forest}
\end{center}
Of course, to evaluate the entire integral for $M$, we would also need to consider trees whose
roots are decorated with $z_1=(q^{12})^{-1}z_2$, and so on, for  a total of four trees.

\begin{dfn}
Given a Levi subgroup $M$ with $N+1$ blocks, a \emph{bookkeeping tree} $T$ for $M$
is a rooted tree with $N$ levels such that the vertices on the $i$-th level below the root
each have $N-i$ child vertices, and each vertex is decorated with an equation of the form
$z_i=0$ or $z_i=(q^{ij})^{-1}z_j$, where the index $j$ does not appear along the path
from the vertex to the root, and the parent root is decorated with an equation
$z_{k}=(q^{ki})^{-1}z_i$ for some $k$. Moreover, we require that the root be decorated with an equation of the form $z_i=0$ or $z_i=(q^{ij})^{-1}z_j$ for $i$ minimal. A \emph{branch} of $T$ is a simple path in $T$
from the root to one of the leaves.
\end{dfn}
\begin{dfn}
Given a branch $B$ of a bookkeeping tree, a \emph{clump in} $B$ is an ordered subset of indices $i$ appearing in the decorations of
successive parent-child vertices, such that no decoration of the form $z_i=0$
occurs along the path from the closest index to the root to the farthest index from the root.
We write $C\prec B$ if $C$ is a clump of $B$. 
\end{dfn}
\begin{ex}
The sets of indices $\{3,4\}$, $\{2,3,4\}$, $\{2,3\}$, $\{2,4\}$, and $\{2,4,3\}$
(note the ordering) are all the clumps of the above tree. The sets $\{1,2,3\}$ and
$\{1,2,4\}$ are not clumps.
\end{ex}
%
\begin{theorem}
\label{theorem fd(1) integral}
Let $G=\GL_n(F)$ and let $d$ be a distinguished involution such that the two-sided cell 
containing $d$ corresponds to the Levi subgroup $M$. Let $N+1$ be the number of blocks in $M$ such that the $i$-th block has size $l_i$. Let $m_{j}$ be the number $l_i$ that are equal to $j$. Define for $r\leq k$
\[
Q_{rk}=q^{i_{k}i_{k+1}}q^{i_{k-1}i_k}\cdots q^{i_ri_{r+1}}
\]
in the notation of \ref{subsection Plancherel measure for GLn}.
Then if $f_d$ is as in Sections \ref{subsection algebra J as subalgebra of the Schwartz algebra},
%
\begin{multline}
\label{fd(1) theorem post-cancellation}
f_d(1)=\frac{\rank(\pi(f_d))}{m_1!\cdots m_{n}!}c_M\sum_{\mathrm{trees}~T}\sum_{\mathrm{branches}~B~\mathrm{of}~T}\prod_{\substack{C\prec B \\ C=\{i_0,\dots, i_t\}}}\frac{(1-q^{l_{i_0}})(1-q^{l_{i_1}})}{1-q^{l_{i_0}+l_{i_1}}}
\\
\cdot\prod_{k=1}^{t-1}\frac{(1-q^{l_{i_{k+1}}})}{(1-q^{l_{i_k}+l_{i_{k+1}}})}\prod_{r=0}^{k-1}\frac{R_{rk}}{1-Q_{rk}q^{i_ri_{k+1}}},
\end{multline}
%
where
\[
R_{rk}=\begin{cases}
1-Q_{rk}q^{g_{i_r}-g_{i_k}}&\text{if}~k<t-1\\
(1-Q_{r,t-1}q^{g_{i_r}-g_{i_t-1}})(1-Q_{r,t-1}q^{g_{i_k}-g_{i_r}})&\text{if}~k=t-1
\end{cases}.
\]
\end{theorem}
\begin{cor}
\label{cor denominators of fd(1) poincare}
The denominator of $f_d(1)$ divides a power of the Poincar\'{e} polynomial $P_{\G/\B}(q)=P_{\Sn_{n}}(q)$ of $\G$.
Moreover, when $d$ is in lowest two-sided cell, corresponding to $M=T$, $f_d(1)=\rank(\pi(f_d))/P_{\G/\B}(q)$ exactly.
\end{cor}
%

%
%
%
\begin{proof}[Proof of Corollary \ref{cor denominators of fd(1) poincare}]
We will show that each of the three forms of denominator that appear in the conclusion of 
Theorem \ref{theorem fd(1) integral} divide $P_{\G/\B}(q)$, and thus that their product divides a power of $P_{\G/\B}(q)$.
The denominators of $c_M$ are all of the form 
$1+q+\cdots+q^{l_{i}-1}$, and so divide $P_{\G/\B}(q)$ as $l_i\leq n$ for all $i$.
Note that as $l_{i_i}+l_{i_{k+1}}\leq n$, the leftmost denominators in \eqref{fd(1) theorem post-cancellation} satisfy the conclusion of the corollary also. Finally,  $Q_{rk}q^{i_ri_{k+1}}=q^{l_{i_{k+1}}+l_{i_k}+\cdots+ l_{i_r}}$, and $R_{rk}$ is likewise always a polynomial in $q$ (as opposed to $q^{-1}$) divisible by $1-q$. Again using that $l_{i_{k+1}}+l_{i_k}+\cdots+ l_{i_r}\leq n$, we are done with the first statement.

We now take up the second statement, the proof\footnote{We thank A. Braverman for explaining to us the observation whose obvious generalization we present below.}
of which does
not require any computations at all and which holds for any $w\in\cc_0$.
Recalling that the only $K$-spherical tempered 
representations of $G$ are principal series representations \cite{MacdonaldZonal}, 
let $\mu_K$ be the Haar
measure on $G$ such that $\mu_K(K)=1$, $\mu_I$ be the Haar measure such that $\mu_I(I)=1$,
and denote temporarily $\mathrm{d}\pi_{\mu}$ the 
Plancherel measure normalized according to a chosen Haar measure $\mu$ on $G$.
Likewise write $\pi_K(f)$ for the Fourier transform with respect to $\mu_K$ and $\pi_I(f)$ for
the Fourier transform with respect to $\mu_I$.

As we have
\[
\mu_I=P_{\G/\B}(q)\mu_K,
\]
and formal degrees scale inversely to the Haar measure, we have
\begin{equation}
\label{eqn Plancherel renormalization lowest cell}
d\pi_{\mu_I}=\frac{1}{P_{\G/\B}(q)}d\pi_{\mu_K}
\end{equation}
for tempered principal series.
Now let $w\in\cc_0$ and consider $f_w=\tilde{\phi}(t_w)$. By Lemma \ref{Lemma trace is regular} and the Satake isomorphism, there is a function
$h_w$ in the spherical Hecke algebra such that for all principal series $\pi$,
\[
\mathrm{trace}\left(\pi_I(f_w)\right)=\mathrm{trace}\left(\pi_K(h_w)\right)
\]
as regular Weyl-invariant functions of the Satake parameter. By the Plancherel formula and 
\eqref{eqn Plancherel renormalization lowest cell},
\begin{equation}
\label{eq d lowest cell trace via Haar renorm}
f_w(1)=\dInt{\mathrm{trace}\left(\pi_I(f_w)\right)}{\pi_{\mu_I}}=\dInt{\mathrm{trace}\left(\pi_K(h_w)\right)}{\pi_{\mu_I}}=\frac{1}{P_{\G/\B}(q)}\dInt{\mathrm{trace}\left(\pi_K(h_w)\right)}{\pi_{\mu_K}}
=\frac{h_w(1)}{P_{\G/\B}(q)},
\end{equation}
In particular, if $w=d$ is a distinguished involution in the lowest two-sided cell, then 
\[
\mathrm{trace}\left(\pi_I(f_d)\right)=\rank\left(\pi_I(f_d)\right)=\rank\left(\pi_I(f_d)\right)\cdot\mathrm{trace}\left(\pi_K(1_K)\right),
\]
and \eqref{eq d lowest cell trace via Haar renorm} becomes
\begin{equation}
\label{f_d(1) big cell trick}
f_d(1)=\frac{\rank(\pi(f_d))}{P_{\G/\B}(q)}=\frac{1}{P_{\G/\B}(q)}.
\end{equation}
Here the rank is given by \cite{XiI}, Proposition 5.5. (In \textit{op. cit.} there is the assumption of simple-connectedness, but it 
is easy to see that the distinguished involutions for the extended affine Weyl group $\tilde{W}(\tilde{\G}(F))$ of the universal cover $\tilde{\G}$ are distinguished involutions for $\Waff$ using the definition in \cite{affineII} and uniqueness of the $\{C_w\}$-basis, and that the lowest cell is just the lowest cell
of $\tilde{W}(\tilde{\G}(F))$ intersected with $\Waff$.)
\end{proof}
Once we have established injectivity of $\tilde{\phi}$, we will show by a counting argument that 
$\rank\left(\pi(t_d)\right)=1$ for any distinguished involution $d$, in the 
case $G=\GL_n$, see Theorem \ref{thm conjecture is true GLn}.
Note that \eqref{f_d(1) big cell trick} is an example of the behaviour conjectured in Remark \ref{rem more precise version}.
\begin{rem}
\label{rem interesting to find functions with constant trace}
It would be interesting to find $I$-biinvariant Schwartz functions $h$ playing the role of $h_w$ for the other two-sided cells, namely such that $\trace{\pi}{h}$ was regular, nonzero only for a single pair 
$(M,\omega)$, and the value $h(1)$ was known as a function of $q^{1/2}$.
\end{rem}
\begin{rem}
If $\mathcal{P}$ is a maximal parahoric subgroup of $G$, then the longest word $w_{\mathcal{P}}$ is
a distinguished involution in the lowest two-sided cell. As above, and as we will explain again in Section
\ref{section parahoric-fixed}, $\pi(f_{w_{\mathcal{P}}})$ is nonzero only for the principal series
representations, and its image is contained in $\pi^\mathcal{P}$. Thus after we will have shown
injectivity, \cite{KeysClassI} gives another proof that $\rank(\pi(f_{w_{\mathcal{P}}}))=1$,
and more generally that $\pi(t_d)$ has rank 1 on the principal series for any distinguished
involution $d\in\cc_0$.
\end{rem}
\begin{lem}
\label{lemma sum of z_i powers must be -N}
Let $e_0,\dots, e_n\in\Z$. Then
\begin{equation}
\label{eq integral with z_ie_i powers}
\int_{\T}\cdots\int_{\T}z_0^{e_0}\cdots z_N^{e_N}\prod_{i<j}\Gamma^{ij}\mathrm{d}z_0\cdots\mathrm{d}z_N=0
\end{equation}
unless $e_0+\cdots +e_N=-N$. 
\end{lem}
Although we do not use the Lemma in the sequel, we include it because it illustrates an efficient way to 
compute the functions $f_w$ in practice, as we explain after its proof.
%
%
\begin{cor}
\label{corollary fw(1) denominator divides poincare}
Let $w\in\Waff$. Then $f_w(1)$ is a rational function of $q$ with denominator dividing a power of the 
Poincar\'{e} polynomial of $\G$. The numerator is a Laurent polynomial $p_1(q^{1/2})+p_2(q^{-1/2})$ in 
$q^{1/2}$, where the degree of $p_1$ is bounded uniformly in terms of $\Waff$. The 
denominator of $f_w(1)$ depends only on the two-sided cell containing $w$.
\end{cor}
\begin{rem}
In light of Lemma \ref{Lemma trace is regular}, Lemma \ref{lemma sum of z_i powers must be -N} and 
Corollary \ref{corollary fw(1) denominator divides poincare} have the following interpretation. 
Let $\Gamma$ be a left cell in a two-sided cell $\cc$.
Then in \cite{XiAMemoir}, Xi shows that all the rings $J_{\Gamma\cap\Gamma^{-1}}$ are isomorphic to the representation ring of the associated Levi subgroup $M_\cc$, and $J_\cc$ is a matrix algebra over $J_{\Gamma\cap\Gamma^{-1}}$. 
Therefore $w\in \Gamma\cap\Gamma^{-1}$ are labelled by dominant weights of $M_\cc$, and
if $t_\lambda$ is such an element, Xi's results show that if $\bq=q$ and $\pi=\Ind_B^G(\nu)$
is an irreducible representation of $J_\cc$, then
\[
\trace{\pi}{t_\lambda}=\trace{V(\lambda)}{\nu},
\]
where we view $\nu$ as a semisimple conjugacy class in $M_\cc$, and $V(\lambda)$ is the irreducible representation of $M_\cc$ of highest weight $\lambda$. Then we have
that $f_\lambda(1)\neq 0$ only if $\lambda$ is of height $0$ with respect to the 
basis $\varepsilon_i\colon\diag(a_1,\dots, a_n)\mapsto a_i$.
\end{rem}
The proofs of Lemma \ref{lemma sum of z_i powers must be -N} and 
Corollary \ref{corollary fw(1) denominator divides poincare}
will use the notation of the proof of Theorem \ref{theorem fd(1) integral}, and we defer them until after the 
proof of the theorem.
\subsubsection{Example computations and a less singular cell}
In this section we provide two example computations to elucidate the coming proof of Theorem
\ref{theorem fd(1) integral}.
\begin{ex}
\label{example GL2 Plancherel integral}
Let $G=\GL_2$ and $M=A$. Then $l_1=l_2=1$ and $g_1=g_2=0$. Let $d=s_0$ or $s_1$. Then
\begin{align*}
f_d(1)&=\frac{1}{2}\frac{1}{2\pi i}\frac{1}{2\pi i}q^{-1}\int_{\T}\int_{\T}q\frac{(z_1-z_2)(z_1-z_2)}{(z_1-q^{-1}z_2)(z_1-qz_2)}\frac{\mathrm{d}z_1}{z_1}\frac{\mathrm{d}z_2}{z_2}
\\
&=
\frac{1}{2}\frac{1}{2\pi i}\int_{\T}\mathrm{Res}_{z_1=q^{-1}z_2}\frac{(z_1-z_2)(z_1-z_2)}{(z_1-q^{-1}z_2)(z_1-qz_2)}\frac{1}{z_1}
+\mathrm{Res}_{z_1=0}\frac{(z_1-z_2)(z_1-z_2)}{(z_1-q^{-1}z_2)(z_1-qz_2)}\frac{1}{z_1}\frac{\mathrm{d}z_2}{z_2}
\\
&=
\frac{1}{2}\frac{1}{2\pi i}\int_{\T}
\frac{(q^{-1}z_2-z_2)(q^{-1}z_2-z_2)}{(q^{-1}z_2-qz_2)q^{-1}z_2}+\frac{z_2^2}{z_2^2}\frac{\mathrm{d}z_2}{z_2}
\\
&=
\frac{1}{2}\frac{1}{2\pi i}\int_{\T}\frac{(q^{-1}-1)^2}{q^{-2}-1}+1\frac{\mathrm{d}z_2}{z_2}
\\
&=
\frac{1}{q+1}.
\end{align*}
The factor $\frac{1}{2}$ reflects the fact that we integrate with the respect to the pushforward
of the above $\Sn_2$-invariant measure to the quotient $\T\times\T/\Sn_2$.

This agrees with the theorem, which instructs us to calculate $f_d(1)$ as follows:
There are two trees, each of which has one vertex and no edges. The trees are $z_1=0$ and $z_1=q^{-1}z_2$.
Each has one branch. The first has no clumps, so the entire is product is empty. The second
tree has one clump $C=\{1,2\}$ for which $t=1$,  and we obtain
\[
f_d(1)=\frac{1}{2}\left(1+\frac{1-q}{1+q}\right)=\frac{1}{2}\frac{2}{1+q}=\frac{1}{1+q}.
\]
\end{ex}
Now we give an example of a less singular cell.
\begin{ex}
\label{ex less singular cell}
Let $G=\GL_4$, and let $\cc$ be the two-sided cell corresponding to the partition $4=2+2$. Then
$P_W(\bq)=(1+\bq)(1+\bq+\bq^2)(1+\bq+\bq^2+\bq^3)$ has five distinct roots 
$\bq=-1,\pm i, \zeta_1,\zeta_2$, where $\zeta_i$ are primitive third roots of unity. We compute 
$P_{\cc}$ using the Plancherel formula. We have

We have $l_1=l_2=2$, $q_{12}=1$, $q^{12}=q^2$, 
\[
c_M=\frac{q^4(q-1)^2}{2(q+1)^2},
\]
and
\begin{align*}
\left(\frac{1}{2\pi i}\right)^{2}\iint_{\T^2}\Gamma^{12}\frac{\mathrm{d}z_1}{z_1}
\frac{\mathrm{d}z_2}{z_2}&=
\left(\frac{1}{2\pi i}\right)^{2}\iint_{\T^2}\frac{(z_1-z_2)(z_1-z_2)}{(z_1-q^2z_2)(z_1-q^{-2}z_2)}\frac{\mathrm{d}z_1}{z_1}
\frac{\mathrm{d}z_2}{z_2}
\\
&=
\frac{1}{2\pi i}\int_{\T}\mathrm{Res}_{z_1=0}\frac{\Gamma^{12}}{z_1}+\mathrm{Res}_{z_1=q^{-2}z_2}\frac{\Gamma^{12}}{z_1}
\frac{\mathrm{d}z_2}{z_2}
\\
&=
1+\frac{1}{2\pi i}\int_{\T}\frac{(q^{-2}-1)^2q^2}{(q^{-2}-q^2)}\frac{\mathrm{d}z_2}{z_2}
\\
&=
1+\frac{(1+q)(1-q^2)}{1+q+q^2+q^3}
\\
&=
\frac{2}{1+q^2}.
\end{align*}
Accounting for $c_M$, we see that $J_{\cc}$ is regular at $\bq=\zeta_1, \zeta_2$.
\end{ex}
\subsubsection{Proofs of Theorem \ref{theorem fd(1) integral}, Lemma \ref{lemma sum of z_i powers must be -N}, and Corollary \ref{corollary fw(1) denominator divides poincare}}
\begin{proof}[Proof of Theorem \ref{theorem fd(1) integral}]
%
%
Let $n=l_0+\cdots +l_N$. We may assume that $l_0\leq l_2\leq\cdots\leq l_N$.
It suffices to evaluate the integral 
\[
\left(\frac{1}{2\pi i}\right)^{N+1}\int_{\T}\cdots\int_{\T}\prod\Gamma_{i,j,g}\frac{\mathrm{d}z_0}{z_0}\cdots\frac{\mathrm{d}z_N}{z_N}
=
\left(\frac{1}{2\pi i}\right)^{N+1}
\int_{\T}\cdots\int_{\T}\prod_{i<j}\Gamma^{ij}\frac{\mathrm{d}z_0}{z_0}\cdots\frac{\mathrm{d}z_N}{z_N}.
\]
We claim that the value of this integral is
\begin{multline}
\label{fd(1) theorem pre-cancellation}
c_M\sum_{\mathrm{trees}~T}\sum_{\mathrm{branches}~B\mathrm{of}~T}\prod_{\substack{C\prec B \\ C=\{i_0,\dots, i_t\}}}\prod_{k=0}^{t-1}
\frac{(1-q^{i_ki_{k+1}}q_{i_ki_{k+1}})(1-q^{i_ki_{k+1}}(q_{i_ki_{k+1}})^{-1})}{1-(q^{i_ki_{k+1}})^2}
\\
\cdot
\prod_{r=0}^{k-1}\frac{(1-Q_{rk}q_{i_ri_{k+1}})(1-Q_{rk}(q_{i_ri_{k+1}})^{-1})}{(1-Q_{rk}q^{i_ri_{k+1}})(1-Q_{rk}(q^{i_ri_{k+1}})^{-1})},
\end{multline}
where the sum over trees is taken over all bookkeeping trees for the integral. When $k=0$, we 
interpret the product over $r$ as being empty.

First we explain how \eqref{fd(1) theorem pre-cancellation} simplifies to \eqref{fd(1) theorem post-cancellation}. All cancellations will take place within the same clump $C$ of some branch $B$, which we now fix. We have
\[
1-Q_{rk}(q^{i_ri_{k+1}})^{-1}=1-q^{g_{i_k}+g_{i_{k+1}}+1}q^{i_ki_{k+1}}\cdots q^{i_ri_{r+1}}q^{-g_{i_r}-g_{i_{k+1}}-1}=1-
Q_{r,k-1}q^{g_{i_k}-g_{i_r}},
\]
which is one of the factors in the product $(1-Q_{r,k-1}q_{i_ri_k})(Q_{r,k-1}(q_{i_ri_k})^{-1})$. 
The surviving factor in the numerator at index $(k-1,r)$ is then equal to $(1-Q_{r,k-1}q^{g_{i_r}-g_{i_k}})$. In short, the above factors in the denominator cancel with a numerator occurring with the 
same $r$-index but $k$-index one lower. Such a factor occurs whenever $r<k-1$ (note that this inequality 
does not hold when $k=1$ and $r=0$). When $r=k-1$, we have 
\[
1-Q_{k-1,k}(q^{i_ki_{i+1}})^{-1}=1-q^{i_ki_{k+1}}q^{i_{k-1}i_k}(q^{i_{k-1}i_{k+1}})^{-1}=q^{g_{i_k}+g_{i_{k+1}}+1}q^{g_{i_{k-1}}-g_{i_{k+1}}}=1-q^{2g_{i_k}+1},
\]
which is one of the factors in $(1-q^{i_{k+1}i_k}q_{i_ki_{k+1}})(1-q^{i_ki_{k+1}}(q_{i_{k}}i_{k+1})^{-1})$. The cancellation leaves behind the factor $1-q^{l_{i_{k+1}}}$ in 
the numerator, except for $k=0$; this term keeps both its denominators. At this point we have shown that the factor corresponding to $C$ in \eqref{fd(1) theorem pre-cancellation} simplifies to
\[
\frac{(1-q^{l_{i_0}})(1-q^{l_{i_1}})}{1-q^{l_{i_0}+l_{i_1}}}\prod_{k=1}^{t-1}\frac{(1-q^{l_{i_{k+1}}})}{(1-q^{l_{i_k}+l_{i_{k+1}}})}\prod_{r=0}^{k-1}\frac{R_{rk}}{1-Q_{rk}q^{i_ri_{k+1}}},
\]
where
\[
R_{rk}=\begin{cases}
1-Q_{rk}q^{g_{i_r}-g_{i_k}}&\text{if}~k<t-1\\
(1-Q_{r,t-1}q^{g_{i_r}-g_{i_t-1}})(1-Q_{r,t-1}q^{g_{i_k}-g_{i_r}})&\text{if}~k=t-1
\end{cases}.
\]
This means that \eqref{fd(1) theorem pre-cancellation} simplifies to 
\eqref{fd(1) theorem post-cancellation}.

To prove \eqref{fd(1) theorem pre-cancellation}, we will use the residue theorem for each variable consecutively, keeping 
track of the constant expressions in $q$ that we extract after integrating with respect to each variable $z_i$.
More precisely, we will track what happens in a single summand corresponding to some set of successive
choices of poles to take residues at.
Note that all the rational functions that will appear, namely the 
$\Gamma^{ij}$ or the rational functions that result from substitutions into the $\Gamma^{ij}$,
become equal to $1$ once $z_i$ or $z_j$ is set to zero. Therefore poles at $z_i=0$ serve
simply to remove all factors involving $z_i$ from inside the integrand (we will see what these rational 
functions are below). It follows that the summand whose branch we are computing is a product over clumps 
in the corresponding branch, so it suffices to compute the value of a given clump for some ordered subsets $\{i_0,i_1,\dots, i_l\}$ of the indices $\{0,\dots, N\}$. As we are inside a clump, we will consider only poles occurring at nonzero complex numbers. Thus we are left only to 
determine what happens within a  single clump.

We first integrate with respect to the variable $z_{i_0}$. The residue theorem gives
\begin{align*}
&\left(\frac{1}{2\pi i}\right)^{l+1}\int_{\T}\cdots\int_{\T}\prod_{i<j}\Gamma^{ij}\frac{\mathrm{d}z_{i_0}}{z_{i_0}}\cdots\frac{\mathrm{d}z_{i_l}}{z_{i_l}}
\\
&=
\left(\frac{1}{2\pi i}\right)^{l}
\int_{\T}\cdots\int_{\T}\sum_{l\neq i_0}\mathrm{Res}_{z_{i_0}=(q^{i_0l})^{-1}z_l}\frac{1}{z_{i_0}}\prod_{i<j}\Gamma^{ij}\frac{\mathrm{d}z_{i_1}}{z_{i_1}}\cdots\frac{\mathrm{d}z_{i_l}}{z_{i_l}}
+ 
\left(\frac{1}{2\pi i}\right)^{l}
\int_{\T}\cdots\int_{\T}\prod_{\substack{i<j \\ i,j\neq i_0}}\Gamma^{ij}\frac{\mathrm{d}z_{i_1}}{z_{i_1}}\cdots\frac{\mathrm{d}z_{i_l}}{z_{i_l}}.
\end{align*}
As noted above, the second integral belongs to a different branch (in fact, in the case of $i_0$, to a 
different tree); our procedure will deal with it separately, and we will now consider what happens with the 
first integral. 

For the first integral, consider 
one of the summands corresponding to $z_{i_0}=(q^{i_0i_1})^{-1}z_{i_1}$ for some $i_1$. We have
\begin{align*}
&\left(\frac{1}{2\pi i}\right)^{l}\int_{\T}\cdots\int_{\T}\mathrm{Res}_{z_{i_0}=(q^{i_0i_1})^{-1}z_{i_1}}\frac{1}{z_{i_0}}\prod_{i<j}\Gamma^{ij}\frac{\mathrm{d}z_{i_1}}{z_{i_1}}\cdots\frac{\mathrm{d}z_{i_l}}{z_{i_l}}
\\
&=
\frac{(1-q^{i_0i_1}q_{i_0i_1})(1-q^{i_0i_1}(q_{i_0i_1})^{-1})}{1-(q^{i_0i_1})^2}
\left(\frac{1}{2\pi i}\right)^{l}
\int_{\T}\cdots\int_{\T}
\prod_{j\neq i_1,i_0}\frac{(z_{i_1}-q^{i_0i_1}q_{i_0j}z_j)(z_{i_1}-q^{i_0i_1}(q_{i_0j})^{-1}z_j)}{(z_{i_1}-q^{i_0i_1}q^{i_0j}z_j)(z_{i_1}-q^{i_0i_1}(q^{i_0j})^{-1}z_j)}
\\
&\prod_{\substack{i<j \\ i,j\neq i_0}}\Gamma^{ij}\frac{\mathrm{d}z_{i_1}}{z_{i_1}}\frac{\mathrm{d}z_{i_2}}{z_{i_2}}\cdots\frac{\mathrm{d}z_{i_l}}{z_{i_l}}.
\end{align*}
Recall that,
if $i_1>2$, even though formally we have defined the symbols $q_{ij}$ and $q^{ij}$  only for $i<j$, the 
symmetry of the factors allows us  to write $q_{ij}$ even if $i>j$, thanks to the factor with $(q_{ij})^{-1}$ also present in the numerator.


Now we integrate with respect to $z_{i_1}$. Observe that the leftmost product over $j\neq i_1, i_0$ does not contribute poles. Indeed, the first factor each denominator
does not have its zero contained in $\T$, and the second factor in each denominator has its zero at
$z_{i_1}=q^{i_0i_1}(q^{i_0i_2})^{-1}z_{i_2}$. The power of $q$ appearing is
\[
g_{i_0}+g_{i_1}+1-(g_{i_0}+g_{i_2}+1)=g_{i_1}-g_{i_2},
\]
and so $q^{i_0i_1}(q^{i_0i_2})^{-1}$ is equal to $q_{i_1i_2}$ or $(q_{i_2i_1})^{-1}$, whichever
is defined. Thus this zero cancels with a zero in the numerator of $\Gamma^{i_1i_2}$ or $\Gamma^{i_2i_1}$, whichever is defined. Therefore we need only consider the simple poles at $z_{i_1}=(q^{i_1i_2})^{-1}z_{i_2}$
for $i_1<i_2$ and $z_{i_1}=(q^{i_2i_1})^{-1}z_{i_2}$ for $i_2<i_1$. Observe that the 
residues will be the same for either inequality. In the case $i_2<i_1$, for example, $\Gamma^{i_2i_1}$ 
needs to be rewritten so that its simple pole inside $\T$ is in the correct format to calculate the 
residue by substitution:
\begin{align*}
&\mathrm{Res}_{z_{i_1}=(q^{i_2i_1})^{-1}z_{i_2}}\frac{(z_{i_2}-q_{i_2i_1}z_{i_1})(z_{i_2}-(q_{i_2i_1})^{-1}z_{i_1})}{(z_{i_2}-q^{i_2i_1}z_{i_1})(z_{i_2}-(q^{i_2i_1})^{-1}z_{i_1})}\frac{1}{z_{i_1}}
\\
&=
\mathrm{Res}_{z_{i_1}=(q^{i_2i_1})^{-1}z_{i_2}}\frac{-(q^{i_2i_1})^{-1}(z_{i_2}-q_{i_2i_1}z_{i_1})(z_{i_2}-(q_{i_2i_1})^{-1}z_{i_1})}{(z_{i_2}-q^{i_2i_1}z_{i_1})(z_{i_2}-(q^{i_2i_1})^{-1}z_{i_1})}\frac{1}{z_{i_1}}
\\
&=
\frac{(1-q^{i_2i_1}q_{i_2i_1})(1-q^{i_2i_1}(q_{i_2i_1})^{-1})}{1-(q^{i_2i_1})^{2}}.
\end{align*}
Therefore after integrating within the clump at hand with respect to $z_{i_0}$ and then $z_{i_1}$, we
have the expression
\begin{align}
\label{f_d(1) integral after integrating in two variables}
&\frac{(1-q^{i_0i_1}q_{i_0i_1})(1-q^{i_0i_1}(q_{i_0i_1})^{-1})}{1-(q^{i_0i_1})^2}
\frac{(1-q^{i_2i_1}q^{i_0i_1}q_{i_0i_2})(1-q^{i_1i_2}q^{i_0i_1}(q_{i_0i_2})^{-1})}{(1-q^{i_2i_1}q^{i_0i_1}q^{i_0i_2})(1-q^{i_2i_1}q^{i_0i_1}(q^{i_0i_2})^{-1})}
\\
\cdot
&
\frac{(1-q^{i_2i_1}q_{i_2i_1})(1-q^{i_2i_1}(q_{i_2i_1})^{-1})}{(1-(q^{i_2i_1})^2)}
\nonumber
\\
\cdot
&
\left(\frac{1}{2\pi i}\right)^{l-1}
\int_{\T}\cdots\int_{\T}
\prod_{j\neq i_0, i_1, i_2}\frac{(z_{i_2}-q^{i_1i_2}q^{i_0i_1}q_{i_0j}z_j)(z_{i_2}-q^{i_1i_2}q^{i_0i_1}(q_{i_0j})^{-1}z_j)}{(z_{i_2}-q^{i_1i_2}q^{i_0i_1}q^{i_0j}z_j)(z_{i_2}-q^{i_1i_2}q^{i_0i_1}(q^{i_0j})^{-1}z_j)}
\nonumber
\\
\cdot
&
\prod_{j \neq i_0,i_1, i_2}\frac{(z_{i_2}-q^{i_1i_2}q_{i_1j}z_j)(z_{i_2}-q^{i_1i_2}(q_{i_1j})^{-1}z_j)}{(z_{i_2}-q^{i_1i_2}q^{i_1j}z_j)(z_{i_2}-q^{i_1i_2}(q^{i_1j})^{-1}z_j)}
\prod_{\substack{i<j \\ i,j\neq i_0,i_1}}\Gamma^{ij}\frac{\mathrm{d}z_{i_2}}{z_{i_2}}\cdots
\frac{\mathrm{d}z_{i_l}}{z_{i_l}}.
\nonumber
\end{align}
Now we integrate with respect to $z_{i_2}$. Again, only poles from the product of $\Gamma^{ij}$'s
occur with nonzero residues: in total we have simple poles contained in $\T$ possibly at $z_{i_2}=q^{i_1i_2}q^{i_0i_1}(q^{i_0j})^{-1}z_j$, at $z_{i_2}=q^{i_1i_2}(q^{i_1j})^{-1}z_j$
and at $z_{i_2}=(q^{i_2,j})^{-1}z_j$ for $j\neq i_0, i_1, i_2$. 
It may happen that these poles are not all distinct, but all zeros in the denominator of the 
former two types are in fact cancelled by zeros of denominator anyway. The necessary factors occur in the 
product immediately adjacent on the right. Indeed, we have
\[
g_{i_1}+g_{i_2}+1+g_{i_0}+g_{i_1}+1-g_{i_0}-g_j-1=g_{i_1}+g_{i_2}+1+g_{i_1}-g_j
\]
and so either $q^{i_1i_2}q^{i_0i_1}(q_{i_0j})^{-1}=q^{i_1i_2}q_{i_1j}$ or 
$q^{i_1i_2}q^{i_0i_1}(q_{i_0j})^{-1}=q^{i_1i_2}(q_{i_1j})^{-1}$ (or $q_{ji_1}$ or $(q_{ji_1})^{-1}$).

Likewise we have 
\[
g_{i_1}+g_{i_2}+1-g_{i_1}-g_j-1=g_{i_2}-g_j
\]
as happened when we integrated with respect to $z_{i_0}$. 

Now we see the following pattern: At each stage, we integrate with respect to the variable we
took a residue at in the previous step. When integrating with respect to $z_{i_r}$, there 
will be $r+1$ products of rational functions, and each rational function in the leftmost
$r$ products will contribute a pole, in addition to the pole contributed by $\Gamma^{i_ri_{r+1}}$ or $\Gamma^{i_{r+1}i_r}$, whichever is defined. However, each is nullified by having zero residue thanks
to a denominator in the product immediately to the right. Integrating with respect to $z_{i_r}$
will result in extracting $r+1$ new rational factors, each of the form claimed in the 
theorem. 
When it comes to integrating with respect to $z_{i_l}$, all variables $z_{i_l}$
will cancel from the remains of the Gamma functions by homogeneity, and the factor $\frac{1}{z_{i_l}}$ will result in the final integral contributing just the remaining $l+1$ rational
factors.
\end{proof}
Lemma \ref{lemma sum of z_i powers must be -N} is a porism of the preceding proof.
\begin{proof}[Proof of Lemma \ref{lemma sum of z_i powers must be -N}]
We will evaluate the integral \eqref{eq integral with z_ie_i powers} by
applying the residue theorem successively for each variable as in the proof of Theorem 
\ref{theorem fd(1) integral}. Note that as functions of any variable $z_k$, the functions
$\prod_{i<j}\Gamma^{ij}$ and all the other products, for example those appearing in \eqref{f_d(1) integral 
after integrating in two variables}, have numerator and denominator with equal degrees. 
Thus the overall sum of powers of all $z_i$ in the integrand of \eqref{eq integral with z_ie_i powers} is $e_0+\cdots+e_N$, and in general, the sum of powers of all $z_i$ in the integrand of an expression
like \eqref{f_d(1) integral after integrating in two variables} is the sum of the degrees of the 
monomial $z_{i_j}$ terms.

When evaluating \eqref{eq integral with z_ie_i powers} along a single branch, we find again that the only poles that appear are of 
the form $z_i=(q^{ij})^{-1}z_j$, or $z_i^{e_i}=0$ for $e_i<0$. We will track the effect that evaluating each successive residue
has on the total degree of the integrand, and then conclude using the fact that 
$\int_\T z^r\,\mathrm{d}z=2\pi i\delta_{r,-1}$. First, observe that evaluating a residue of 
the form $z_i=(q^{ij})^{-1}z_j$ increases the sum of all powers by $1$; a factor $z_j$ is contributed
to the resulting integrand. The sum of all powers is likewise increased by $1$ when evaluating
the residue at a simple pole of $z_i^{-1}$ at $0$. To compute the residue at a pole of $z_i^{-e_i}$ at $0$ for $e_i>1$,
consider the Taylor expansion at $0$ of $\Gamma^{ij}$. We have
\[
\frac{1}{z_i-q^{ij}z_j}=-\frac{1}{q^{ij}z_j}-\frac{z_i}{(q^{ij}z_j)^2}-\frac{z_i^2}{(q^{ij}z_j)^3}-\cdots
\]
and
\[
\frac{1}{z_i-(q^{ij})^{-1}z_j}=-\frac{q^{ij}}{z_j}-\left(\frac{q^{ij}}{z_j}\right)^2z_i-
\left(\frac{q^{ij}}{z_j}\right)^3z_i^2-\cdots.
\]
Multiplying these series and further multiplying by the denominator $z_i^2-(q_{ij}+q_{ij}^{-1})z_iz_j+z_j$,
it follows that the Taylor expansion of $\Gamma^{ij}$ is
\begin{multline}
\label{eq Taylor expasion Gamma ij}
1+\frac{q^{ij}+(q^{ij})^{-1}-q_{ij}-q_{ij}^{-1}}{z_j}z_i+\frac{(q^{ij})^2+2+(q^{ij})^{-2}-(q_{ij}+q_{ij}^{-1})(q^{ij}+q^{-ij})}{z_j^2}z_i^2+\cdots
\\
+
\frac{-(q_{ij}+q_{ij}^{-1})((q^{ij})^{n-1}+\cdots +(q^{ij})^{-n+1})+(q^{ij})^{n-2}+\cdots +(q^{ij})^{-n+2}+(q^{ij})^{n}+\cdots +(q^{ij})^{-n}}{z_j^n}z_i^n
\\
+\cdots.
\end{multline}
It is clear that the salient point of \eqref{eq Taylor expasion Gamma ij}, that $z_i^n$ appears with 
a coefficient proportional to $z_j^{-n}$, holds also for the all the products like those in 
\eqref{f_d(1) integral after integrating in two variables}. Thus computing a residue at the pole of $z_i^{-e_i}$ at $0$
will result in a new integrand, the total degree of which has increased by $e_i-e_i+1=1$.
It now follows that after integrating with respect to $N-1$ variables, the final integral 
to be computed will be a constant times $(2\pi i)^{-1}\int_{\T} z_N^{e_0+\cdots+e_N+N-1}\mathrm{d}z_N$. This is nonzero if 
and only if $e_0+\cdots +e_N=-N$. Summing over all branches of a bookkeeping tree, we see that
\eqref{eq integral with z_ie_i powers} is nonzero only if $e_0+\cdots +e_N=-N$.
\end{proof}
The lemma is helpful for computing explicit examples, as it points out that one can always avoid dealing
with higher-order poles. Indeed, given an integral of the form
\eqref{eq integral with z_ie_i powers}, we may assume by Lemma \ref{lemma sum of z_i powers must be -N} that $e_0+\cdots+e_N=-N$. If $e_i=-1$ for all $i$ then the 
integral \eqref{eq integral with z_ie_i powers} is just the integral from Theorem \ref{theorem fd(1) integral}.
Otherwise there is some $e_{i_0}\geq 0$. We may assume that $i_0=0$. Then the only poles in $z_0$ 
are of the form $z_0=(q^{0i_1})^{-1}z_{i_1}$ for indices $i_1>0$. Therefore we compute that \eqref{eq integral with z_ie_i powers} is equal to a sum of terms of the form
\begin{multline}
\label{eq f_w(1) integral after one integration}
\frac{(1-q^{0i_1}q_{0i_1})(1-q^{0i_1}(q_{0i_1})^{-1})}{1-(q^{0i_1})^2}(q^{0i_1})^{-e_0-1}
\left(\frac{1}{2\pi i}\right)^{N-1}
\int_{\T}\cdots\int_{\T}z_{i_1}^{e_{i_1}+1+e_{0}}z_{i_2}^{e_{i_2}}\cdots z_{i_N}^{e_{i_N}}
\\
\cdot\prod_{j\neq i_1,0}\frac{(z_{i_1}-q^{0i_1}q_{0j}z_j)(z_{i_1}-q^{0i_1}(q_{0j})^{-1}z_j)}{(z_{i_1}-q^{0i_1}q^{0j}z_j)(z_{i_1}-q^{0i_1}(q^{0j})^{-1}z_j)}
\prod_{\substack{i<j \\ i,j\neq 0}}\Gamma^{ij}\mathrm{d}z_{i_1}\cdots\mathrm{d}z_{i_N}.
\end{multline}
The total degree of the integrand is now $e_0+\cdots +e_N+1=-(N-1)$. Therefore either
$e_{i_0}+e_{i_1}+1=e_{i_2}=\cdots =e_{-N}=-1$, or we may again assume without loss of generality that 
the exponent of some $z_{i_j}$ is nonnegative, and proceed with evaluating \eqref{eq f_w(1) integral after one integration} 
by integrating with respect to $z_{i_j}$. We may continue in this way, never having to deal
with more than a simple pole at $0$. (Of course, the resulting order of integration need not be the same
as the order used in the proofs of Theorem \ref{theorem fd(1) integral} and Corollary \ref{corollary fw(1) 
denominator divides poincare}.)

We can now prove Corollary \ref{corollary fw(1) denominator divides poincare}.
\begin{proof}[Proof of Corollary \ref{corollary fw(1) denominator divides poincare}]
In light of Lemma \ref{Lemma trace is regular}, it is enough to prove that the conclusions of the present 
corollary hold for integrals of the form \eqref{eq integral with z_ie_i powers}. We again follow the 
algorithm from the proof of Theorem \ref{theorem fd(1) integral}, so it is sufficient to
consider a single branch, and for this it suffices to observe that, given, for example, a variant 
\begin{align*}
&
\left(\frac{1}{2\pi i}\right)^{l-1}
\int_{\T}\cdots\int_{\T}
\prod_{j\neq i_0, i_1, i_2}\frac{(z_{i_2}-q^{i_1i_2}q^{i_0i_1}q_{i_0j}z_j)(z_{i_2}-q^{i_1i_2}q^{i_0i_1}(q_{i_0j})^{-1}z_j)}{(z_{i_2}-q^{i_1i_2}q^{i_0i_1}q^{i_0j}z_j)(z_{i_2}-q^{i_1i_2}q^{i_0i_1}(q^{i_0j})^{-1}z_j)}
\nonumber
\\
&
\cdot
\prod_{j \neq i_0,i_1, i_2}\frac{(z_{i_2}-q^{i_1i_2}q_{i_1j}z_j)(z_{i_2}-q^{i_1i_2}(q_{i_1j})^{-1}z_j)}{(z_{i_2}-q^{i_1i_2}q^{i_1j}z_j)(z_{i_2}-q^{i_1i_2}(q^{i_1j})^{-1}z_j)}
\prod_{\substack{i<j \\ i,j\neq i_0,i_1}}\Gamma^{ij}z_{i_2}^{e_{i_2}}\cdot\cdots\cdot
z_{i_l}^{e_{i_l}}\mathrm{d}z_{i_2}\cdots\mathrm{d}z_{i_l}
\end{align*}
of \eqref{f_d(1) integral after integrating in two variables} in which $e_{i_2}<-1$, the quotient rule
gives that the residue at $z_{i_2}=0$ is a equal to a linear combination over 
$\Q[q^{\frac{1}{2}}, q^{-\frac{1}{2}}]$ of integrals of the form
\[
\left(\frac{1}{2\pi i}\right)^{l-2}
\int_{\T}\cdots\int_{\T}\prod_{\substack{i<j \\ i,j\neq i_0,i_1,i_2}}\Gamma^{ij}z_{i_3}^{e_{i_3}'}
\cdot\cdots\cdot
z_{i_l}^{e_{i_l}'}\mathrm{d}z_{i_3}\cdots\mathrm{d}z_{i_l}
\]
for new exponents $e'_{i_k}$. This and the 
calculations in the proof of Theorem \ref{theorem fd(1) integral} make clear that the only positive powers of 
$q^{1/2}$ that appear in the any denominator are those that appeared as in Theorem 
\ref{theorem fd(1) integral}. These are controlled by the possible block sizes of $M$, and hence are bounded in 
terms of $\Waff$. Throughout this procedure, the denominator has been contributed to only by the Plancherel 
density itself, which depends only on $M$. The last claim of the corollary now follows from Proposition 
\ref{prop plancherel decomposition compatible with cell decomposition}. 
\end{proof}
\subsection{The functions $f_w$ for general $\G$}
\label{subsection the functions fw for general G}
For general $\G$, we will follow the same plan as for $\G=\GL_n$. The only difference is that 
we have less control over which denominators can appear. Indeed, this is true even for formal degrees, but 
complications are also introduced by residual coset we integrate over, or equivalently, by lack of explicit
control of the Satake parameter of $\pi=i_P^G(\omega\otimes\nu)$ for arbitrary discrete series representations
$\omega\in\mathcal{E}_2^I(M)$. 
\begin{theorem}
\label{theorem general G f_w(1) rational function with correct denominator}
Let $w\in\Waff$. Then $f_w(1)$ is a rational function of $q$ with poles drawn from a finite set of roots of unity depending only on $\Waff$. The numerator is a Laurent polynomial in 
$q^{1/2}$. The denominator depends only on the two-sided cell containing $w$. If $w$ is in the lowest 
two-sided cell, then the denominator divides the Poincar\'{e} polynomial of $\G$.
\end{theorem}
In the proof, we do not attempt to record in any information about degrees of the numerators.
It will therefore be necessary to control the possible numerators of $f_w(1)$ in a different manner 
than for $\G=\GL_n$ in order to prove Proposition \ref{prop phi1 existence} below.
\begin{proof}[Proof of Theorem \ref{theorem general G f_w(1) rational function with correct denominator}]
First, we note that the reasoning for the lowest cell used in the proof of Corollary \ref{cor denominators of fd(1) poincare} holds for arbitrary $G$. Therefore \eqref{eq d lowest cell trace via Haar renorm} 
proves the last claim just as for $\GL_n$.
%

More generally, let $w\in\Waff$ and $M_P$ be a Levi subgroup corresponding to the two-sided cell containing $w$. 
Let $N=\rk A_P$. Given a coroot $\alpha^\vee$ and a basis of the cocharacter lattice as explained in Section
\ref{subsection plancherel following Opdam}, we write $\alpha^\vee=z_1^{e_1}\cdots z_n^{e_n}$
for integers $e_i=e_i(\alpha)$.

By Lemma \ref{Lemma trace is regular} and Theorem \ref{FOS theorem}, it suffices to show the conclusions of 
the theorem hold for integrals of the form
\begin{multline}
\label{opdam integral in coordinates}
\left(\frac{1}{2\pi i}\right)^n\int_{\T}\cdots\int_{\T}
\prod_{\alpha\in R_{1,+}\setminus R_{P, 1,+}}
\frac{q_\alpha(z_{1}^{e_1}z_2^{e_2}\cdots z_n^{e_n}-q_\omega)(z_{1}^{e_1}z_2^{e_2}\cdots z_n^{e_n}-q_\omega^{-1})}{\left(z_{1}^{\frac{e_1}{2}}z_2^{\frac{e_2}{2}}\cdots z_n^{\frac{e_n}{2}}+q_\omega^{1/2} q_\alpha^{1/2}\right)\left(z_{1}^{\frac{e_1}{2}}z_2^{\frac{e_2}{2}}\cdots z_n^{\frac{e_n}{2}}+q_\omega^{-1/2} q_\alpha^{-1/2}\right)}
\\
\cdot
\frac{1}{\left(z_{1}^{\frac{e_1}{2}}z_2^{\frac{e_2}{2}}\cdots z_n^{\frac{e_n}{2}}-q_\omega^{1/2} q_\alpha^{1/2}q_{2\alpha}\right)\left(z_{1}^{\frac{e_1}{2}}z_2^{\frac{e_2}{2}}\cdots z_n^{\frac{e_n}{2}}-q_\omega^{1/2} q_\alpha^{1/2}q_{2\alpha}\right)}z_1^{f_1}\cdots z_n^{f_n}\,
\frac{\mathrm{d}z_1}{z_1}
\cdots\frac{\mathrm{d}z_n}{z_n},
\end{multline}
where $f_i\in\Z$ and $q_\omega=q_{\omega,\alpha}$ is the value of $\alpha^\vee$ on the Satake parameter of 
$\omega$---this is a positive power of $q^{1/2}$ (of $q$ if $\alpha^\vee/2$ is a coroot) by Sections 7 and 8 of 
\cite{OpdamSpectral}. We have $q_{2\alpha}=1$ and the resulting simplification 
\eqref{eqn opdam factor simplification} whenever $\alpha^\vee/2$ is not a coroot.

With notation fixed, the theorem is essentially an observation. Indeed, suppose that we integrate with 
respect to $z_1$, and wish to compute the contribution to the residue at a pole
\[
z_1=\xi  q_\omega^{-\frac{1}{e_1}}q_\alpha^{-\frac{1}{e_1}}z_2^{\frac{e_2}{e_1}}\cdots z_n^{\frac{e_n}{e_1}}
\]
arising from the factor
\begin{multline*}
\frac{q_\alpha(z_{1}^{e_1}z_2^{e_2}\cdots z_n^{e_n}-q_\omega)(z_{1}^{e_1}z_2^{e_2}\cdots z_n^{e_n}-q_\omega^{-1})}{(z_{1}^{e_1}z_2^{e_2}\cdots z_n^{e_n}-q_\omega q_\alpha)(z_{1}^{e_1}z_2^{e_2}\cdots z_n^{e_n}-q_\omega^{-1} q_\alpha^{-1})}\frac{1}{z_1}=
\\
\frac{
q_\alpha(z_1^{e_1}-q_\omega z_2^{-e_2}\cdots, z_n^{-e_n})(z_1^{e_1}-q_\omega^{-1} z_2^{-e_2}\cdots, z_n^{-e_n})
}{(z_1^{e_1}-q_\omega q_\alpha z_2^{-e_2}\cdots, z_n^{-e_n})\prod_{\zeta}(z_1-\zeta q_\omega^{-\frac{1}{e_1}}q_\alpha^{-\frac{1}{e_1}}z_2^{\frac{e_2}{e_1}}\cdots z_n^{\frac{e_n}{e_1}})}\frac{1}{z_1},
\end{multline*}
where $\xi$ and the $\zeta$ are primitive $e_1$-st roots of unity. The contribution is then 
\begin{multline}
\label{eqn general residue pre simplification}
\frac{q_\alpha
(q_\omega^{-1}q_\alpha^{-1}z_2^{-e_2}\cdots z_n^{-e_n}-q_\omega z_2^{-e_2}\cdots z_n^{-e_n})
(q_\omega^{-1}q_\alpha^{-1}z_2^{-e_2}\cdots z_n^{-e_n}-q_\omega^{-1} z_2^{-e_2}\cdots z_n^{-e_n})
}
{
(q_\omega^{-1}q_\alpha^{-1}z_2^{-e_2}\cdots z_n^{-e_n}-q_\omega q_\alpha z_2^{-e_2}\cdots z_n^{-e_n})
\prod_{\zeta\neq\xi}
(\xi q_\omega^{\frac{-1}{e_1}}q_\alpha^{\frac{-1}{e_1}}z_2^{\frac{-e_2}{e_1}}\cdots z_n^{\frac{-e_n}{e_1}}-\zeta q_\omega^{\frac{-1}{e_1}} z_2^{\frac{-e_2}{e_1}}\cdots z_n^{\frac{-e_n}{e_1}})
}
\\
\cdot
\frac{1}
{
\xi  q_\omega^{-\frac{1}{e_1}}q_\alpha^{-\frac{1}{e_1}}z_2^{\frac{e_2}{e_1}}\cdots z_n^{\frac{e_n}{e_1}}
},
\end{multline}
which simplifies to
\begin{equation}
\label{eqn general residue post simplification}
\frac{
(1-q_\alpha q_\omega^2)(1-q_\alpha)
}{
(1-q_\alpha^2q_\omega^2)
\prod_{\zeta\neq\xi}(1-\xi^{-1}\zeta)
}.
\end{equation}
We have used no special properties of the integers $e_i$ or $f_j$, and 
thus after integrating with respect to $z_1$, \eqref{opdam integral in coordinates} is equal to a sum 
of integrals with respect to $z_2,\dots, z_n$ again of the form \eqref{opdam integral in coordinates}, 
except the factors in the denominator will now involve powers $q_\omega(\alpha)q_{\omega}(\alpha')^{1/2}$
and $q_\omega(\alpha)q_\omega(\alpha')^{-1/2}$, as for $\G=\GL_n$. The coefficients of the this sum are the 
form
\begin{equation*}
\frac{Q}{(1\pm q^{r_1})^{c_1}\cdots (1\pm q^{r_k})^{c_k}},
\end{equation*}
where $Q$ is a Laurent polynomial in $q^{\frac{1}{2}}$, $r_i\in\frac{1}{2}\N$ with, \textit{a priori}, complex 
coefficients, and $c_i\in\N$. 
Indeed, as we never required any cancellations with the numerator to extract rational functions of $q$ of 
the required form,
it is clear that the simplification of \eqref{eqn general residue pre simplification} to 
\eqref{eqn general residue post simplification} works essentially the same way for any higher order poles that
appear---again thanks to the quotient rule---and that when integrating with respect to subsequent variables, 
the additional rational functions of $q$ appearing have the same shape as in the proof of Theorem
\ref{theorem fd(1) integral}. It is also clear that the factors corresponding to non-reduced roots behave
similarly.

Therefore \eqref{opdam integral in coordinates} has poles in $q$ only at a finite number of 
roots of unity. Again by the quotient rule, exponents $r_i$ and $c_i$ depend only on $M$. Clearly
there are only finitely-many exponents $r_i$ that appear for any $M$. The theorem now follows.
\end{proof}
\subsection{Relating $t_w$ and $f_w$}
\label{subsection relating tw and fw}
We will now relate the Schwartz functions $f_w$ on $G$ to the elements
$\phi^{-1}(t_w)$ of a completion $\HHh$ of $\HH$, whose definition
we will now recall. In this section $\G$ is general.
\subsubsection{Completions of $\HH$ and $J\otimes_\C\mathcal{A}$}
\label{subsubsection completions of HH and JotimesZA}
Let $\hat{\mathcal{A}}=\C((\bq^{-1/2}))$ and $\hat{\mathcal{A}}^-=\C[[\bq^{-1/2}]]$.
Write $\HHh$ for the $\hat{\mathcal{A}}$-algebra
\[
\HHh:=\sets{\sum_{x\in\Waff}b_xT_x}{b_x\in\hat{\mathcal{A}},~ b_x\to 0~\text{as}~\ell(x)\to\infty},
\]
where we say that $b_x\to 0$ as $\ell(x)\to\infty$ if for 
all $N>0$, $b_x\in(\bq^{-1/2})^N\hat{\mathcal{A}}^-$ for all $x$ sufficiently long. 

Consider also the completions
\[
\HHh_{C'}:=\sets{\sum_{x\in\Waff}b_xC'_x}{b_x\in\hat{\mathcal{A}},~b_x\to 0~\text{as}~\ell(x)\to\infty}
\] 
and
\[
\HHh_{C}:=\sets{\sum_{x\in\Waff}b_xC_x}{b_x\in\hat{\mathcal{A}},~b_x\to 0~\text{as}~\ell(x)\to\infty}
\] 
of $\HH$ (note the difference between $C'_x$ and $C_x$), as well as the completion
\[
\JJ:=\sets{\sum_{x\in\Waff}b_xt_x}{b_x\in\hat{\mathcal{A}},~b_x\to 0~\text{as}~\ell(x)\to\infty}
\]
of $J\otimes_\C\mathcal{A}$.

In \cite{affineII}, Lusztig shows that
$\phi$ extends to an isomorphism of $\hat{\mathcal{A}}$-algebras $\HHh_C\to\JJ$. In this way the elements 
$t_w\in J\subset\JJ$ may be identified with elements of $\HHh_{C}$ via $\phi$.
\begin{lem}
\label{lem Cw contained in Tw completion}
We have $\HHh_{C'}\subset\HHh$. The inclusion is continuous.
\end{lem}
\begin{proof}
Given an infinite sum $\sum_{x}b_xC'_x$, upon rewriting this sum in the standard basis, the coefficient
of some $T_y$ is 
\[
a_y:=\sum_{x\geq y}b_x\bq^{-\frac{\ell(x)}{2}}P_{y,x}(\bq).
\]
As $\deg P_{y,x}\leq \frac{1}{2}(\ell(x)-\ell(y)-1)$, we have that $\bq^{-\frac{\ell(x)}{2}}P_{y,x}(\bq)$
is a polynomial in $\bq^{-1/2}$. Therefore the above sum defines a formal Laurent series. Moreover, as
$\ell(y)\to \infty$, it is clear that $a_y\to 0$. Continuity of the inclusion is clear from 
the formula for $a_y$.
\end{proof}
\subsubsection{The functions $f_w$ and the basis elements $t_w$}
\label{subsubsection the functions fw and the basis elements tw}
We shall now explain how the map $\tilde{\phi}\colon J\to\Cc(G)^I$ induces a map of $\mathcal{A}$-algebras
$\hat{\phi}\colon J\otimes_\C\mathcal{A}\to\HHh$. 
\begin{prop}
\label{prop phi1 existence}
There is a map of $\mathcal{A}$-algebras $\hat{\phi}\colon J\otimes_\C\mathcal{A}\to\HHh$ such that if
$\hat{\phi}(t_w)=\sum_{x}a_{x,w}T_x$, then $a_{x,w}(q)=f_w(x)$. Moreover, there is a constant $N$
depending only on $\Waff$ such that $a_{x,w}\in(\bq^{1/2})^N\hat{\mathcal{A}}^-$ for all $x,w\in\Waff$.\end{prop}
The most difficult part of the proof of the proposition is showing that $a_{x,w}\to 0$
as $\ell(x)\to\infty$.  To this end we have 
\begin{lem}
\label{lem f_w(1) polynomial growth and a1w bounded in q}
Let $w\in\Waff$. Then the degree in $q$ of the numerator of $f_w(1)$ is bounded uniformly in $w$ by some $N$ depending only on $\Waff$, and hence 
in the notation of Proposition \ref{prop phi1 existence} above, we have $a_{1,w}\in(\bq^{1/2})^N\hat{\mathcal{A}}^-$ for all $w\in\Waff$.
\end{lem}
\begin{rem}
In type $A$,
the Lemma follows immediately from the order of integration given after the proof 
of Lemma \ref{lemma sum of z_i powers must be -N}.
\end{rem}
\begin{proof}[Proof of Lemma \ref{lem f_w(1) polynomial growth and a1w bounded in q}]
Let $\pi_\omega=\Ind_P^G(\nu\otimes\omega)$ be a tempered $I$-spherical representation of $G$ arising by induction 
from a parabolic $P$ such that $\rank A_P=k$. Let $\T_\omega^k$ be the compact torus parametrizing twists of $\omega$, 
and let $t_w$ be given. By Lemma \ref{Lemma trace is regular}, $\trace{\pi}{f_w}$ is a regular function
on $\T^k$, \textit{i.e.} a Laurent polynomial in the coordinates $z_1,\dots, z_k$ on $\T^k$. The coefficients
of this Laurent polynomial are independent of $q$, as $J$ and its representation theory are independent of $q$. As we have 
\[
|1-q^{r}|\leq |z_i-q^{r}z_1^{e_1}\cdots z_k^{e_k}|\leq |1+q^{r}|
\]
for all $z_i\in\T$ and $r$, we may bound, for large $q$, the absolute value $|\mu_{M_P}(z)|$ of the Plancherel 
density for 
$M_P$ by a rational function $U_{M_P,\omega}(q)$. By Theorem \ref{FOS theorem}, we may do the same for formal 
degrees $d(\omega)$, bounding $|d(\omega)|$ by a rational function $D_\omega(q)$. Thus we define a rational 
function of $q$ by 
\[
\frac{h_1(q)+h_2(q^{-1})}{k_1(q)+k_2(q^{-1})}=\sum_{\omega}U_{M_P,\omega}(q)D_\omega(q)\max_{z\in\T_\omega^k}
|\trace{\pi_\omega}{f_w}|,
\]
where $h_1,k_1\in\C[q]$ and $h_2,k_2\in q^{-1}\C[q^{-1}]$, and the sum is over the finitely many, up to unitary twist, pairs $(M,\omega)$, for $\omega\in\mathcal{E}_2(M)$, such that $\trace{\pi_\omega}{f_w}\neq 0$.

Fix $w\in\Waff$. Then
\[
|f_w(1)|=\left|\sum_{(P,\omega)}\int_{\Oo(\omega)}\trace{\pi_\omega}{f_w}d(\omega)\,\mathrm{d}\mu_{M_P}(z)\right|
\leq
\frac{h_1(q)+h_2(q^{-1})}{k_1(q)+k_2(q^{-1})}.
\]
Crucially, this expression holds for all $q\gg 1$.

On the other hand, by Corollary \ref{corollary fw(1) denominator divides poincare} and Theorem \ref{theorem 
general G f_w(1) rational function with correct denominator}, $f_w(1)$ is a rational function of $q$,
and for $q$ sufficiently large, we may write
\[
|f_w(1)|=\frac{f_1(q)+f_2(q^{-1})}{d_1(q)+d_2(q^{-1})}
\]
where $f_1,d_1\in\C[q]$ and $f_2,d_2\in q^{-1}\C[q^{-1}]$. Therefore for all $q\gg 1$ we have
\begin{equation}
\label{eqn Lem 8 big inequality}
|(f_1(q)+f_2(q^{-1}))||(k_1(q)+k_2(q^{-1}))|\leq 
|(d_1(q)+d_2(q^{-1})||(h_1(q)+h_1(q^{-1}))|.
\end{equation}
We claim that this implies 
\begin{equation}
\label{eqn Lem 8 target inequality}
|f_1(q)k_1(q)|\leq |d_1(q)h_1(q)|
\end{equation}
for $q$ sufficiently large.
Indeed, let $\epsilon>0$ be given and choose $q\gg 1$ such that
\[
\left||(f_1(q)+f_2(q^{-1}))||(k_1(q)+k_2(q^{-1}))|-|f_1(q)k_1(q)|\right|\leq \epsilon
\]
and
\[
\left||(d_1(q)+d_2(q^{-1}))||(h_1(q)+h_2(q^{-1}))|-|d_1(q)h_1(q)|\right|\leq \epsilon,
\]
and \eqref{eqn Lem 8 big inequality} holds. Then we have
\begin{align*}
|f_1(q)||k_1(q)|
&\leq |(f_1(q)+f_2(q^{-1}))||(k_1(q)+k_2(q^{-1}))|+\epsilon
\\
&\leq |(d_1(q)+d_2(q^{-1}))||(h_1(q)+h_2(q^{-1}))|+\epsilon
\\
&\leq |d_1(q)||h_1(q)|+2\epsilon,
\end{align*}
which proves \eqref{eqn Lem 8 target inequality}.

Therefore
\[
\deg f_1\leq \deg f_1+\deg k_1\leq\deg d_1+\deg h_1.
\]
Now, the denominator of $f_w(1)$, and hence $\deg d_1$, depends only on the two-sided cell containing $w$,
again by Corollary \ref{corollary fw(1) denominator divides poincare} and Theorem \ref{theorem 
general G f_w(1) rational function with correct denominator}. We can also bound $\deg h_1$ uniformly 
in terms of $\Waff$, as it depends only on the Plancherel measure and the finitely many possible formal
degrees appearing in the parametrization of the $I$-spherical part of the tempered dual of $G$.
This proves the lemma.
\end{proof}
If $t_w\in t_dJt_{d'}$ for $d\neq d'$, then $t_w$ is a commutator. We record this observation as
\begin{lem}
\label{lem only diagonal f_w have nonzero trace}
We have that $f_w(1)=0$ unless $f_d\star f_w\star f_d\neq 0$ for some distinguished involution $d$.
\end{lem}
Now we can prove the proposition.
\begin{proof}[Proof of Proposition \ref{prop phi1 existence}]
Let $w\in\Waff$.
Write $\tilde{\phi}(t_w)=f_w=\sum_{x}A_{x,w}T_x$ as Schwartz functions on $G$ so that $A_{x,w}=f_w(x)$. We must show that there is a unique element $a_{x,w}\in\hat{\mathcal{A}}$ such that $a_{x,w}(q)=A_{x,w}$ as complex numbers.
We will then check that $a_{x,w}\to 0$ rapidly enough as $\ell(x)\to\infty$ for 
$\sum_{x}a_{x,w}T_x$ to define an element of $\HHh$. 

By Corollary \ref{corollary fw(1) denominator divides poincare} and Theorem \ref{theorem general G f_w(1) rational function with correct denominator} , there is a formal power series
in $\hat{\mathcal{A}}^-$ with constant term equal to $1$ that specializes to the denominator of $A_{1,w}$
when $\bq=q$. Moreover, there is a unique formal Laurent series $a_{1,w}\in\hat{\mathcal{A}}$
such that $a_{1,w}(q)=A_{1,w}$ for all prime powers. 
Indeed, $a_{1,w}$ is convergent
for $\bq=q$, and the difference of any two such series defines a meromorphic function of $\bq^{-1/2}$ outside the unit disk with zeros at $q=p^r$ for every $r\in\N$. As these prime powers accumulate at 
$\infty$, such a meromorphic function must be identically zero.

If $f\in\Cc^{I\times I}$ is a Harish-Chandra Schwartz function, then
\begin{multline*}
q^{-\ell(x)}(f\star T_{x^{-1}})(1)=q^{-\ell(x)}\dIntOver{G}{f(g)T_{x^{-1}}(g^{-1})}{\mu_I(g)}
\\
=q^{-\ell(x)}\dIntOver{IxI}{f(g)}{\mu_I(g)}
=q^{-\ell(x)}\mu_I(IxI)f(x)
=f(x).
\end{multline*}
By definition, $f_w(x)=A_{x,w}$. 

On the other hand, according to Lemma \ref{lemma tilde phi is inverse to phi on H}, we have, for $\omega(x^{-1})_f$ as defined above Lemma \ref{lem Goldman involution exchanges KL bases},
\begin{align}
q^{-\ell(x)}(f_w\star T_{x^{-1}})(1)
&=
q^{-\ell(x)}\left(\tilde{\phi}(t_w)\star\tilde{\phi}\left(\phi_q({}^\dagger T_{x^{-1}})\right)\right)(1)
\nonumber
\\
&=
q^{-\ell(x)}\tilde{\phi}\left(t_w\phi_q({}^\dagger T_{x^{-1}})\right)(1)
\nonumber
\\
&=
q^{-\ell(x)}\tilde{\phi}\left(t_w\phi_q\left(\sum_{y\leq x^{-1}}q^{\frac{\ell(y)}{2}}(-1)^{\ell(x^{-1})-\ell(y)}Q_{y,x^{-1}}(q){}^\dagger C'_y\right)\right)(1)
\label{prop lands in completion change to C'}
\\
&=
q^{-\ell(x)}\tilde{\phi}\left(t_w\phi_q\left(\sum_{y\leq x^{-1}}q^{\frac{\ell(y)}{2}}(-1)^{\ell(x^{-1})}(-1)^{\ell(\omega(x^{-1})_f)}Q_{y,x^{-1}}(q)(C_y)\right)\right)(1)
\label{prop applied Goldman}
\\
&=
(-1)^{\ell(x^{-1})}(-1)^{\ell(\omega(x^{-1})_f)}q^{-\ell(x)}\tilde{\phi}\left(t_w\sum_{y\leq x^{-1}}q^{\frac{\ell(y)}{2}}Q_{y,x^{-1}}(q)\sum_{\substack{r \\ d\in\mathcal{D} \\ a(d)=a(r)}}h_{y,d,r}t_r\right)(1)
\label{prop lands in completion calculations after applying phi}
\\
&=
(-1)^{\ell(x^{-1})+\ell(\omega(x^{-1})_f)}q^{-\ell(x)}\tilde{\phi}\left(\sum_{y\leq x^{-1}}q^{\frac{\ell(y)}{2}}Q_{y,x^{-1}}(q)\sum_{\substack{r\sim_L d \\ d\in\mathcal{D} \\ a(d)=a(r)=a(w)}}h_{y,d,r}t_wt_r\right)(1)
\nonumber
\\
&=
(-1)^{\ell(x^{-1})+\ell(\omega(x^{-1})_f)}q^{-\ell(x)}\sum_{y\leq x^{-1}}q^{\frac{\ell(y)}{2}}Q_{y,x^{-1}}(q)\sum_{\substack{r\sim_L d_w \\ a(r)=a(w)}}h_{y,d_w,r}(f_w\star f_r)(1),
\label{prop lands in completion final}
\end{align}
where $d_w$ is the unique distinguished involution in the right cell containing $w$.

In line \eqref{prop lands in completion change to C'}, we rewrote $T_{x^{-1}}$
in terms of the $C'$-basis of $H$, using the inverse Kazhdan-Lusztig polynomials $Q_{y,x^{-1}}$.
In line \eqref{prop applied Goldman}, we applied the involution ${}^\dagger(-)$ (see Lemma \ref{lem Goldman involution exchanges KL bases}).
In line 
\eqref{prop lands in completion calculations after applying phi} we applied Lusztig's map $\phi_q$,
and then in line \eqref{prop lands in completion final}, we applied the map $\tilde{\phi}$.
Also in line \eqref{prop lands in completion final}, we used that left (respectively right) cells give left (respectively right) ideals of $J$,
and so $t_wt_r$ is an integral linear combination of $t_{z^{-1}}$, with 
$d\sim_L z^{-1}\sim_R d_w$. By lemma \ref{lem only diagonal f_w have nonzero trace},
$f_{z^{-1}}(1)\neq 0$ only if $d=d_w$, that is, $z^{-1}\sim_L d_w$.

We use \eqref{prop lands in completion final} to define 
$a_{x,w}\in\hat{\mathcal{A}}$. By the same arguments as above, $a_{x,w}$ is unique
and defines a meromorphic function of $\bq^{-1/2}$. It remains to show that as $\ell(x)\to\infty$,
$a_{x,w}\to 0$ in the $(\bq^{-1/2})$-adic topology. This follows in fact from \eqref{prop lands in completion final}. Indeed, the product $f_w\star f_r$ is an $\N$-linear combination of functions $f_z$, and the values $f_z(1)$ are rational functions of $q$, the numerators of which have 
uniformly bounded degree in $q$ by Lemma \ref{lem f_w(1) polynomial growth and a1w bounded in q}. The polynomials
$h_{y,d_w,r}$ have bounded degree in $q$ (for example in terms of the $a$-function). Finally,
the degree in $q$ of 
\[
q^{-\ell(x)}q^{\frac{\ell(y)}{2}}Q_{y,x^{-1}}(q)
\]
is at most 
\begin{equation}
\label{eq decay rate}
q^{-\ell(x)}q^{\frac{\ell(y)}{2}}q^{\frac{\ell(x^{-1})-\ell(y)-1}{2}}=
q^{-\ell(x)}q^{\frac{\ell(y)}{2}}q^{\frac{\ell(x)-\ell(y)-1}{2}}=q^{\frac{-\ell(x)-1}{2}}\to 0
\end{equation}
as $\ell(x)\to\infty$.
This completes the definition of $\hat{\phi}$ as a 
map of $\mathcal{A}$-modules. 

It is easy to see that $\hat{\phi}$ is a morphism of rings, essentially because
$\tilde{\phi}$ is. Indeed, we have
\begin{equation}
\label{eq phi1 hom first product}
\hat{\phi}(t_wt_{w'})=\sum_{z}\gamma_{w,w',z^{-1}}\sum_{x}a_{x,z}T_x
\end{equation}
while on the other hand
\begin{equation}
\label{eq phi1 hom second product}
\hat{\phi}(t_w)\hat{\phi}(t_{w'})=\sum_{x}a_{x,w}T_x\cdot\sum_{y}a_{y,w'}T_y
\end{equation}
and when $\bq=q$, we have that \eqref{eq phi1 hom second product} becomes by definition
\[
\tilde{\phi}(t_w)\star\tilde{\phi}(t_{w'})=\tilde{\phi}\left(\sum_{z}\gamma_{w,w',z^{-1}}t_z\right)=
\sum_{z}\sum_{x}\gamma_{w,w',z^{-1}}A_{x,z}T_x.
\]
Hence for infinitely many prime powers we have that the specializations of \eqref{eq phi1 hom first product}
agrees with those of \eqref{eq phi1 hom second product}, and hence \eqref{eq phi1 hom first product}
is equal to \eqref{eq phi1 hom second product} in $\HHh$. A similar argument shows that $\hat{\phi}$ preserves 
units.
\end{proof}
\begin{rem}
The proof, specifically \eqref{eq decay rate}, gives a necessary condition for an element of $\HHh$ to belong to 
the image of $\hat{\phi}$: the coefficients must decay asymptotically at least as fast as 
$\bq^{-\frac{\ell(x)}{2}}$.
\end{rem}
\begin{prop}
\label{prop phi_1=phi inverse}
There is a commutative diagram
\begin{center}
\begin{tikzcd}
\HH\arrow[r,"\phi"]&J\otimes_\C\mathcal{A}\arrow[r, "\hat{\phi}"]\arrow[d, "\phi^{-1}"]&\HHh\\
&\HHh_{C}\arrow[r, "{}^\dagger(-)"]&\HHh_{C'}\arrow[u, hook],
\end{tikzcd}
\end{center}
and we have $\hat{\phi}={}^\dagger(-)\circ\phi^{-1}$ as morphisms of $\mathcal{A}$-algebras $J\otimes_\C\mathcal{A}\to\HHh$. In particular, $a_{x,w}$ has integer coefficients for all $x,w\in\Waff$.
\end{prop}
\begin{proof}
The second claim follows from the first if we show that $\hat{\phi}$ extends to a continuous morphism
$\JJ\to\HHh$, by density of $\phi(\HH)$ in $\JJ\simeq\HHh_C$, and the third claim follows from the second and 
the fact that the completions we consider are actually defined over $\Z$ \cite[Thm. 2.8]{affineII}.

Note that as $\tilde{\phi}\circ\phi_q={}^\dagger(-)$ on $H$ for all $q$, we have that $\hat{\phi}={}^\dagger(-)\circ\phi^{-1}$ on
$\phi(\HH)$. This says that the diagram commutes.

We now show that $\hat{\phi}$ extends to a continuous map $\JJ\to\HHh_C$. Let $\sum_{w}b_wt_w$ define an element 
of $\JJ$ and define
\[
\hat{\phi}\left(\sum_{w}b_wt_w\right)=\sum_{y}b'_yT_y,
\]
where
\[
b_y'=\sum_{w}b_wa_{y,w}.
\]
We must first show that this infinite sum of elements of $\hat{\mathcal{A}}$ is well-defined. By 
Lemma \ref{lem f_w(1) polynomial growth and a1w bounded in q} and \eqref{prop lands in completion final},
we have that there is $M\in\N$ such that $a_{y,w}\in(\bq^{1/2})^M\hat{\mathcal{A}}^-$ for all $y,w$.
Therefore $b'_y$ is well-defined, and as $a_{y,w}\to 0$ as $\ell(y)\to\infty$, 
we have $b'_y\to 0$ as $\ell(y)\to\infty$. Therefore $\hat{\phi}$ extends to $\JJ$.

To show continuity, it suffices
to show that if $\{\sum_{w}b_{w,n}t_{w}\}_n$ is a sequence of elements of $\JJ$ tending to 
$0$ as $n\to\infty$, then 
\[
\sum_{w}b_{w,n}\hat{\phi}(t_w)=\sum_{y}b'_{y,n}T_y\to 0
\]
as $n\to\infty$ in $\HHh$, where $b'_{y,n}=\sum_{w}b_{w,n}a_{y,w}$. For all $R>0$, there is
$N>0$ such that $n>N$ implies $b_{x,n}\in(\bq^{-1/2})^R\hat{\mathcal{A}}^-$ for all $x$. We have seen that there is $M$ depending
only on $\Waff$ such that $a_{y,w}\in(\bq^{1/2})^M\hat{\mathcal{A}}^-$ for all $w,y$.
Therefore $b'_{y,n}\to 0$ as $n\to\infty$, because $b_{w,n}\to 0$ as $n\to\infty$.
\end{proof}

Note that Proposition \ref{prop phi_1=phi inverse} means in particular that $\hat{\phi}(J)\subset\HHh_{C'}$.
\begin{cor}
\label{cor t1 formula}
We have
\[
(\phi\circ^\dagger(-))^{-1}(t_1)=f_1(1)\sum_{w\in\Waff}(-1)^{\ell(w)}q^{-\ell(w)}T_w.
\]
\end{cor}
\begin{proof}
If $w=1$, everything in \eqref{prop lands in completion final} reduces to $r=d_w=y=1$, and we need only recall that $Q_{1,x}(\bq)=1$ for any $x\in\Waff$. 
This follows from unicity of the inverse Kazhdan-Lusztig polynomials and the identity
$\sum_{x\leq w}(-1)^{\ell(x)}P_{x,w}(\bq)=0$ for $w\neq 1$ \cite[Exercise 5.17]{BjornerBrenti}.
\end{proof}
\begin{cor}
\label{cor BK map injective}
The map $\tilde{\phi}$ defined in \cite{BK} and recalled in diagram \eqref{eqn BK summary diagram} is 
injective.
\end{cor}
\begin{proof}
Let $q>1$ and let $j\neq 0$ be an element of $J$.
We must show that $\tilde{\phi}(j)\neq 0$. By injectivity of $\hat{\phi}$, we have $\hat{\phi}(j)\neq 0$. 
By definition of the map $\hat{\phi}$, this means that there exists $q_0>1$, $u,s\in G^\vee$ such that
$us=su$ with $s$ compact, and a representation $\rho$ of $\pi_0(Z_{G^\vee}(u,s))$ such that 
$jK(u,s,\rho, q_0)\neq 0$. But then $jK(u,s,\rho, q)\neq 0$, $K(u,s,\rho,q)$ being a different specialization of the restriction of the same $J$-module $E(u,s,\rho)$ as for $K(u,s,\rho, q_0)$, 
and $K(u,s,\rho, q)$ is also tempered. It follows that $\tilde{\phi}(j)\neq 0$.
\end{proof}

\subsection{The case of $\GL_n$}
\label{subsection proof of them conjection is true GLn for GLn}
\begin{theorem}
\label{thm conjecture is true GLn}
Let $\Waff$ be of type $\tilde{A}_n$. Then statements 1 and 2 in the statement of Theorem 
\ref{thm general G denominators} are true, together with a stronger version of statement 3:
Let $u=(r_1,\dots, r_k)$ be a unipotent conjugacy class in $\GL_n(\C)$.
Let $d$ be a distinguished involution in the two-sided cell $\cc(u)\subset\Waff(\GL_n)$ corresponding to $u$.
Let $\pi$ be the unique family of parabolic inductions that $t_d$
does not annihilate. Then $\rank\pi(t_d)=1$.
\end{theorem}
\begin{proof}
By Corollary \ref{corollary fw(1) denominator divides poincare}
and Propositions  \ref{prop phi1 existence} and \ref{prop phi_1=phi inverse}, $a_{1,x}$ is a 
rational function of $\bq$ for all $w$.
Then equation \eqref{prop lands in completion final} implies that $a_{x,w}$, being a sum of rational 
functions with Laurent polynomial coefficients, is a rational function of $\bq$ for all $x$. The same 
equation,
together with the fact that $J_\cc$ is a two-sided ideal for each cell $\cc$ shows that the denominator
of $a_{x,w}$ depends only on the two-sided cell containing $w$. This proves the first claim.
The second claim now follows from the first claim and the first statement of Corollary \ref{corollary fw(1) 
denominator divides poincare} and the fact that $\Waff$ has finitely-many two-sided cells.

Finally, let $u=(r_1,\dots, r_k)$. We have $\pi=\Ind_P^G(\St_{M_P}\otimes\nu)$, for the unique
Levi $M_P$ such that $u$ is distinguished in $M_P^\vee$, and as $\dim\St_{M_P}^I=1$, clearly we have
\[
\dim\pi^I=\frac{n!}{r_1!\cdots r_k!}.
\]
This is also the number of distinguished involutions in $\cc(u)$, by \cite{ShiLNM}.
The claim follows from Corollary \ref{cor BK map injective}.
\end{proof}
\subsection{Proof of Theorem \ref{thm general G denominators}}
\begin{proof}[Proof of Theorem \ref{thm general G denominators}]
Theorem \ref{theorem general G f_w(1) rational function with correct denominator} together with 
Propositions \ref{prop phi1 existence} and \ref{prop phi_1=phi inverse} show that 
$a_{1,w}$ is a rational function of $\bq$ with denominator depending only on the two-sided cell
containing $w$. Equation \eqref{prop lands in completion final} again shows that $a_{x,w}$
is a rational function of $\bq$ with denominator depending only on the two-sided cell
containing $w$; up to twists the set $\mathcal{E}_2(M)^I$ is finite for every Levi subgroup $M$,
so we may multiply through to include the denominators of all required formal degrees, which 
are in fact rational of the correct form by Theorem \ref{FOS theorem}. Therefore there is a 
polynomial $P_\G(\bq)$ that clears denominators of all $a_{x,w}$. This proves the first statement 
of the Theorem.

Now, by Proposition 9 and Remark 2 of \cite{BDD}, $\phi$ induces an surjection
\[
\bar{\phi}\colon \HH/[\HH,\HH]\left[\frac{1}{P_W(\bq) }\right]\onto J/[J,J]\otimes\mathcal{A}.
\]
Therefore for every $j\in J$, there is $N=N(j)\in\N$ and $h\in \HH$ such that
\[
j\equiv h\frac{1}{P_{W}(\bq)^N}
\]
in $J/[J,J]\otimes\mathcal{A}$. Considering traces and invoking Proposition \ref{prop phi_1=phi inverse}, 
we see that the denominator of every $a_{x,w}$ divides a power of $P_{W}(\bq)$. Therefore there is 
$N=N_{\Waff}$ depending only $\Waff$ such that we can take $P^1_\G(\bq)=P_W(\bq)^N$. This proves the 
second statement of the Theorem. The third statement was proven in Corollary \ref{cor denominators of fd(1) poincare}.
\end{proof}
\section{Representations with fixed vectors under parahoric subgroups}
\label{section parahoric-fixed}
In this section we will give an application of the $J$-action on the tempered 
$H$-modules. The first statement is an immediate corollary of the existence
of this action, but the second statement relies on Corollary 
\ref{cor BK map injective}. 
\begin{theorem}
\label{thm existence of parahoric-fixed vectors}
Let $\Ppp$ be a parahoric subgroup of $G$. Let $\pi$ be a simple tempered representation of $G$ 
with $I$-fixed vectors with Kazhdan-Lusztig parameter $(u,s,\rho)$. Let $w_\Ppp$ be the longest
element in the parabolic subgroup of $\Waff$ defined by $\Ppp$ and $\Bb_u^\vee$ be the Springer fibre
for $u$.
\begin{enumerate}
\item 
If 
\[
\ell(w_\Ppp)> a(u)=\dim_\C\Bb_u^\vee,
\]
then 
\[
\pi^{\Ppp}=\{0\}.
\]
\item
Conversely, let $u_\Ppp$ be the unipotent conjugacy class corresponding to the two-sided cell
containing $w_{\Ppp}$. Then there exists $s\in Z_{G^\vee}(u_\Ppp)$, a Levi subgroup $M^\vee$ of $G^\vee$ 
minimal such that $(u_\Ppp,s)\in M^\vee$, and a discrete series representation $\omega\in\mathcal{E}_2(M)$ 
such that 
\[
\pi^{\Ppp}=i_{P_M}^G(\omega\otimes\nu)^{\Ppp}\neq \{0\}
\]
for all $\nu$ non-strictly positive and the parameter of $\pi$ is $(u_\Ppp,s)$.
\end{enumerate}
\end{theorem}
\begin{proof}
Let $\Ppp$ and $w_\Ppp$ be as in the statement. Then
$C'_{w_\Ppp}$ is proportional by a power of $q$ to the indicator function
$1_{\mathcal{P}}$ in $H$. Moreover, $w_\Ppp$ is a distinguished 
involution, with $a(w_\mathcal{P})=\ell(w_\Ppp)$. By Proposition 
\ref{prop phi_1=phi inverse} and the fact that $\phi$ is ``upper-triangular" 
with respect to the $a$-function, we have $(\phi\circ ()^\dagger)(C'_{w_\Ppp})J_\cc=0$ for $\cc$
corresponding to $u$ if $a(w_{\mathcal{P}})>a(u)$.

For the second statement, by Corollary \ref{cor BK map injective}, 
there is a tempered representation $\pi=\Ind_P^G(\omega\otimes\nu)$ with $\pi(t_{w_\Ppp})\neq 0$; the 
unipotent part of its parameter is $u_\Ppp$.
In particular, there is a vector $v\in\pi$ such that $\pi(t_{w_\Ppp})v=v$.
We have 
\[
\left(\phi\circ^\dagger(-)\right)(C'_{w_\mathcal{P}})t_{w_P}
=(-1)^{\ell(w_{\mathcal{P}})}\sum_{\overset{d\in\mathcal{D}}{z\sim_L d}}h_{w_{\mathcal{P}},d,z}t_zt_{w_{\mathcal{P}}}
=
(-1)^{\ell(w_{\mathcal{P}})}\sum_{z}h_{w_{\mathcal{P}},w_{\mathcal{P}},z}t_z
=\vol(\mathcal{P})t_{w_P},
\]
as $t_zt_{w_{\mathcal{P}}}\neq 0$ only if $z\sim_Lw_{\mathcal{P}}$, and 
$C_{w_{\mathcal{P}}}C_{w_{\mathcal{P}}}=(-1)^{\ell(w_{\mathcal{P}})}\vol(\mathcal{P})C_{w_{\mathcal{P}}}$.
Thus $v\in\pi^\mathcal{P}$. As $\trace{\pi}{t_d}$ is constant with 
respect to $\nu$, the last part of the claim follows.
\end{proof}
A version of this statement for enhanced parameters in the case
$\Ppp=G(\Oo)$ appears in \cite[Prop. 10.1]{APM}.
\begin{ex}
The principal series representations have $u=\{1\}$, maximal $a$-value, 
and fixed vectors under every maximal compact subgroup of $G$. On
the other extreme, the Steinberg representation has $u$ regular, 
and does not have fixed vectors under any proper parahoric.
\end{ex}
In addition to the interpretation given in Section \ref{subsubsection intro application parahoric},
in the examples below, the $J$-action also detects which direct summands of a reducible $\mathcal{P}$-spherical tempered representation are themselves $\mathcal{P}$-spherical.
\begin{ex}
Let $G=\SL_2(F)$. As remarked in \cite{D}, the distinguished involutions
in the lowest two-sided cell are the simple reflections $s_0$ and $s_1$,
and each $t_{s_i}$ is invariant under one of the two conjugacy classes 
of maximal parahoric subgroup of $G$.
For unitary principal series representations $\pi$, one has 
$\trace{\pi}{t_{s_0}}=\trace{\pi}{t_{s_1}}=1$. At the quadratic character,
the corresponding principal series representation is reducible, and each
summand contains fixed vectors under precisely one of the maximal parahorics.
Indeed, in \cite{D} this computation is carried out at the level of the Schwartz space of the basic affine space.
\end{ex}
\begin{ex}
Let $G=\SO_5(F)=\PGSp_4(F)$, with affine Dynkin diagram labelled as 
\begin{center}
\dynkin[labels={0,1,2}] B[1]{2}.
\end{center}
There are five conjugacy classes of parahoric subgroups, each obtained by projection from $\GSp_4(F)$: the maximal parahoric $\PGSp_4(\Oo)$, 
the image of the paramodular group $\mathsf{K}$ corresponding to $\{0,2\}$, the image of the Siegel parahoric subgroup $\mathsf{Si}$ corresponding to $\{1\}$, the image of the Klingen parahoric 
$\mathsf{Kl}$ corresponding to $\{0\}$, and the Iwahori subgroup.

The columns of the below table give all the $I$-spherical tempered representations of $G$ by denoting
the representation $\Ind_P^G(\omega)$ by $\omega$. We recall
the few cases of reducibility of these inductions immediately below. The rows list unipotent conjugacy 
classes in $\Gd=\Sp_4(\C)$ such that all tempered standard modules 
$K(u,s,\rho)$ are in row $u$.
That is, the rows record which summand of $J$ acts on each representation.

\begin{center}
\begin{tabular}{|c|c|c|c|c|}
\hline 
Cell\textbackslash Levi & $\GL_1\times\GL_1$ & $\GL_1\times\SO_3$  &$\GL_2$ & $\SO_5$ \\ 
\hline 
$(1,\dots, 1)$ & $\nu$ &  &  &  \\ 
\hline 
$(2,1,1)$ &  & $\nu\otimes\pm\St_{\SO_3}$ &  &  \\ 
\hline 
$(2,2)$ &  &  & $\xi\St_{\GL_2}$ & $\pm\tau_2$ \\ 
\hline 
$(4)$ &  &  &  & $\pm\St$ \\ 
\hline 
\end{tabular} 
\end{center}
The discrete series are as in \cite{ReederFormalDeg}. 
The only reducibility, by \cite[Prop. 3.3]{Matic}, is the reducibility
\[
\Ind_P^G(\xi\St_{\GL_2})=\tau_\mathrm{triv}\oplus\tau_\mathrm{sgn}
\]
for $\xi^2=1$.

We can now compute the traces of some elements $t_d$ using the description
of the simple $J$-modules given in \cite{XiBook}. For the cells
$(1,\dots, 1)$, $(2,1,1)$, and $(4)$, we have $\trace{\pi}{t_d}=1$  for all $\pi$.
We have
\[
\trace{\tau_{\triv}}{t_{s_0}}=\trace{\tau_2}{t_{s_0}}=1=\trace{\tau_{\sgn}}{t_{s_1},}=\trace{\tau_{\triv}}{t_{s_1}}.
\]

In the below table, $\trace{\pi}{t_d}$ is recorded in bold face, whereas the dimension of 
$\pi^{\Ppp}$, taken from \cite[Table A.15]{BrooksSchmidt}, is recorded in normal face. Representations
attached to the same cell but belonging to different packets are separated with a dotted line.

\begin{center}
\begin{tabular}{|c|c|cccccc|c|}
\hline 
$t_{w_{\mathcal{P}}}$\textbackslash\, $\pi=\Ind_P^G(-)$ & $a$ & \multicolumn{1}{|c|}{$\nu$} & \multicolumn{1}{|c|}{$\nu\otimes\pm\St_{\SO_3}$} & $\tau_\mathrm{triv}$ & $\tau_\mathrm{sgn}$ & \multicolumn{1}{:c|}{$\pm\tau_2$}  & $\pm\St$ & $\Ppp$\\ 
\hline 
$t_{w_0}$ & 4 & \multicolumn{1}{c|}{\textbf{1}} &  &  &  &  & & $\PGSp_4(\Oo)$ \\ 
\cline{1-4} \cline{9-9}
$t_{s_0s_2}$ & 2 & 2 &  \multicolumn{1}{|c|}{\textbf{1}} &  &  &  & & $\mathsf{K}$\\ 
\cline{1-2} \cline{4-7} \cline{9-9}
$t_{s_0}$&1 & 4 & 2 & \multicolumn{1}{|c}{\textbf{1}} &  & \multicolumn{1}{:c|}{\textbf{1}} &  & $\mathsf{Kl}$ \\ \cline{9-9}
$t_{s_1}$ & 1 & 4  & 1  & \multicolumn{1}{|c}{\textbf{1}} & \textbf{1} & \multicolumn{1}{:c|}{} & & $\mathsf{Si}$ \\ \cline{1-2}\cline{5-9}
$t_1$ & 0 & 8 & 4 & 3 & 1 & 4  & \multicolumn{1}{|c|}{\textbf{1}} & $I$\\ 
\hline 
\end{tabular} 
\end{center}
\end{ex}

\paragraph{Acknowledgements.}
The author thanks Alexander Braverman for introducing him to this material and the problem, for suggesting the use of the Plancherel formula, and for helpful conversations.
The author also thanks Roman Bezrukavnikov,  Itai Bar-Natan,
Gurbir Dhillon, Malors Espinosa-Lara, Julia Gordon, George Lusztig, and Matthew Sunohara for helpful conversations. The author thanks Maarten Solleveld for explaining the generalization in \cite{Solleveld} of the results of \cite{fomFormalDegUnip}, suggesting the reference \cite{APM}, and for careful reading of an early version of this manuscript. 
The author owes a great debt to the anonymous referee for careful reading resulting in many corrections and improvements both mathematical and expositional. 

Some of this work was completed at 
the Max Planck Institute in Mathematics for Bonn, the directors and staff of which the author thanks for their
hospitality. This research was supported by NSERC.

\paragraph{Competing interests:} The author declares none.

\bibliography{plancherel_paper_biblio.bib}

\end{document}